%% file: kms68_addendum.tex
\documentclass[11pt,a4paper,reqno]{amsart}
\usepackage{vmargin,color}
\usepackage[latin1]{inputenc}
\usepackage{amsmath,amsfonts,amsthm,epsfig,graphicx,amssymb}
\usepackage{esint}
\usepackage{caption}
\newcommand{\A}{\mathcal A}

\newcommand{\E}{\mathcal{E}}
\newcommand{\F}{\mathcal F}

\renewcommand{\H}{\mathcal{H}}

\newcommand{\Sc}{\mathcal S}

\newcommand{\ee}{\mathrm{e}}

\newcommand{\N}{\mathbb{N}}

\newcommand{\R}{\mathbb{R}}
\renewcommand{\SS}{\mathbb{S}}

\newcommand{\Om}{\Omega}
\renewcommand{\S}{\Sigma}
\renewcommand{\a}{\alpha}

\newcommand{\g}{\gamma}
\newcommand{\de}{\delta}
\newcommand{\e}{\varepsilon}
\renewcommand{\k}{\kappa}
\renewcommand{\l}{\lambda}
\newcommand{\s}{\sigma}

\newcommand{\om}{\omega}
\newcommand{\vphi}{\varphi}
\newcommand{\Lip}{{\rm Lip}}

\newcommand{\Div}{{\rm div}\,}
\newcommand{\Id}{{\rm Id}\,}
\newcommand{\dist}{{\rm dist}}

\newcommand{\spt}{{\rm spt}}

\newcommand{\weakstar}{\stackrel{*}{\rightharpoonup}}

\newcommand{\pa}{\partial}

\newcommand{\cc}{\subset\subset}

\newcommand{\cl}{\mathrm{cl}\,}

\newcommand{\C}{\mathcal{C}}

\newcommand{\KK}{\mathcal{K}}

\theoremstyle{plain}
\newtheorem{theorem}{Theorem}[section]
\newtheorem{lemma}[theorem]{Lemma}

\newtheorem*{theorem*}{Theorem}
\newtheorem*{corollary*}{Corollary}

\theoremstyle{definition}

\newtheorem{example}[theorem]{Example}

\newtheorem{remark}[theorem]{Remark}

\newtheorem*{notation*}{Notation}

\numberwithin{equation}{section}
\numberwithin{figure}{section}
\newcommand{\id}{{\rm id}\,}

\title{Plateau's problem as a singular limit \\ of capillarity problems}

\author{Darren King, Francesco Maggi, and Salvatore Stuvard}
\address{Department of Mathematics, The University of Texas at Austin, 2515 Speedway, Stop C1200, Austin TX 78712-1202, USA}
\medskip
\email{king@math.utexas.edu}
\medskip

\email{maggi@math.utexas.edu}
\medskip

\email{stuvard@math.utexas.edu}

\begin{document}

\begin{abstract}
{\rm Soap films at equilibrium are modeled, rather than as surfaces, as regions of small total volume through the introduction of a capillarity problem with a homotopic spanning condition. This point of view introduces a length scale in the classical Plateau's problem, which is in turn recovered in the vanishing volume limit. This approximation of area minimizing hypersurfaces leads to an energy based selection principle for Plateau's problem, points at physical features of soap films that are unaccessible by simply looking at minimal surfaces, and opens several challenging questions.}
\end{abstract}

\maketitle

\tableofcontents

\section{Introduction}

\subsection{Overview}\label{section overview} The theory of minimal surfaces with prescribed boundary data provides the basic model for soap films hanging from a wire frame: given an $(n-1)$-dimensional surface $\Gamma\subset\R^{n+1}$ without boundary, one seeks $n$-dimensional surfaces $M$ such that
\begin{equation}
  \label{minimal surfaces}
  H_M=0\,,\qquad\pa M=\Gamma\,,
\end{equation}
where $H_M$ is the mean curvature of $M$ (and $n=2$ in the physical case). A limitation of \eqref{minimal surfaces} as a physical model is that, in general, \eqref{minimal surfaces} may be non-uniquely solvable, including unstable (and thus, not related to observable soap films) solutions. Area minimization can be used to construct stable (and thus, physical) solutions, providing a strong motivation for the study of {\it Plateau's problem}; see \cite{coldingminiBOOK}. Here we are concerned with a more elementary physical limitation of \eqref{minimal surfaces}, namely, the absence of a length scale: if $M$ solves \eqref{minimal surfaces} for $\Gamma$, then $t\,M$ solves \eqref{minimal surfaces} for $t\,\Gamma$, no matter how large $t>0$ is.

\medskip

Following \cite{maggiscardicchiostuvard}, we introduce a length scale in the modeling of soap films by thinking of them as regions $E\subset\R^{n+1}$ with small volume $|E|=\e$. At equilibrium, the isotropic pressure at a point $y$ interior to the liquid but immediately close to its boundary $\pa E$ is
\begin{equation}
  \label{balance of p 1}
  p(y)=p_0+\s\,\vec{H}_{\pa E}(y)\cdot\nu_E(y)\,,
\end{equation}
where $p_0$ is the atmospheric pressure, $\s$ is the surface tension, $\nu_E$ the outer unit normal to $E$, and $\vec{H}_{\pa E}$ the mean curvature vector of $\pa E$; at the same time, for any two points $y,z$ inside the film we have
\begin{equation}
  \label{balance of p 2}
  p(y)-p(z)=\rho\,g\,(z-y)\cdot e_{n+1}\,,
\end{equation}
where $\rho$ is the density of the fluid, $g$ the gravity of Earth and $e_{n+1}$ is the vertical direction. In the absence of gravity, \eqref{balance of p 1} and \eqref{balance of p 2} imply that $H_E=\vec{H}_{\pa E}\cdot\nu_E$ is {\it constant} along $\pa E$. A heuristic analysis shows that if $\pa E$ is representable, locally, by the two graphs $\{x\pm (h(x)/2)\,\nu_M(x):x\in M\}$ defined by a positive function $h$ over an ideal mid-surface $M$, then $H_M$ should be {\it small, but non-zero} (even in the absence of gravity); see \cite[Section 2]{maggiscardicchiostuvard}. As it is well-known, one cannot prescribe non-vanishing mean curvature with arbitrarily large boundary data, see, e.g. \cite{giustiINV,duzaarfuchs}. Hence this point of view can potentially capture physical features of soap films that are not accessible by modeling them as minimal surfaces.

\medskip

The goal of this paper is starting the analysis of the variational problem playing for \eqref{balance of p 1} and \eqref{balance of p 2} the role that Plateau's problem plays for \eqref{minimal surfaces}. The new aspect is not in the energy minimized, but in the boundary conditions under which the minimization occurs. Indeed, the equivalence between the constancy of $H_E$ and the balance equations \eqref{balance of p 1} and \eqref{balance of p 2}, leads us to work in the classical framework of Gauss' capillarity model for liquid droplets in a container. Given an open set $\Om\subset\R^{n+1}$ (the container), the surface tension energy\footnote{For simplicity, we are setting to zero the adhesion coefficient with the container; see, e.g. \cite{Finn}.} of a droplet occupying the open region $E\subset\Om$ is given by
\[
\s\,\H^n(\Om\cap\pa E)\,,
\]
where $\H^n$ denotes $n$-dimensional Hausdorff measure (surface area if $n=2$, length if $n=1$). In the case of soap films hanging from a wire frame $\Gamma$, we choose as container $\Om$ the set
\[
\Om=\R^{n+1}\setminus I_\de(\Gamma)\,,
\]
corresponding to the complement of the ``solid wire'' $I_\de(\Gamma)$, where $I_\de$ denotes the closed $\de$-neighborhood of a set. The minimization of $\H^n(\Om\cap\pa E)$ among open sets $E\subset\Om$ with $|E|=\e$ leads indeed to finding minimizers whose boundaries have constant mean curvature. However, these boundaries will not resemble soap films at all, but will rather consist of  small ``droplets'' sitting at points of maximal curvature for $I_\de(\Gamma)$; see
\begin{figure}
  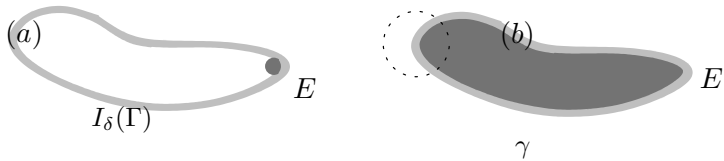\caption{{\small Minimizers of the capillarity problem in the unusual container $\Om$ consisting of the complement of a $\de$-neighborhood $I_\de(\Gamma)$ of a curve $\Gamma$ (depicted in light gray). The shape of $E$ is drastically different depending on whether or not a homotopic spanning condition is prescribed: (a) without a $\C$-spanning condition, we observe tiny droplets sitting near points of maximal mean curvature of $\pa\Om$; (b) with a $\C$-spanning condition, small rounds droplets will not be admissible, and a different region of the energy landscape is explored; minimizers are now expected to stretch out and look like soap films.}}\label{fig span}
\end{figure}
Figure \ref{fig span}, and \cite{baylerosales,fall,maggimihaila} for more information.

\medskip

To observe soap films, rather than droplets, we must require that $\pa E$ stretches out to span $I_\de(\Gamma)$. To this end, we exploit a beautiful idea introduced by Harrison and Pugh in \cite{harrisonpughACV}, as slightly generalized in \cite{DLGM}. The idea is fixing a {\bf spanning class}, i.e. a homotopically closed\footnote{By this we mean that if $\gamma_0,\gamma_1$ are smooth embeddings of $\SS^1$ into $\Omega$, $\g_0\in\C$, and there exists a continuous map $f:[0,1]\times\SS^1\to\Omega$ with $f(t,\cdot)=\gamma_t$ for $t=0,1$, then $\gamma_1\in\C$.} family $\C$ of smooth embeddings of $\SS^1$ into $\Om=\R^{n+1}\setminus I_\de(\Gamma)$, and to say\footnote{Notice that, in stating condition \eqref{spanning condition}, the symbol $\gamma$ denotes the subset $\gamma(\SS^1)\subset\Omega$. We are following here the same convention set in \cite{DLGM}.} that a relatively closed set $S\subset\Om$ is {\bf $\C$-spanning $I_\de(\Gamma)$} if
\begin{equation}
  \label{spanning condition}
  S\cap\g\ne\emptyset\qquad\forall \g\in\C\,.
\end{equation}
Given a choice of $\C$, we have a corresponding version of Plateau's problem
\begin{equation}
  \label{ell intro}
  \ell=\inf\Big\{\H^n(S):\mbox{$S$ is relatively closed in $\Om$, and $S$ is $\C$-spanning $I_\de(\Gamma)$}\Big\}\,,
\end{equation}
as illustrated in Figure
\begin{figure}
  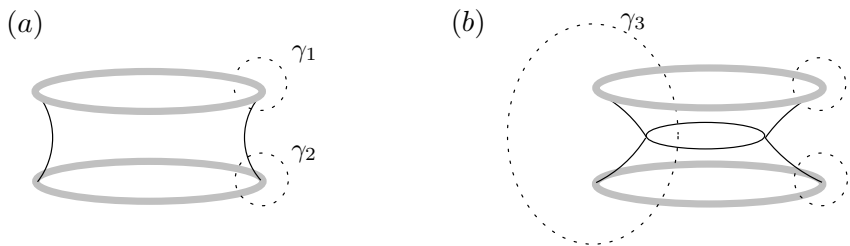\caption{{\small The variational problem \eqref{ell intro} with $\Gamma$ given by two parallel circles centered on the same axis at a mutual distance smaller than their common radius. Different choices of $\C$ lead to different minimizers $S$ in $\ell$: (a) if $\C$ is generated by the loops $\g_1$ and $\g_2$, then $S$ is the area minimizing catenoid; (b) if we add to $\C$ the homotopy class of $\g_3$, then $S$ is the {\it singular} area minimizing catenoid, consisting of two catenoidal necks, meeting at equal angles along a circle of $Y$-points bounding a ``floating'' disk. Such singular catenoid cannot be approximated in energy by smooth surfaces: hence the choice of casting $\ell$ in a class of non-smooth surfaces.}}\label{fig hpc}
\end{figure}
\ref{fig hpc}. The variational problem $\psi(\e)$ studied here is thus a reformulation of $\ell$ as a capillarity problem with a homotopic spanning condition, namely:
\begin{equation}\nonumber
  \psi(\e)=\inf\Big\{\H^n(\Om\cap\pa E):\mbox{$E\subset\Om$, $|E|=\e$, $\Om\cap\pa E$ is $\C$-spanning $I_\de(\Gamma)$}\Big\}\,,\qquad\e>0\,.
\end{equation}
We now give {\it informal} statements of our main results (e.g., we make no mention to singular sets or comment on reduced vs topological boundaries); see section \ref{section main statements} for the formal ones.

\medskip

\noindent {\bf Existence of generalized minimizers and Euler-Lagrange equations (Theorem \ref{thm lsc} and Theorem \ref{thm basic regularity})}:
      {\it There always exists a generalized minimizer $(K,E)$ for $\psi(\e)$: that is, there exists a set $K\subset\Om$, relatively closed in $\Om$ and $\C$-spanning $I_\de(\Gamma)$, and there exists an open set $E\subset\Om$ with $\Om\cap\pa E\subset K$ and $|E|=\e$, such that
      \[
      \psi(\e)=\F(K,E)=2\,\H^n(K\setminus\pa E)+\H^n(\Om\cap\pa E)\,.
      \]
      Moreover, $(K,E)$ minimizes $\F$ with respect to all its diffeomorphic images: in particular, $\Om\cap\pa E$ has constant mean curvature $\l\in\R$ and $K\setminus\pa E$ has zero mean curvature.}

\medskip

\noindent {\bf Convergence to the Plateau's problem (Theorem \ref{thm convergence as eps goes to zero})}:
      {\it We always have $\psi(\e)\to 2\,\ell$ when $\e\to 0^+$, and if $(K_j,E_j)$ are generalized minimizers for $\psi(\e_j)$ with $\e_j\to 0^+$, then, up to extracting subsequences, we can find a minimizer $S$ for $\ell$ with
      \[
      2\,\int_{K_j\setminus\pa E_j}\vphi+\int_{\pa E_j}\vphi\to2\,\int_S\vphi\qquad\forall \vphi\in C^0_c(\Om)\,,
      \]
      as $j\to\infty$; in other words, generalized minimizers in $\psi(\e_j)$ with $\e_j\to 0^+$ converge as Radon measures to minimizers in the Harrison-Pugh formulation of Plateau's problem.}


\begin{example}[Volume and thickness in the non-collapsed case]\label{example two points}
  {\rm Let $\Gamma$ consists of two points at distance $r$ in the plane, or of an $(n-1)$-sphere of radius $r$ in $\R^{n+1}$. For $\e$ small enough, $\psi(\e)$ should admit a unique generalized minimizer $(K,E)$, consisting of two almost flat spherical caps meeting orthogonally along the torus $I_\de(\Gamma)$ (so that $K=\pa E$ and collapsing does not occur); see Figure \ref{fig example23}-(a). In general, we expect that {\it when all the minimizers $S$ in $\ell$ are smooth, then generalized minimizers in $\psi(\e)$ are not collapsed, and, for small $\e$, $K=\pa E$ is a two-sided approximation of $S$, with $H_{E}=\psi'(\e)\to 0$ and
  \begin{equation}
    \label{psi eps quadro expected}
    \psi(\e)=2\,\ell+C\,\e^2+{\rm o}(\e^2)\,,\qquad\mbox{as $\e\to 0^+$}\,,
  \end{equation}
  for a positive $C$}. This insight is consistent with the idea (see \cite{maggiscardicchiostuvard}) that {\it almost minimal surfaces} arise in studying soap films with a thickness. In particular, {\it volume and thickness will be directly related} in terms of the geometry of $\Gamma$. Sending $\e\to 0^+$ with $\Gamma$ fixed or, equivalently, considering $t\,\Gamma$ for large $t$ at $\e$ fixed, will make the thickness decrease until it reaches a threshold below which we do not expect soap films to be stable. A critical thickness can definitely be identified with the characteristic length scale of the molecules of surfactant, below which the model stops making sense. But depending on temperatures, actual soap films with even larger thicknesses should burst out due to the increased probability of fluctuations towards unstable configurations.}
\end{example}

\begin{example}[Volume and thickness in the collapsed case]\label{example tripe sing}
  {\rm At small volumes, and in presence of singularities in the minimizers of $\ell$, collapsing is energetically convenient, and allows $\psi(\e)$ to approximate $2\,\ell$ from below. If $\Gamma\subset\R^2$ consists of the three vertices of an equilateral triangle, for small $\de$ the unique minimizer of $\ell$ consists of a $Y$-configuration. For small $\e$, we expect generalized minimizers $(K,E)$ of $\psi(\e)$ to be collapsed, see Figure
  \begin{figure}
  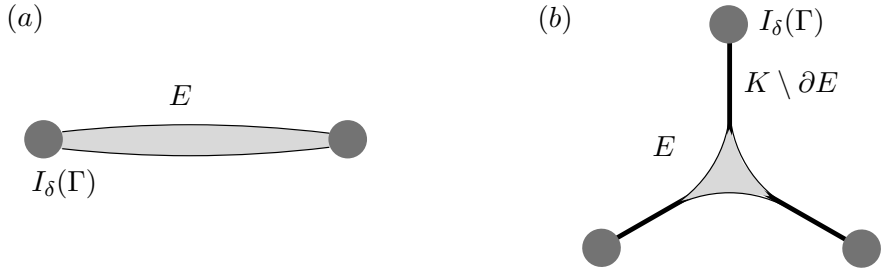\caption{{\small (a) If $\Gamma$ consists of two points, then the minimizer is not collapsed, and is bounded by two very flat circular arcs; (b) when $\Gamma$ consists of the vertices of an equilateral triangle, the generalized minimizer is indeed collapsed. The three segments defining $K\setminus\pa E$ are depicted in bold, and $E$ is a negatively curved curvilinear triangle nested around the singular point of the unique minimizer of $\ell$.}}\label{fig example23}
  \end{figure}
  \ref{fig example23}-(b): there, $E$ is a curvilinear triangle made up of three circular arcs whose length is ${\rm O}(\sqrt\e)$, and whose (negative) curvature is ${\rm O}(1/\sqrt{\e})$. The thickness of an actual soap film in this configuration should thus be considerably larger near the singularity than along the collapsed region, and the volume and the thickness of the film are somehow {\it independent} geometric quantities. This suggests, in presence of singularities, the need for introducing a second length scale in the model. A possibility is replacing the sharp interface energy $\H^n(\Om\cap\pa E)$ with a diffused interface energy, like the Allen-Cahn energy
  \[
  \E_\eta(u)=\eta\,\int_\Om|\nabla u|^2+\frac1\eta\int_\Om W(u)\,,\qquad \eta>0\,,
  \]
  for a double-well potential with $\{W=0\}=\{-1,1\}$. We expect $\{u>0\}$ to (approximately) coincide with the union of a curvilinear triangle of area $\e$ with three stripes having the collapsed segments as their mid-sections, and of width $\eta\,|\log\eta|$; cf. with \cite{delpino1}.}
\end{example}

\begin{example}[Capillarity as a selection principle for Plateau's problem]\label{example selection}
  {\rm The following statement holds (as a heuristic principle): {\it Generalized minimizers of $\psi(\e)$ converge to those minimizers of Plateau's problem \eqref{ell intro} with larger singular set, and when no singular minimizers are present, they select those whose second fundamental form has maximal $L^2$-norm}. Since the second part of this selection principle is justified by standard second variation arguments, we illustrate the first part only. In
  \begin{figure}
    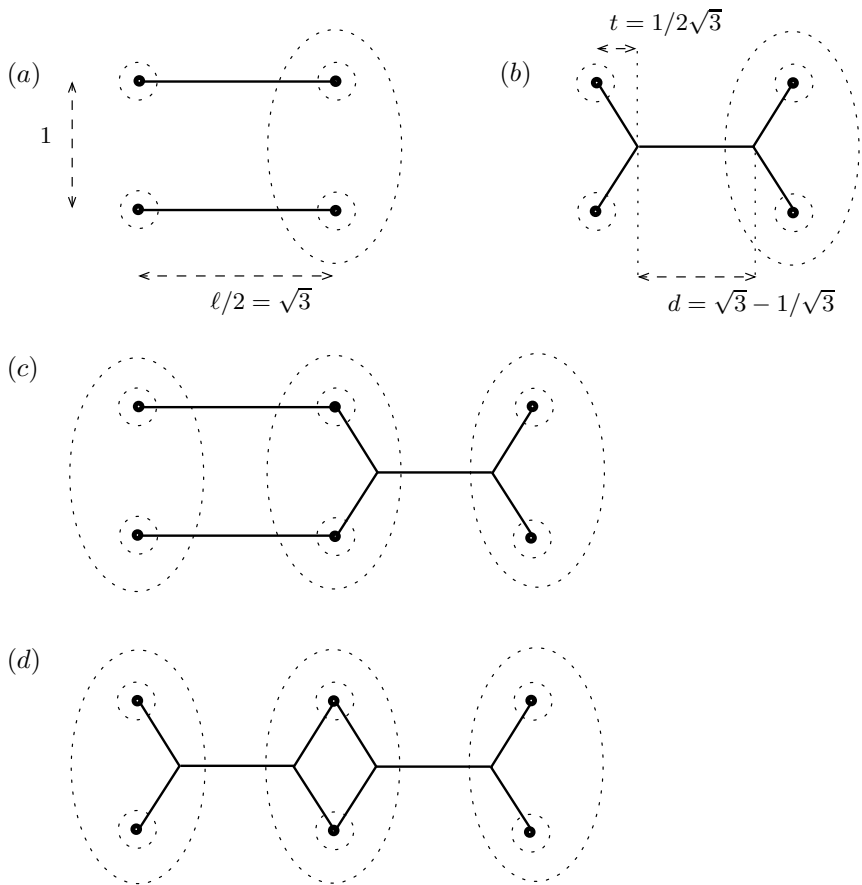\caption{{\small (a) and (b): a four points configuration $\Gamma$ with a choice of $\C$ such that $\ell$ admits two minimizers, one with and one without singularities; (c) and (d): a six points configuration $\Gamma$ with a choice of $\C$ such that $\ell$ admits many minimizers, possibly with a variable number of singularities; here we have depicted two of them, including the one with four singular points that is selected by the $\psi(\e)$ problems.}}\label{fig example}
  \end{figure}
  Figure \ref{fig example}, $\Gamma$ is either given by four or by six points, that are suitably spaced so that $\ell$ has different minimizers. As $\e\to 0^+$, $\psi(\e)$ selects those $\ell$-minimizers with singularities over the ones without singularities; and when more minimizers with singularities are present, it selects the ones with the largest number of singularities. Indeed, the approximation of a smooth minimizer in $\ell$ will require an energy cost larger than $2\,\ell$. At the same time, each time a singularity is present, minimizers of $\psi(\e)$ can save length in the approximation, thus paying less than $2\,\ell$ in energy, and the more the singularities, the bigger the gain. To check this claim, pick $N$ singularities, and denote by $\e_i$ the volume placed near the $i$-th singularity and by $r_i$ the radius of the three circular arcs enclosing $\e_i$. Each wetted singularity has area $c_1\,r_i^2$, while the total relaxed energy of the approximating configuration is $\F=2\,\ell-c_2\,\sum_{i=1}^N\,r_i$. Minimizing under the constraint $\e=c_1\,\sum_{i=1}^N\,r_i^2$, we must take $r_i=\sqrt{\e/N c_1}$, thus finding
  \[
  \psi(\e)=2\,\ell-c_2\,\sqrt{\frac{\e\,N_{{\rm max}}}{c_1}}\,,
  \]
  if $N_{{\rm max}}$ is the maximal number of singularities available among minimizers of $\ell$. This example suggests that (in every dimension) {\it in the presence of singular minimizers of $\ell$, one should have}
  \begin{equation}
    \label{psip primo meno infinito}
  \mbox{$\psi'(\e)\to-\infty$ as $\e\to 0^+$}\,.
  \end{equation}
  This is of course markedly different from what we expect to be the situation when $\ell$ has only smooth minimizers, see \eqref{psi eps quadro expected}. We finally notice that a selection principle for the capillarity model (without homotopic spanning conditions) via its Allen-Cahn approximation has been recently obtained by Leoni and Murray, see \cite{leonimurray1,leonimurray2}.  }
\end{example}

\subsection{Statements of the results}\label{section main statements} We now give a more technical introduction to our paper, with precise statements, more bibliographical references, and comments on the proofs.

\medskip

\noindent {\bf Plateau's problem with homotopic spanning}: We fix a compact set $W\subset\R^{n+1}$ (the ``wire frame'') and denote the region accessible by the soap film as
\[
\Om=\R^{n+1}\setminus W\,.
\]
The typical case we have in mind is $W=I_\de(\Gamma)$, as discussed in section \ref{section overview}, but this is not necessary. We fix a {\bf spanning class} $\C$, that is a non-empty family of smooth embeddings of $\SS^1$ into $\Om$ which is closed by homotopy in $\Om$.  We assume that $W$ and $\C$ are such that the {\bf Plateau's problem defined by $\C$}
\begin{equation}
  \label{plateau problem}
  \ell=\inf\big\{\H^n(S):S\in\Sc\big\}
\end{equation}
is such that\footnote{The condition $\ell<\infty$ clearly implies that no $\g\in\C$ is homotopic to a constant map.} $\ell<\infty$. Here, for the sake of brevity, we have introduced
\[
\Sc=\big\{S\subset\Om:\mbox{$S$ is relatively closed in $\Om$ and $S$ is $\C$-spanning $W$}\big\}\,.
\]
As proved in \cite{harrisonpughACV,DLGM}, if $\ell<\infty$, then there exists a compact, $\H^n$-rectifiable set $S$ such that $\H^n(S)=\ell$; see also \cite{harrisonJGA,davidshouldwe,fangAPISA,hpOPEN,hpCOHOMO,dPdRghira,delederosaghira,friedKP,harrisonpughGENMETH,FangKola,derosaSIAM} for related existence results. In addition, $S$ minimizes $\H^n$ with respect to Lipschitz  perturbations of the identity localized in $\Om$, so that: (i) $S$ is a classical minimal surface outside of an $\H^n$-negligible, relatively closed set in $\Om$ by \cite{Almgren76}; (ii) if $n=1$, $S$ consists of finitely many segments, possibly meeting in three equal angles at singular $Y$-points in $\Om$; (iii) if $n=2$, $S$ satisfies {\bf Plateau's laws} by \cite{taylor76}: namely, $S$ is locally diffeomorphic either to a plane, or to a cone $Y=T^1\times\R$, or to a cone $T^2$, where $T^n$ is the cone over the origin defined by the $(n-1)$-dimensional faces of a regular tetrahedron in $\R^{n+1}$. The validity of Plateau's laws in this context makes \eqref{plateau problem} more suitable when one is motivated by physical considerations: indeed, minimizers of the codimension one Plateau's problem in the class of rectifiable currents are necessarily smooth if $n\le 6$. Although smoothness is desirable for geometric applications, it creates an {\it a priori} limitation when studying actual soap films; see also \cite{davidshouldwe,harrisonpughACV,DLGM}.

\medskip

\noindent {\bf The capillarity problem and the relaxed energy}: Next, we give a precise formulation of the capillarity problem $\psi(\e)$ at volume $\e>0$, which is defined as
\begin{equation}
  \label{psi eps}
  \psi(\e)=\inf\Big\{\H^n(\Om\cap\pa E):\mbox{$E\in\E$, $|E|=\e$, $\Om\cap\pa E$ is $\C$-spanning $W$}\Big\}\,.
\end{equation}
Here we have introduced the family of sets
\begin{equation}\label{class E}
  \E=\Big\{E\subset\Om:\,\mbox{$E$ is an open set and $\pa E$ is $\H^n$-rectifiable}\Big\}\,.
\end{equation}
If $E\in\E$, then $\pa E$ is $\H^n$-finite and covered by countably many Lipschitz images of $\R^n$ into $\R^{n+1}$. Thus, $E$ is of finite perimeter in $\Om$ by a classical result of Federer, and its (distributional) perimeter $P(E;U)$ in an open set $U\subset\Om$ is equal to $\H^n(U\cap\pa^* E)$, where $\pa^*E$ is the reduced boundary of $E$ (notice that, in general, $P(E;U)\le\H^n(U\cap\pa E)$). The relaxed energy $\F$ is defined by
\begin{eqnarray*}
\F(K,E;U)&=&\H^n(U\cap\pa^*E)+2\,\H^n\big(U\cap(K\setminus\pa^*E)\big)\,,\qquad \mbox{$U\subset\Om$ open}\,,
\\
\F(K,E)&=&\F(K,E;\Om)\,,
\end{eqnarray*}
on every pair $(K,E)$ in the family $\KK$ given by
\[
\begin{split}
  \KK=
  \Big\{(K,E):\,&\mbox{$E\subset\Om$ is open with $\Om\cap\cl(\pa^*E)=\Om\cap\pa E\subset K$\,,}
  \\
  &\mbox{$K\in\Sc$ and $K$ is $\H^n$-rectifiable in $\Om$}\Big\}\,.
\end{split}
\]
By the requirement $K\in\Sc$, $K$ is $\C$-spanning $W$, while $\Om\cap\pa E$, which is always a subset of $K$, may be not be $\C$-spanning $W$; we expect this when collapsing occurs, see Figure \ref{fig example23}.

\medskip

\noindent {\bf Assumptions on $\Om$}: We make two main geometric assumptions on $W$ and $\C$. Firstly, in constructing a system of volume-fixing variations for a given minimizing sequence of $\psi(\e)$ (see step two of the proof of Theorem \ref{thm lsc}) we shall assume that
\begin{equation}
  \label{hp on W and C 0}
  \mbox{$\exists\,\tau_0>0$ such that, for every $\tau<\tau_0$, $\R^{n+1}\setminus I_{\tau}(W)$ is connected}\,.
\end{equation}
This is compatible with the idea that, in the physical case $n=2$, $W$ represents a ``solid wire''. Secondly, to verify the finiteness of $\psi(\e)$ (see step one in the proof of Theorem \ref{thm lsc}), we require that
\begin{eqnarray}
  \label{hp on W and C}
  \mbox{$\exists\,\eta_0>0$ and a minimizer $S$ in $\ell$ s.t. $\g\setminus I_{\eta_0}(S)\ne\emptyset$ for every $\g\in\C$}\,.
\end{eqnarray}
This is clearly a generic situation, which (thanks to the convex hull property of stationary varifolds) is implied, for example, by the much more stringent condition that $\g\setminus Z\ne\emptyset$ for every $\g\in\C$ where $Z$ is the closed convex hull of $W$. Finally, we shall also assume that ``$\pa\Om=\pa W$ is smooth'': by this we mean that locally near each $x\in\pa\Om$, $\Om$ can be described as the epigraph of a smooth function of $n$-variables.

\bigskip

\noindent {\bf Existence of minimizers and Euler-Lagrange equations}: Our first main result is the existence of generalized minimizers of $\psi(\e)$.

\begin{theorem}[Existence of generalized minimizers]\label{thm lsc}
  Let $\ell<\infty$, $\pa W$ be smooth and let \eqref{hp on W and C 0} and \eqref{hp on W and C} hold. If $\{E_j\}_j$ is a minimizing sequence for $\psi(\e)$, 
  then there exists a pair $(K,E)\in\KK$ with $|E|=\e$ such that, up to possibly extracting subsequences, and up to possible modifications of each $E_j$ outside a large ball containing $W$ (with both operations resulting in defining a new minimizing sequence for $\psi(\e)$, still denoted by $\{E_j\}_j$), we have that
  \begin{equation}\label{mininizing seq conv to gen minimiz}
    \begin{split}
    &\mbox{$E_j\to E$ in $L^1(\Om)$}\,,
    \\
    &\H^n\llcorner(\Om\cap\pa E_j)\weakstar \theta\,\H^n\llcorner K\qquad\mbox{as Radon measures in $\Om$}\,,
    \end{split}
  \end{equation}
  as $j\to\infty$, where $\theta:K\to\R$ is an upper semicontinuous function with
  \begin{equation}
    \label{theta density}
    \mbox{$\theta= 2$ $\H^n$-a.e. on $K\setminus\pa^*E$},\qquad\mbox{$\theta=1$ on $\Om\cap\pa^*E$}\,.
  \end{equation}
  Moreover, $\psi(\e)=\F(K,E)$ and, for a suitable constant $C$, $\psi(\e)\le 2\,\ell+C\,\e^{n/(n+1)}$.
\end{theorem}

\begin{remark}
  {\rm Whenever $(K,E)\in\KK$ is such that $|E|=\e$, $\F(K,E)=\psi(\e)$ and there exists a minimizing sequence $\{E_j\}_j$ for $\psi(\e)$ which converges to $(K,E)$ as in \eqref{mininizing seq conv to gen minimiz}, we say that $(K,E)$ is a {\bf generalized minimizer of $\psi(\e)$}. We say that $(K,E)$ is {\bf collapsed} if $K\setminus\pa E\ne\emptyset$. If $(K,E)$ is not collapsed, then $E$ is a (standard) minimizer of $\psi(\e)$.}
\end{remark}

Next, we derive the Euler-Lagrange equations for a generalized minimizer and apply Allard's theorem.

\begin{theorem}[Euler-Lagrange equation for generalized minimizers]\label{thm basic regularity}
 Let $\ell<\infty$, $\pa W$ be smooth and let \eqref{hp on W and C 0} and \eqref{hp on W and C} hold. If $(K,E)$ is a generalized minimizer of $\psi(\e)$ and $f:\Om\to\Om$ is a diffeomorphism such that $|f(E)|=|E|$, then
 \begin{equation}
   \label{minimality KE against diffeos}
   \F(K,E)\le\F(f(K),f(E))\,.
 \end{equation}
 In particular:
 \begin{enumerate}
   \item[(i)] there exists $\l\in\R$ such that
  \begin{equation}
    \label{stationary main}
    \l\,\int_{\pa^*E}X\cdot\nu_E\,d\H^n=\int_{\pa^*E}\Div^K\,X\,d\H^n+2\,\int_{K\setminus\pa^*E}\Div^K\,X\,d\H^n\,,
  \end{equation}
  for every $X\in C^1_c(\R^{n+1};\R^{n+1})$ with $X\cdot\nu_\Om=0$ on $\pa\Om$, where $\Div^K$ denotes the tangential divergence along $K$;
  \item[(ii)] there exists $\Sigma\subset K$, closed and with empty interior in $K$, such that $K\setminus\Sigma$ is a smooth hypersurface, $K\setminus(\Sigma\cup\pa E)$ is a smooth embedded minimal hypersurface, $\H^n(\Sigma\setminus\pa E)=0$, $\Om\cap(\pa E\setminus\pa^*E)\subset \Sigma$ has empty interior in $K$, and $\Om\cap\pa^*E$ is a smooth embedded hypersurface with constant scalar (w.r.t. $\nu_E$) mean curvature $\l$.
 \end{enumerate}
\end{theorem}

\begin{remark}\label{remark mink}
  {\rm Although we do not pursue this point here, we mention that we would expect $(K,E)$ to be a proper minimizer of $\F$ among pairs $(K',E')\in\KK$ with $|E'|=\e$ (and not just when $K'=f(K)$ for a diffeomorphism $f$, as proved in \eqref{minimality KE against diffeos}). To show this we would need to approximate in energy a generic $(K',E')$ by competitors $\{F_j\}_j$ for $\psi(\e)$. The natural {\it ansatz} for this approximation would be taking $F_j=U_{\eta_j}(K'\cup E')\setminus I_{\eta_j}(K'\cap E')$ for $\eta_j\to 0^+$, where $U_\eta$ denotes the {\it open} $\eta$-neighborhood of a set. The convergence of this approximation is delicate, and can be made to work by elaborating on the ideas contained in \cite{ambrosiocolevilla,villa} at least for $(K',E')$ in certain subclasses of $\KK$.}
\end{remark}

\begin{remark}
  {\rm Theorem \ref{thm basic regularity} points at two interesting free boundary problems. The first problem concerns the size and properties of $\pa E\setminus\pa^*E$, which is the transition region between constant and zero mean curvature; similar free boundary problems (on graphs rather than on unconstrained surfaces) have been considered, e.g., in \cite{cjk,caffasilvasavinDOUBLE,caffasilvasavinDOUBLE2}. The second problem concerns the wetted region $\pa\Om\cap\pa E$, which could either be $\H^n$-negligible or not, recall Figure \ref{fig example23}: in the former case, $\pa\Om\cap\pa E$ should be $(n-1)$-dimensional, while in the latter case $\pa\Om\cap\pa E$ should be a set of finite perimeter inside $\pa\Om$, and Young's law $\nu_\Om\cdot\nu_E=0$ should hold at generic boundary points of $\pa\Om\cap\pa E$ relative to $\pa\Om$; see for example \cite{dephilippismaggiCAP-ARMA,dephilippismaggiCAP-CRELLE}.}
\end{remark}

\noindent {\bf Convergence towards Plateau's problem}: The next theorem establishes the nature of Plateau's problem $\ell$ as the singular limit of the capillarity problems $\psi(\e)$ as $\e\to0^+$.

\begin{theorem}
  [Plateau's problem as a singular limit of capillarity problems]\label{thm convergence as eps goes to zero} If $\ell<\infty$, $\pa W$ be smooth, and \eqref{hp on W and C 0} and \eqref{hp on W and C} hold, then $\psi$ is lower semicontinuous on $(0,\infty)$ and
  \begin{equation}
    \label{convergence of optimal value}
    \lim_{\e\to 0^+}\psi(\e)=2\,\ell\,.
  \end{equation}
  In addition, if $\{(K_h,E_h)\}_h$ is a sequence of generalized minimizers of $\psi(\e_h)$ for $\e_h\to 0^+$ as $h\to\infty$, then there exists a minimizer $S$ in $\ell$ such that, up to extracting subsequences and as $h\to\infty$,
  \begin{equation}
    \label{weak star convergence of gen minimizers}
      \H^n\llcorner(\Om\cap\pa^*E_h)+2\,\H^n\llcorner(K_h\setminus\pa^*E_h)\weakstar\,2\,\H^n\llcorner S\,,\qquad\mbox{as Radon measures in $\Om$}\,.
  \end{equation}
\end{theorem}

\begin{remark}
  {\rm The behavior of $\psi(\e)-2\,\ell$ as $\e\to 0^+$ is expected to depend heavily on whether minimizers of $\ell$ have or do not have singularities, as noticed in \eqref{psi eps quadro expected} and \eqref{psip primo meno infinito}. In particular, we expect $\psi'(\e)\to 0^+$ only in special situations: when this happens, we have a vanishing mean curvature approximation of Plateau's problem which is related to Rellich's conjecture, see e.g. \cite{breziscoronRELL}.}
\end{remark}

\begin{remark}\label{remark L1}
  {\rm The Hausdorff convergence of $K_h$ to $S$ is not immediate (nor is the convergence in varifolds sense). Given \eqref{weak star convergence of gen minimizers}, Hausdorff convergence would follow from an area lower bound on $K_h$. In turn, this could be deduced (thanks to area monotonicity) from a uniform $L^p$-bound, for some $p>n$, on the mean curvature vectors $\vec{H}_{V_h}$ of the integer varifolds $V_h$ supported on $K_h$, with multiplicity $2$ on $K_h\setminus\pa^*E_h$, and multiplicity $1$ on $\pa^*E_h$. Notice however that, by \eqref{stationary main}, if $\l_h$ is the Lagrange multiplier of $(K_h,E_h)$, then $\vec{H}_{V_h}=\l_h\,\nu_{E_h}\,1_{\pa^*E_h}$,  so that, even when $n=1$, the only uniform $L^p$-bound that can hold is the one with $p=1$; see Example \ref{example tripe sing}.}
\end{remark}

\noindent {\bf Proofs}: We approach Theorem \ref{thm lsc} with the method introduced in \cite{DLGM} to solve \eqref{plateau problem}, which is now briefly summarized. The idea in \cite{DLGM} is considering a minimizing sequence $\{S_j\}_j$ for $\ell$, which (up to extracting subsequences) immediately leads to a sequence of Radon measures $\mu_j=\H^n\llcorner S_j\weakstar\mu$ as Radon measures in $\Om$, with $S=\spt\,\mu$ $\C$-spanning $W$. By comparing $S_j$ with its cup competitors $S_j'$, see
\begin{figure}
  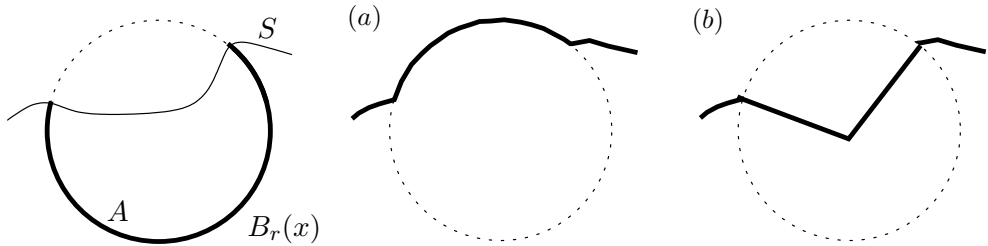\caption{{\small (a) the cup competitor of a set $S$ in $B_r(x)$ relative to an $\H^n$-maximal connected component $A$ of $\pa B_r(x)\setminus S$; (b) the cone competitor of $S$ in $B_r(x)$.}}\label{fig dlgm}
\end{figure}
Figure \ref{fig dlgm}-(a), and then letting $j\to\infty$, it is shown that $\mu(B_r(x))\ge \theta_0(n)\,r^n$ for every $x\in\spt\,\mu$; by comparing $S_j$ with its cone competitors $S_j'$, and then letting $j\to\infty$, it is proved that $r^{-n}\,\mu(B_r(x))$ is increasing in $r$. By Preiss' theorem \cite{preiss,DeLellisNOTES} it follows that $\mu=\theta\,\H^n\llcorner S$ and that $S$ is $\H^n$-rectifiable. Finally, spherical isoperimetry and a geometric argument imply that $\theta\ge 1$ $\H^n$-a.e. on $S$, which in turn suffices to conclude that $S$ is a minimizer in $\ell$ since, by lower semicontinuity,  $\H^n(S)\le\mu(\Om)\le\liminf_j\mu_j(\Om)=\ell$, and because $S$ is in the competition class of $\ell$.

Adapting this approach to a minimizing sequence $\{E_j\}_j$ for $\psi(\e)$ requires the introduction of new ideas. First, cup and cone competitors for $\{E_j\}_j$ have to be defined as {\it boundaries}, a feature that requires taking into consideration two kind of cup competitors, and that also leads to other difficulties. Second, local variations need to be compensated by volume-fixing variations, which must be uniform along the elements of the minimizing sequence. At this stage, we can prove that $\mu_j=\H^n\llcorner(\Om\cap\pa E_j)\weakstar\mu=\theta\,\H^n\llcorner K$ for an $\H^n$-rectifiable set $K$ which is $\C$-spanning $W$. The same argument as in \cite{DLGM} shows that $\theta\ge 1$, and the lower bound $\theta\ge 2$ $\H^n$-a.e. on $K\setminus\pa^*E$ requires a further elaboration which takes into account that we are considering the convergence of boundaries. We cannot conclude that $\F(K,E)=\psi(\e)$ just by lower semicontinuity because clearly $(K,E)$ is not in the competition class of $\psi(\e)$. We thus improve lower semicontinuity by some non-concentration estimates: at infinity, at the boundary and by folding against $K$. The latter are the most interesting ones, and they require a careful comparison argument based on the introduction of a third kind of competitors, called slab competitors. The construction of the various competitors is discussed in section \ref{section five competitors}, while the proof of Theorem \ref{thm lsc} is contained in section \ref{section existence of generalized minimizers}. Slab competitors are also used in the delicate proof of \eqref{minimality KE against diffeos}, whose starting point are some ideas originating in \cite{depauwHardt}, as further developed in \cite{DLGM} when addressing the formulation of Plateau's problem for David's sliding minimizers; see section \ref{section theorem basic regularity}. Finally, in section \ref{section convergence to plateau} we prove Theorem \ref{thm convergence as eps goes to zero}: the main difficulty, explained there in more detail, is that, at vanishing volume, we have no non-trivial local limit sets to be used for constructing uniform volume-fixing variations.

\bigskip

\noindent {\bf Structure of generalized minimizers:} Theorem \ref{thm lsc}, Theorem \ref{thm basic regularity} and Theorem \ref{thm convergence as eps goes to zero} lay the foundations to study the properties of generalized minimizers of $\psi(\e)$. The most intriguing questions are concerned with the relations between the properties of minimizers in Plateau's problem $\ell$, like the presence or the absence of singularities, and the properties of minimizers in $\psi(\e)$ at small $\e$: collapsing vs non-collapsing and the sign of $\l$, limiting behavior of $\l$ as $\e\to 0^+$, dimensionality of the wetted part of the wire, etc. This is of course a very large set of problems, which will require further investigations. In the companion paper \cite{kms2}, we start this kind of study by proving that collapsed minimizers have non-positive Lagrange multipliers, deduce from this property that they satisfy the convex hull property, and lay the ground for the forthcoming paper \cite{kms3}, where we further investigate the regularity of the collapsed set $K\setminus\pa^*E$.

\bigskip

\noindent {\bf Acknowledgement:} We thank an anonymous referee for several useful remarks that helped us improving the quality of the paper. Antonello Scardicchio has contributed with many inspiring discussions to the physical background of this work. This work was completed during the Spring 2019 while FM was first a member of IAS in Princeton, through support from the Charles Simonyi Endowment, and then a visitor of ICTP in Trieste.  All the authors were supported by the NSF grants DMS-1565354, DMS-RTG-1840314 and DMS-FRG-1854344.

\section{Cone, cup and slab competitors, nucleation and collapsing}\label{section five competitors} Section \ref{section notation} contains the notation and terminology used in the paper. Section \ref{section preliminaries} collects some basic properties of $\C$-spanning sets. Sections \ref{section cup first}, \ref{section slab} and \ref{section cone} deal with cup, slab and cone competitors. Section \ref{section nucleation} contains the nucleation lemma for volume-fixing variations, and section \ref{section spherical collapsing} concerns density lower bounds for collapsing sequences of sets of finite perimeter.

\subsection{Notation and terminology}\label{section notation} We denote by $|A|$ and $\H^s(A)$ the Lebesgue and the $s$-dimensional Hausdorff measures of $A\subset\R^{n+1}$, by $I_\eta(A)$ and $U_\eta(A)$ the closed and open $\eta$-neighborhoods of $A$, by $B_r(x)$ the open ball of center at $x$ and radius $r$. We work in the framework of \cite{SimonLN,AFP,maggiBOOK}. Given $k\in\N$, $1\le k\le n$, a Borel set $M\subset\R^{n+1}$ is {\bf countably $\H^k$-rectifiable} if it is covered by countably many Lipschitz images of $\R^k$; it is {\bf (locally) $\H^k$-rectifiable} if, in addition, $M$ is (locally) $\H^k$-finite. If $M$ is locally $\H^k$-rectifiable, then for $\H^k$-a.e. $x\in M$ there exists a unique $k$-plane $T_xM$ such that, as $r\to 0^+$, $\H^k\llcorner(M-x)/r\weakstar\H^k\llcorner T_xM$ as Radon measures in $\R^{n+1}$; $T_xM$ is called the {\bf approximate tangent plane to $M$ at $x$}. Given a Lipschitz map $f:\R^{n+1}\to\R^{n+1}$, we denote by $J^Mf$ its {\bf tangential jacobian along $M$}, so that if $f$ is smooth and $f(x)=x+t\,X(x)+{\rm o}(t)$ in $C^1$ as $t\to 0^+$, then $J^Mf=1+t\,\Div^MX+{\rm o}(t)$ where $\Div^M X$ is the {\bf tangential divergence of $X$ along $M$}; moreover, $M$ has {\bf distributional mean curvature vector $\vec{H}\in L^1_{\rm loc}(U;\H^k\llcorner M)$ in $U$} open, if
\[
\int_M\,\Div^M\,X\,d\H^k=\int_M\,X\cdot\vec{H}\,d\H^k\,,\qquad\forall X\in C^\infty_c(U;\R^{n+1})\,,
\]
see \cite[Sections 8 and 9]{SimonLN}. A Borel set $E\subset\R^{n+1}$ has {\bf finite perimeter} if there exists an $\R^{n+1}$-valued Radon measure on $\R^{n+1}$, denoted by $\mu_E$, such that $\langle\mu_E,X\rangle=\int_E\Div X$ whenever $X\in C^1_c(\R^{n+1};\R^{n+1})$ and $P(E;\R^{n+1})=|\mu_E|(\R^{n+1})<\infty$. The set of points $x\in\R^{n+1}$ such that $|\mu_E|(B_r(x))^{-1}\,\mu_E(B_r(x))\to\nu_E(x)\in\SS^n$ as $r\to 0^+$ is denoted by $\pa^*E$, and called the {\bf reduced boundary $\pa^*E$} of $E$. Then $\mu_E=\nu_E\,\H^n\llcorner\pa^*E$, $\pa^*E$ is $\H^n$-rectifiable in $\R^{n+1}$, and $T_x\pa^*E=\nu_E(x)^\perp$ for every $x\in\pa^*E$. The {\bf set $E^{(t)}$ of points of density $t\in[0,1]$ of $E$} is given by those $x\in \R^{n+1}$ with $|E\cap B_r(x)|/|B_r(x)|\to t$ as $r\to 0^+$, and (see, e.g., see \cite[Theorem 16.2]{maggiBOOK}),
\begin{equation}
  \label{federers theorem}
  \mbox{$\{\pa^*E,E^{(0)},E^{(1)}\}$ is a partition of $\R^{n+1}$ modulo $\H^n$}\,.
\end{equation}
Federer's criterion \cite[4.5.11]{FedererBOOK} states that if the {\bf essential boundary} $\pa^\ee E=\R^{n+1}\setminus(E^{(0)}\cup E^{(1)})$ is $\H^n$-finite, then $E$ is of finite perimeter in $\R^{n+1}$. If $E$ is open, then $\pa^\ee E\subset\pa E$: hence, if $E\in\E$ and $\H^n(\pa\Om)<\infty$, then $E$ is of finite perimeter.

\subsection{Some preliminary results}\label{section preliminaries} In the following, $W$ is a compact set, $\C$ a spanning class for $W$ and $\Om=\R^{n+1}\setminus W$.

\begin{lemma}\label{statement K spans}
 If $\{K_j\}_j$ are relatively closed sets in $\Om$, such that each $K_j$ is $\C$-spanning $W$ and $\H^n\llcorner K_j\weakstar\mu$ as Radon measures in $\Om$, then $K=\Om \cap \spt\mu$ is $\C$-spanning $W$.
\end{lemma}

\begin{proof} See \cite[Step 2, proof of Theorem 4]{DLGM}.
\end{proof}

\begin{lemma}\label{lemma 10}
  Let $K$ be relatively closed in $\Om$ and let $B_r(x)\cc\Om$. Then $K$ is $\C$-spanning $W$ if and only if, whenever $\g\in\C$ is such that $\g\cap K\setminus B_r(x)=\emptyset$, then there exists a connected component of $\g\cap \cl(B_r(x))$ which is diffeomorphic to an interval, and whose end-points belong to distinct connected components of $\cl(B_r(x))\setminus K$, as well as to distinct components of $\pa B_r(x)\setminus K$.
\end{lemma}

\begin{proof}
  This is \cite[Lemma 10]{DLGM}.
\end{proof}

\begin{lemma}\label{statement spanning is close by Lipschitz maps}
  If $K$ is $\C$-spanning $W$, $B_r(x)\cc \Om$, and $f:\R^{n+1}\to\R^{n+1}$ is a bi-Lipschitz map with $\{f\ne\id\}\cc B_r(x)$ and $f(B_r(x))\subset B_r(x)$, then $f(K)$ is $\C$-spanning $W$.
\end{lemma}

\begin{proof}
  By $f(K)\setminus B_r(x)=K\setminus B_r(x)$, if $f(K)$ is not $\C$-spanning $W$, then there exists $\g\in\C$ with $\g\cap K\setminus B_r(x)=\emptyset$ such that $\g\cap f(K)=\emptyset$. Hence, the curve $\tilde \g := f^{-1} \circ \g$ is a continuous embedding of $\mathbb{S}^1$ in $\Omega$, homotopic to $\g$ in $\Omega$, and such that $\tilde \g \cap K = \emptyset$. Since $\tilde \g$ and $W$ are compact and $K$ is closed, $\tilde \g$ has positive distance from $K\cup W$, and by smoothing out $\tilde\g$ we define a smooth embedding $\hat \g$ of $\mathbb{S}^1$ into $\Omega$, disjoint from $K$, and homotopic to $\tilde \g$ (and therefore to $\g$) in $\Omega$, a contradiction.
\end{proof}

\begin{lemma}
  \label{lemma close by Lipschitz at boundary}
  If $\pa\Om$ is smooth, then there exists $r_0>0$ with the following property. If $x\in\pa \Om$, $\Om \subset \Om'$, $f:\cl(\Om)\to\cl(\Om')=f(\cl(\Om))$ is a homeomorphism with $f(\pa\Om)=\pa\Om'$, $\{f\ne\id\}\cc B_{r_0}(x)$, and $f(B_{r_0}(x)\cap\cl(\Omega))=B_{r_0}(x)\cap\cl(\Omega')$, and if $K$ is $\C$-spanning $W$, then $K'=f(K\cap\Om^*)$ is relatively closed in $\Om$ and is $\C$-spanning $W$, where $\Om^*=f^{-1}(\Om)$.
\end{lemma}

\begin{proof}
  {\it Step one}: We show that, for $K$ relatively closed in $\Om$ and $B_{r_0}(x)$ as in the statement, $K$ is $\C$-spanning $W$ if and only if, whenever $\g\in\C$ is such that $\g\cap K\setminus B_{r_0}(x)=\emptyset$, then there exists a connected component of $\g\cap \cl(B_{r_0}(x))$, diffeomorphic to an interval, and whose end-points belong to distinct connected components of $\Om\cap\cl(B_{r_0}(x))\setminus K$. We only prove the ``only if'' part. First of all, we notice that $\gamma$ cannot be contained in $\Omega\cap B_{r_0}(x)$, because $r_0$ can be chosen small enough to ensure that $\Omega\cap B_{r_0}(x)$ is simply connected, and because $\ell<\infty$ implies that no element of $\C$ is homotopic to a constant. Arguing as in \cite[Step two, proof of Lemma 10]{DLGM} we can assume that $\g$ and $\pa B_{r_0}(x)$ intersect transversally, so that there exist finitely many disjoint $I_i=[a_i,b_i]\subset\SS^1$ such that $\g\cap\cl(B_{r_0}(x))=\bigcup_i\g(I_i)$ with $\g\cap\pa B_{r_0}(x)=\bigcup_i\{\g(a_i),\g(b_i)\}$ and $\g\cap B_{r_0}(x)=\bigcup_i\g((a_i,b_i))$. Assume by contradiction that for each $i$ there exists a connected component $A_i$ of $\Om\cap\cl(B_{r_0}(x))\setminus K$ such that $\g(a_i),\g(b_i)\in A_i$. If $r_0$ is small enough, then $\cl(\Om\cap B_{r_0}(x))$ is diffeomorphic to $\cl(B_1(0)\cap\{x_1>0\})$ through a diffeomorphism mapping $B_{r_0}(x)\cap\pa\Om$ into $B_1(0)\cap\{x_1=0\}$. Using this fact and the connectedness of each $A_i$, we define smooth embeddings $\tau_i:I_i\to A_i$ with $\tau_i(a_i)=\g(a_i)$, $\tau_i(b_i)=\g(b_i)$ and $\tau_i$ homotopic in $\Om\cap B_{r_0}(x)$ to the restriction of $\gamma$ to $I_i$. Moreover, this can be done with $\tau_i(I_i)\cap\tau_j(I_j)=\emptyset$. The new embedding $\bar\g$ of $\SS^1$ obtained by replacing $\g$ with $\tau_i$ on $I_i$ is thus homotopic to $\g$ in $\Om$, and such that $\bar\g\cap K=\emptyset$, a contradiction.

  \medskip

  \noindent {\it Step two}: Since $K\cap\Om^*$ is relatively closed in $\Om^*$, $K'=f(K\cap\Om^*)$ is relatively closed in $\Om=f(\Om^*)$. Should $K'$ not be $\C$-spanning $W$, given that $K'\setminus B_{r_0}(x)=K\setminus B_{r_0}(x)$, we could find $\g\in\C$ with $\g\cap K\setminus B_{r_0}(x)=\emptyset$ and $\g\cap K'=\emptyset$. By step one, there would be a connected component $\s$ of $\g\cap \cl(B_{r_0}(x))$, diffeomorphic to an interval, and such that: (i) the end-points $p$ and $q$ of $\s$ (which lie on $\pa B_{r_0}(x)$) belong to distinct connected components of $\Om\cap\cl(B_{r_0}(x))\setminus K$; and (ii) $p$ and $q$ belong to the same connected component of $\Om\cap\cl(B_{r_0}(x))\setminus K'$. Since $f$ is a homeomorphism, $f(p)=p$, and $f(q)=q$, by (i) we would find that $p$ and $q$ belong to {\it distinct} connected components of
  \[
  f\big(\Om\cap\cl(B_{r_0}(x))\setminus K\big)=\Om'\cap\cl(B_{r_0}(x))\setminus f(K)\;
  \]
  while, by (ii), there would be an arc connecting $p$ and $q$ in $\Om\cap\cl(B_{r_0}(x))\setminus K'$, where
  \begin{eqnarray*}
  \Om\cap\cl(B_{r_0}(x))\setminus K'&=&\Om\cap\cl(B_{r_0}(x))\setminus f(K\cap\Om^*)
  \\
  &=&\Om\cap\cl(B_{r_0}(x))\setminus f(K)
  \,\,\subset\,\,\Om'\cap\cl(B_{r_0}(x))\setminus f(K)\,,
  \end{eqnarray*}
  and hence $p$ and $q$ would belong to a {\it same} component of $\Om'\cap\cl(B_{r_0}(x))\setminus f(K)$.
\end{proof}

\subsection{Cup competitors}\label{section cup first} Given $E\in\E$, $B_r(x)\cc\Om$ and a connected component $A$ of $\pa B_r(x)\setminus\pa E$, cup competitors are used to compare $\H^n(B_r(x)\cap\pa E)$ with $\H^n(\pa B_r(x)\setminus A)$. The construction is more involved than in the case of Plateau's problem considered in \cite{DLGM} as we need to construct cup competitors as {\it boundaries}, and we have to argue differently depending on whether $A\cap E=\emptyset$ or $A\subset E$.

\begin{lemma}[Cup competitors]\label{lemma cup competitor first kind}
  Let $E\in\E$ be such that $\Om \cap \pa E$ is $\C$-spanning $W$, let $x\in\Om$, $0<r<\dist(x,\pa\Om)$, and let $A$ be a connected component of $\pa B_r(x)\setminus\pa E$. Assume that $\pa E\cap\pa B_r(x)$ is $\H^{n-1}$-rectifiable. Then, for every $\eta \in \left( 0, r/2 \right)$ there exists a set $F = F_\eta \in \E$ so that $\Omega \cap \pa F$ is $\C$-spanning $W$, and
\begin{eqnarray}\label{cup fuori da br chiusa}
&& \pa F\setminus\cl(B_r(x))=\pa E\setminus\cl(B_r(x))\,,
  \\
  \label{cup buccia}
&& \lim_{\eta\to 0^+} \H^n \big( (\pa B_r (x) \cap \pa F) \, \Delta \, (\pa B_r (x) \setminus A) \big) = 0\,,
  \\
  \label{cup area totale}
 && \limsup_{\eta\to 0^+}\H^n(\Om\cap\pa F)\le\H^n\big(\Omega \cap \pa E\setminus B_r(x)\big)+2\,\H^n(\pa B_r(x)\setminus A)\,.
  \end{eqnarray}
  Moreover,
  \begin{itemize}

\item[(i)]   If $A\cap E=\emptyset$, then

\begin{eqnarray}
    \label{cup first br area}
    \limsup_{\eta\to 0^+}\H^n(B_r(x)\cap\pa F)\le\H^n\Big(\pa B_r(x)\setminus \big(A\cup (E\cap\pa B_r)\big)\Big)\,;
  \end{eqnarray}

  \item[(ii)] If $A\subset E$, then

\begin{eqnarray}
  \label{cup second br area}
    \limsup_{\eta\to 0^+}\H^n(B_r(x)\cap\pa F)\le\H^n\big(E\cap\pa B_r(x)\setminus A\big)\,.
  \end{eqnarray}

  \end{itemize}

  \end{lemma}

  \begin{remark} \label{rmk cup}

  Before proceeding with the proof of the lemma, let us first provide some additional details on the construction of the competitors $F=F_\eta$, which, as anticipated, is different depending on whether $A \cap E = \emptyset$ or $A \subset E$. In what follows, given $Y\subset\pa B_r(x)$, we set
\[
N_\eta(Y)=\Big\{y-t\,\nu_{B_r(x)}(y):y\in Y\,,t\in(0,\eta)\Big\}\,,\qquad 0<\eta<r\,.
\]
{\it The case when $A \cap E = \emptyset$:} In this case, we define
\begin{equation} \label{who is Y}
Y=\pa B_r(x)\setminus\big(\cl(E\cap\pa B_r(x))\cup\cl(A)\big)\,,
\end{equation}
and then we further distinguish two scenarios, depending on whether the set
\begin{equation} \label{who is S}
S = \pa E \cap \cl (A) \setminus \left[ \cl\left(E \cap \pa B_r (x) \right) \cup \cl(Y) \right]
\end{equation}
is empty or not. When $S=\emptyset$ the cup competitor defined by $E$ and $A$ is given by
  \begin{equation}
    \label{cup competitor first case lemma}
      F=\big(E\setminus\cl(B_r(x))\big)\cup\,N_\eta(Y)\,,
  \end{equation}
  see
  \begin{figure}
  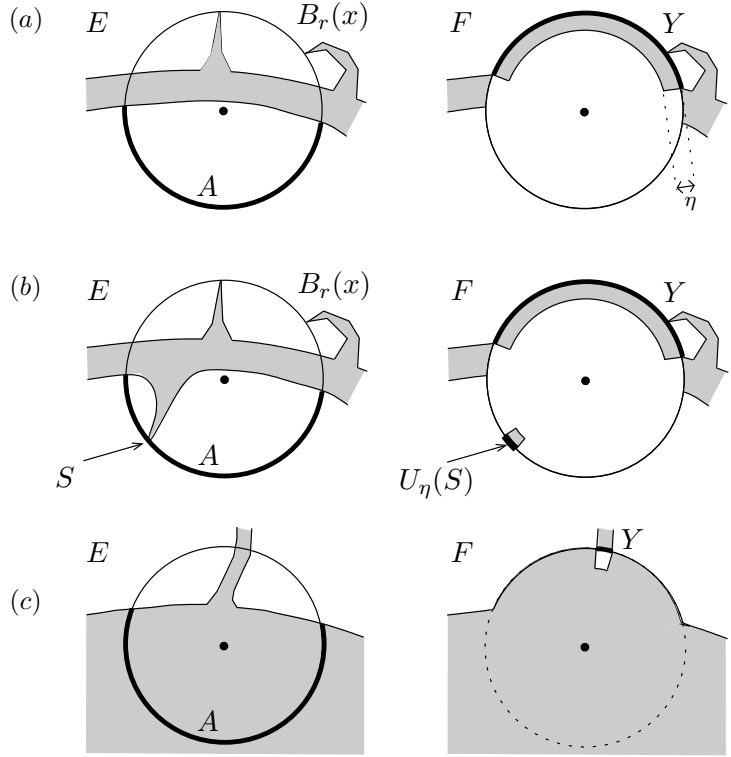\caption{{\small Cup competitors when: (a) $A\cap E=\emptyset$ and $S = \emptyset$; (b) $A \cap E = \emptyset$ and $S \neq \emptyset$; (c) $A\subset E$. Picture (b) really pertains to the case $n\ge 2$, in which the component $A$ in the picture is not necessarily disconnected by the presence of $S$. In the situation of picture (b) the set $F$ defined by \eqref{cup competitor first case lemma} may fail to intersect a test curve $\g$ which was intersecting with $\Om\cap\pa E$ only at points in $S$.}}\label{fig cup00}
  \end{figure}
  Figure \ref{fig cup00}-(a), and step one of the proof. When $S \neq \emptyset$, see Figure \ref{fig cup00}-(b), if we define $F$ as in \eqref{cup competitor first case lemma}, then $\Om \cap \pa F$ may fail to be $\C$-spanning $W$; we thus need to modify \eqref{cup competitor first case lemma}, and to this end, denoting by ${\rm d}_S$ the distance function from $S$ and by $U_\eta (S) = \pa B_r (x) \cap \{{\rm d}_S (y) < \eta\}$, we set
  \begin{equation} \label{cup competitor first case difficult lemma}
  F=\big(E\setminus\cl(B_r(x))\big)\cup\,N_\eta(Z)\,, \qquad Z = Y \cup \left( U_\eta(S) \setminus \cl (E \cap \pa B_r (x)) \right )\,,
  \end{equation}
  see, again, Figure \ref{fig cup00}-(b). This situation, discussed in detail in step two of the proof, is made more delicate since we can prove that the sets defined in \eqref{cup competitor first case difficult lemma} are well-behaved in the limit as $\eta\to 0^+$ only along along a suitable sequence $\eta_k\downarrow 0^+$. For this reason, we will actually define $F_\eta$ as in \eqref{cup competitor first case difficult lemma} only when $\eta=\eta_k$, and then extend the definition by setting $F_\eta=F_{\eta_k}$ for all $\eta \in \left(\eta_{k+1}, \eta_{k}\right)$ (so that, for the sake of homogenity, \eqref{cup area totale} can be stated as an $\eta\to 0^+$-limit in all three cases).

\medskip

\noindent {\it The case when $A \cap E = \emptyset$:} Finally, when $A \subset E$ the cup competitor defined by $E$ and $A$ is given by
  \begin{equation}
    \label{cup competitor second case lemma}
  F=\big(E\cup B_r(x)\big)\setminus\cl\big(N_\eta(Y)\big)\,,\qquad Y=(E\cap\pa B_r(x))\setminus\cl(A)\,,
  \end{equation}
  see Figure \ref{fig cup00}-(c). We treat this case in step three of the proof.
\end{remark}

\begin{proof}
  {\it Step one}: We assume that $A \cap E = \emptyset$ and, after defining $Y$ as in \eqref{who is Y} and $S$ as in \eqref{who is S}, we suppose first that
  \begin{equation} \label{cc1_easy}
   		S = \emptyset\,.
  \end{equation}
    We then define $F$ by \eqref{cup competitor first case lemma}. For the sake of brevity we set $B_r=B_r(x)$. We claim that \eqref{cup fuori da br chiusa} holds, and that we have
  \begin{eqnarray}
    \label{posse 3 1}
    B_r\cap\pa F&=&B_r\cap\pa N_\eta(Y)\,,
    \\
    \label{posse 0 1}
    Y&\subset&\pa F\cap\pa B_r\,,
    \\\label{posse 0 2}
    E\cap\pa B_r&\subset&\pa F\cap\pa B_r\,,
    \\\label{posse 0 3 1}
    \pa B_r\setminus\cl(A)&\subset&\,\pa F\cap\pa B_r\,,
     \\ \label{posseno ammazzarlo}
    \pa E \cap \pa B_r & \subset & \pa F \cap \pa B_r\,,
    \\\label{posse 0 4 1}
    \mbox{$A$, $E\cap\pa B_r$, $Y$ are open and disjoint in $\pa B_r\,$,}\hspace{-4cm}
    \\\label{posse 0 3 2}
    \pa F\cap\pa B_r&\subset&\pa B_r\setminus A\,,
    \\\label{posse 0 4 2}
    \pa B_r\setminus\cl(E)&\subset& A\cup Y\,,
    \\\label{posse 0 4}
    \cl(Y)\setminus Y&\subset&\pa B_r\cap\pa E\,,
    \\\label{posse 0 4 A}
    \cl(A)\setminus A&\subset&\pa B_r\cap\pa E\,,
    \\\label{posse 0 4 E}
    \cl(E\cap\pa B_r)\setminus (E\cap\pa B_r)&=&\pa B_r\cap\pa E\,.
  \end{eqnarray}
  Indeed, \eqref{cup fuori da br chiusa} and \eqref{posse 3 1} follow from $F\cap B_r=N_\eta(Y)\cap B_r$ and $F\setminus\cl(B_r)=E\setminus\cl(B_r)$. To prove \eqref{posse 0 1}: $Y\subset\cl(N_\eta(Y))$ gives $Y\subset\cl(F)$, and $F\cap\pa B_r=\emptyset$ implies  $Y\cap F=\emptyset$. To prove \eqref{posse 0 2}: $E\cap\pa B_r\subset\cl(E\setminus\cl(B_r))$, so that $E\cap\pa B_r\subset\cl(F)$, while $F\cap\pa B_r=\emptyset$ gives $(E\cap\pa B_r)\cap F=\emptyset$. \eqref{posse 0 4 1} is obvious, and \eqref{posse 0 3 1} follows from \eqref{posse 0 1} and \eqref{posse 0 2}. \eqref{posseno ammazzarlo} is then an immediate consequence of \eqref{posse 0 1}, \eqref{posse 0 2}, \eqref{posse 0 3 1}, and the condition in \eqref{cc1_easy}. To prove \eqref{posse 0 3 2}: $A$ is open in $\pa B_r\setminus\pa E$ and $A\cap E=\emptyset$, thus $A\cap\cl(E)=\emptyset$; moreover, $A\cap\cl(Y)=\emptyset$ by \eqref{posse 0 4 1}, hence
  \[
  \pa F\cap\pa B_r\,\,\subset\,\,\cl(F)\cap\pa B_r\,\,\subset\,\,\cl(E)\cup\Big(\cl(N_\eta(Y))\cap\pa B_r\Big)\,\,=\cl(E)\cup\cl(Y)\,,
  \]
  and we deduce \eqref{posse 0 3 2}. To prove \eqref{posse 0 4 2}: if $y\in\pa B_r\setminus\cl(E)$, then $y$ belongs to one of the open connected components of $\pa B_r\setminus\pa E$, so it is either $y\in A$, or $y\in \pa B_r\setminus\cl(A)\subset Y$. To prove \eqref{posse 0 4}: by \eqref{posse 0 4 1} we have $A\cap\cl(Y)=\emptyset$, so that by \eqref{posse 0 4 2}
  \[
  \cl(Y)\setminus Y\,\subset\,\pa B_r\setminus(A\cup Y)\,\subset\,\pa B_r\cap\cl(E)\,,
  \]
  and we conclude by $(E\cap\pa B_r)\cap\cl(Y)=\emptyset$ (again, thanks to \eqref{posse 0 4 1}). Finally, \eqref{posse 0 4 A} and the inclusion ``$\subset$'' in \eqref{posse 0 4 E} are obvious, while the other inclusion in \eqref{posse 0 4 E} follows from \eqref{cc1_easy}. Having proved the claim, we complete the proof. By definition, $F \subset \Omega$ is open. We show that $\Om\cap\pa F$ is $\C$-spanning $W$. Given $\g\in\C$, if $\g\cap\pa E\setminus\cl(B_r)\ne\emptyset$, then $\g\cap\pa F\ne\emptyset$ by \eqref{cup fuori da br chiusa}; if instead $\g\cap\pa E\setminus\cl(B_r)=\emptyset$, then necessarily $\g \cap \pa E \cap \cl (B_r) \neq \emptyset$. Now, if $\g \cap \pa E \cap \pa B_r \neq \emptyset$ then $\g \cap \pa F \neq \emptyset$ by \eqref{posseno ammazzarlo}; otherwise we actually have $\g \cap \pa E \setminus B_r = \emptyset$, and thus, by Lemma \ref{lemma 10}, $\g$ intersects two distinct connect components of $\pa B_r\setminus\pa E$, and at least one of them is contained in $\pa F\cap\pa B_r$: indeed, $\pa F\cap\pa B_r$ contains $\pa B_r\setminus\cl(A)$ by \eqref{posse 0 3 1}, where $\cl(A)$ is disjoint from all the connected components of $\pa B_r\setminus\pa E$ that are different from $A$.

  Now, we prove \eqref{cup buccia}, \eqref{cup area totale}, and \eqref{cup first br area}. First notice that \eqref{posse 0 3 1}, \eqref{posse 0 3 2}, \eqref{posse 0 4 A}, and $\H^n (\pa B_r \cap \pa E)=0$ imply that
  \begin{eqnarray} \label{cup buccia stronger}
  \pa F \cap \pa B_r &=& \pa B_r \setminus A \qquad \mbox{modulo $\H^n$}\,,
  \end{eqnarray}
  which in turn implies \eqref{cup buccia}. Next, we claim that
  \begin{eqnarray}\label{cup area total eta first}
    \H^n(\Om\cap\pa F)&\le&\H^n\big(\Omega \cap \pa E\setminus B_r\big)+\H^n(E\cap\pa B_r)
    \\\nonumber
    &&+\big(2+C(n)\,\eta\big)\,\H^n\Big(\pa B_r\setminus \big(A\cup(E\cap\pa B_r)\big)\Big)
    +C(n)\,\eta\,\H^{n-1}\big(\pa E\cap\pa B_r\big)\,.
  \end{eqnarray}
  To prove the claim, first by $\H^n(\pa E\cap\pa B_r)=0$, \eqref{cup fuori da br chiusa} and \eqref{posse 0 3 2} we have
  \begin{eqnarray}
      \nonumber
      \H^n(\Om\cap\pa F)&=&\H^n(\Omega \cap \pa E\setminus B_r)+\H^n(\cl(B_r)\cap\pa F)
      \\
      \label{posse c1}
      &\le&\H^n(\Omega \cap \pa E\setminus B_r)+\H^n(\pa B_r\setminus A)+\H^n(B_r\cap\pa F)\,.
  \end{eqnarray}
  If $g(y,t)=y-t\,\nu_{B_r}(y)$, then by \eqref{posse 3 1}
  \begin{eqnarray}\nonumber
  B_r\cap\pa F=B_r\cap \pa N_\eta(Y)=g(Y,\eta)\,\cup\, g\Big(\big(\cl(Y)\setminus Y\big)\times[0,\eta]\Big)\,,
  \end{eqnarray}
  so that \eqref{posse 0 4}, the $\H^{n-1}$-rectifiability of $\pa E\cap\pa B_r$, and the area formula give us
  \begin{eqnarray}\label{posse c3}
  \H^n(B_r\cap \pa F)\le (1+C(n)\,\eta)\,\H^n(Y)+C(n)\,\eta\,\H^{n-1}(\pa E\cap\pa B_r)\,.
  \end{eqnarray}
  By $\H^n(\pa E\cap\pa B_r)=0$, \eqref{posse 0 4 A} and \eqref{posse 0 4 E} we have
  \begin{equation}
    \label{posse c4}
      \H^n(Y)=\H^n\big(\pa B_r\setminus \big(A\cup(E\cap\pa B_r)\big)\big)\,,
  \end{equation}
  so that \eqref{posse c1}, \eqref{posse c3} and \eqref{posse c4} imply \eqref{cup area total eta first}. Letting $\eta\to 0^+$ in \eqref{cup area total eta first} we find \eqref{cup area totale}, and doing the same in \eqref{posse c3} and \eqref{posse c4},  we deduce \eqref{cup first br area}.

  \medskip

  \noindent {\it Step two:} In the case $A \cap E = \emptyset$, we now allow for the set $S$ defined in \eqref{who is S} to be non-empty. In this case, if $F$ is defined as in \eqref{cup competitor first case lemma} then the inclusion \eqref{posseno ammazzarlo} is not true in general, and $\Omega \cap \pa F$  may fail to be $\C$-spanning $W$. We then modify the construction as detailed in Remark \ref{rmk cup}, defining $F$ as in \eqref{cup competitor first case difficult lemma}. We notice that $F \subset \Omega$ is open, and that \eqref{cup fuori da br chiusa} holds true, since once again $F \setminus \cl (B_r) = E \setminus\cl (B_r)$. Moreover, we have
 \begin{eqnarray}
 \label{fix:in Br}
 B_r \cap \pa F &=& B_r \cap \pa N_\eta (Z)\,,
\\ \label{fix:Z}
 Z &\subset & \pa F \cap \pa B_r\,,
 \\ \label{fix:E on the sphere}
 E \cap \pa B_r &\subset & \pa F \cap \pa B_r\,,
 \\\label{fix: all except clA}
 \pa B_r \setminus \cl (A) &\subset& \pa F \cap \pa B_r\,,
 \\ \label{fix:key for spanning}
 \pa E \cap \pa B_r &\subset & \pa F \cap \pa B_r\,,
 \\ \label{fix:disjoint comp}
 \mbox{$A$, $E\cap\pa B_r$, $Y$ are open and disjoint in $\pa B_r\,$,}\hspace{-4cm}
 \\ \label{fix:competitor shell}
 \pa F \cap \pa B_r &\subset & \left[\pa B_r \setminus A\right] \cup \left[ \pa B_r \cap \{{\rm d}_S \le \eta\} \right]\,,
 \\ \label{fix:outside of clE}
 \pa B_r \setminus \cl (E) & \subset & A \cup Y\,,
 \\ \label{fix:bdry Y}
 \cl (Y) \setminus Y & \subset & \pa B_r \cap \pa E\,,
 \\ \label{fix:bdry A}
 \cl (A) \setminus A &\subset & \pa B_r \cap \pa E\,,
 \\ \label{fix:cl E on the sphere}
 \cl (E \cap \pa B_r) \setminus (E \cap \pa B_r) & \subset & \pa B_r \cap \pa E\,.
 \end{eqnarray}
The proofs of \eqref{fix:in Br}, \eqref{fix:Z}, \eqref{fix:E on the sphere}, \eqref{fix: all except clA} are identical to the proofs of the corresponding statements in step one with $Z$ replacing $Y$; \eqref{fix:key for spanning} then follows from \eqref{fix:Z}, \eqref{fix:E on the sphere}, and \eqref{fix: all except clA}, since $S \subset U_\eta(S) \setminus \cl(E \cap \pa B_r) \subset Z$; \eqref{fix:disjoint comp} is obvious. To prove \eqref{fix:competitor shell}: as in step one, $A \cap \cl (E) = \emptyset$ and $A \cap \cl (Y) = \emptyset$ by \eqref{fix:disjoint comp}, and
\begin{eqnarray*}
\pa F \cap \pa B_r &&\subset\,\, \cl (F) \cap \pa B_r \,\,\subset \,\,\cl (E) \cup \left( \cl (N_\eta (Z)) \cap \pa B_r \right) 
\\
&&\subset\,\, \cl (E) \cup \cl (Y) \cup \cl (U_\eta(S))\,,
\end{eqnarray*}
so that \eqref{fix:competitor shell} follows from the fact that $\cl (U_\eta(S)) \subset \pa B_r \cap \{{\rm d}_S \leq \eta\}$. Next, we notice that \eqref{fix:outside of clE}, \eqref{fix:bdry Y}, \eqref{fix:bdry A}, and \eqref{fix:cl E on the sphere} are shown analogously to step one (with the identity in \eqref{posse 0 4 E} which becomes an inclusion in \eqref{fix:cl E on the sphere} due to $S$ possibly being not empty). With the above at our disposal, we proceed now to verify the claims of the lemma. First, the proof that $\Omega \cap \pa F$ is $\C$-spanning $W$ follows \emph{verbatim} the argument from step one. Next, \eqref{fix: all except clA}, \eqref{fix:competitor shell}, \eqref{fix:bdry A}, and $\mathcal{H}^n (\pa E \cap \pa B_r)=0$ imply that
\begin{equation}\label{cup buccia eta level}
\mathcal{H}^n\left((\pa F \cap \pa B_r) \, \Delta \, (\pa B_r \setminus A)\right) \leq \mathcal{H}^n (\pa B_r \cap \{{\rm d}_S\leq \eta\})\,.
\end{equation}
In particular, since $\mathcal{H}^{n-1} (S) < \infty$, it holds
\begin{equation} \label{fix:cup competitor shell}
\lim_{\eta \to 0^+} \mathcal{H}^n \left(  (\pa F \cap \pa B_r) \, \Delta \, (\pa B_r \setminus A) \right) = 0\,,
\end{equation}
that is \eqref{cup buccia}. Next, we proceed with estimating $\H^n (\Omega \cap \pa F)$. We first notice that, by \eqref{cup fuori da br chiusa} and $\H^n (\pa E \cap \pa B_r) = 0$
 \begin{eqnarray}
      \nonumber
      \H^n(\Om\cap\pa F)&=&\H^n(\Omega \cap \pa E\setminus B_r)+\H^n(\cl(B_r)\cap\pa F)
      \\
      \label{fix:est1}
      &\le&\H^n(\Om \cap \pa E\setminus B_r)+\H^n(\pa F \cap \pa B_r)+\H^n(B_r\cap\pa F)\,.
  \end{eqnarray}
Setting, as in step one, $g(y,t) = y - t\, \nu_{B_r} (y)$, we then have from \eqref{fix:in Br} that
\begin{equation} \label{fix rect 1}
B_r \cap \pa F = B_r \cap \pa N_\eta (Z) = g (Z,\eta) \cup g\left( \left( \cl (Z) \setminus Z \right) \times \left[0,\eta\right] \right)\,.
\end{equation}
By the area formula, we can easily estimate
\begin{align}
\H^n (g(Z,\eta)) &\leq  (1 + C(n)\,\eta)\, \H^n (Z) \nonumber \\
& \leq   (1 + C(n)\,\eta)\, \Big( \H^n (Y) + \H^n (\pa B_r \cap \{{\rm d}_S < \eta\}) \Big) \nonumber \\ \label{fix:est2}
& \leq (1 + C(n)\,\eta)\, \Big( \H^n (\pa B_r \setminus (A\cup (E \cap \pa B_r))) + \H^n (\pa B_r \cap \{{\rm d}_S < \eta\}) \Big) \,.
\end{align}
On the other hand, it holds
\begin{equation} \label{fix rect 2}
\cl (Z) \setminus Z \subset \left[ \cl (Y) \setminus (Y) \right] \cup [ \cl (\hat U) \setminus \hat U ]\,,
\end{equation}
where $\hat U = U_\eta(S) \setminus \cl (E \cap \pa B_r)$. Since $\cl (\hat U) \subset \cl (U_\eta(S)) \setminus (E \cap \pa B_r)$, \eqref{fix:cl E on the sphere} implies that
\begin{equation} \label{fix rect 3}
\cl(\hat U) \setminus \hat U \subset \left( \pa B_r \cap \{{\rm d}_S = \eta\} \right) \cup \left( \pa B_r \cap \pa E \right)\,,
\end{equation}
and thus \eqref{fix:bdry Y} yields
\begin{equation} \label{fix:est3}
\H^n \left( g ( (\cl(Z) \setminus Z) \times \left[0,\eta\right] ) \right) \leq C(n)\, \eta\, \Big( \H^{n-1} (\pa B_r \cap \pa E) + \H^{n-1} (\pa B_r \cap \{{\rm d}_S = \eta\}) \Big)\,.
\end{equation}
By applying the coarea formula to ${\rm d}_S$, it holds for every $0 < \sigma < r/2$
\begin{equation} \label{fix:coarea}
\int_0^\sigma \H^{n-1} (\pa B_r \cap \{{\rm d}_S = \eta\}) \, d\eta = \H^n (\pa B_r \cap \{{\rm d}_S \leq \sigma \}) < \infty\,,
\end{equation}
and thus there exists a decreasing sequence $\{\eta_k\}_{k=1}^\infty$ with $\lim_{k \to \infty} \eta_k =0$ such that $\pa B_r \cap \{{\rm d}_S = \eta_k\}$ is $\H^{n-1}$-rectifiable and
\begin{equation} \label{fix:coarea trick}
\lim_{k \to \infty} \eta_k \, \H^{n-1} (\pa B_r \cap \{{\rm d}_S =\eta_k\}) = 0\,.
\end{equation}
If $F_k$ is the sequence of cup competitors defined by \eqref{cup competitor first case difficult lemma} in correspondence with the choice $\eta=\eta_k$, we then have from \eqref{fix rect 1}, \eqref{fix rect 2}, \eqref{fix:bdry Y}, and \eqref{fix rect 3} that $\Om \cap \pa F_k$ is $\H^n$-rectifiable, and from \eqref{fix:est1}, \eqref{fix:cup competitor shell}, \eqref{fix:est2}, \eqref{fix:est3}, and \eqref{fix:coarea trick} that
\begin{eqnarray} \label{fix: est final1}
\limsup_{k \to \infty} \H^n (B_r \cap \pa F_k) &\leq& \H^n ( \pa B_r \setminus (A \cup (E \cap \pa B_r)))\,, \\ \label{fix: est final2}
\limsup_{k \to \infty} \H^n (\Omega \cap \pa F_k) &\leq& \H^n (\Omega \cap \pa E \setminus B_r) + 2\, \H^n (\pa B_r \setminus A)\,.
\end{eqnarray}
Defining $F_\eta=F_{\eta_k}$ for all $\eta \in \left( \eta_{k+1}, \eta_{k}\right)$ then allows to conclude both \eqref{cup area totale} and \eqref{cup first br area}.

  \medskip

  \noindent {\it Step three}: We now assume that $A\subset E$, and define $F$ by \eqref{cup competitor second case lemma}, that is
  \begin{equation}
  \label{cup competitor second case lemma repeat}
  F=\big(E\cup B_r\big)\setminus\cl\big(N_\eta(Y)\big)\,,\qquad Y=(E\cap\pa B_r)\setminus\cl(A)\,.
  \end{equation}
  We claim that \eqref{cup fuori da br chiusa} holds, as well as
  \begin{eqnarray}
\label{cc2_Y}
Y & \subset & \pa F \cap \pa B_r\,, \\
\label{cc2_restoftheworld}
\pa B_r \setminus E & \subset & \pa F \cap \pa B_r\,,\\
\label{cc2_keyforspanning}
\pa B_r \setminus \cl(A) & \subset &  \pa F \cap \pa B_r\,,\\
\label{cc2_in}
B_r\cap\pa F & \subset & B_r \cap \pa N_\eta(Y)\,,\\
\label{cc2_partition}
A, \pa B_r \setminus \cl(E), Y \mbox{ are open and disjoint in $\pa B_r$}\,,\hspace{-4cm}\\
\label{cc2_bdry}
\pa F \cap \pa B_r & \subset & \pa B_r \setminus A\,,\\
\label{cc2 AAA}
\cl(A)\setminus A &\subset &\pa B_r\cap\pa E\,,
\\
\label{cc2_Ybdry}
\cl(Y) \setminus Y & \subset & \pa B_r \cap \pa E\,.
 \end{eqnarray}
First, $F \setminus \cl(B_r) = E \setminus \cl(B_r)$ implies \eqref{cup fuori da br chiusa}. To prove \eqref{cc2_Y}: since $E$ is open we have $E \cap \pa B_r\subset\cl(E\setminus\cl(B_r))=\cl(F\setminus\cl(B_r))$ (by \eqref{cup competitor second case lemma repeat}), thus $Y\subset\cl(F)$; we conclude as $Y \cap F = \emptyset$. As $F\cap\pa B_r\subset E\cap\pa B_r$, to prove \eqref{cc2_restoftheworld} we just need to show that $\pa B_r\setminus E\subset\cl(F)$: since $\cl(U)\setminus\cl(V)\subset\cl(U\setminus\cl(V))$ for every $U,V\subset\R^{n+1}$, by $\pa B_r\cap\cl(N_\eta(Y))\subset\cl(E)$,
\begin{eqnarray*}
\pa B_r\setminus\cl(E)&\subset&\cl(B_r)\setminus\cl(N_\eta(Y))\subset\cl\big(B_r\setminus\cl(N_\eta(Y)\big)\subset\cl(F)\,,
\\
(\pa B_r\cap\pa E)\setminus\cl(N_\eta(Y))&\subset&\cl(E)\setminus\cl(N_\eta(Y))\,\,\subset\,\,\cl(E\setminus\cl(N_\eta(Y)))\,\,\subset\,\,\cl(F)\,,
\\
\pa B_r\cap\pa E\cap\cl(N_\eta(Y))&\subset&\pa E\cap\cl(Y)\,\,\subset\,\,\pa F\,,
\end{eqnarray*}
where the last inclusion follows by \eqref{cc2_Y}. Next, \eqref{cc2_keyforspanning} follows by \eqref{cc2_Y}, \eqref{cc2_restoftheworld} and
\[
\pa B_r \setminus \cl(A) = \left[ (E\cap\pa B_r) \setminus \cl(A) \right] \cup \left[ \pa B_r \setminus (E \cup \cl(A)) \right]
\subset Y \cup (\pa B_r \setminus E)\,.
\]
To prove \eqref{cc2_in}: setting $V^c=\R^{n+1}\setminus V$, by $B_r\cap F=B_r\cap \cl(N_\eta(Y))^c$ we find $B_r\cap\pa F= B_r \cap \pa \left[ \cl(N_\eta(Y))^c \right]$, where, as a general fact on open set $U\subset\R^{n+1}$, we have
\[
\pa[\cl(U)^c]=\cl(\cl(U)^c)\setminus\cl(U)^c=\cl(U)\cap\cl(\cl(U)^c)\,,\qquad\cl(\cl(U)^c)\subset U^c\,,
\]
and thus $\pa[\cl(U)^c]\subset\pa U$. Next, \eqref{cc2_partition} is obvious, and implies $A\cap\cl(Y)=\emptyset$ where $\cl(Y)=\cl(N_\eta(Y))\cap\pa B_r$, so that $A \cap \pa B_r \subset E \cap \pa B_r \setminus \cl(N_\eta(Y))= F \cap \pa B_r$, and \eqref{cc2_bdry} follows. To prove \eqref{cc2 AAA}, just notice that $A\subset E$ and $A$ is a connected component of $\pa B_r\setminus\pa E$. To prove \eqref{cc2_Ybdry}: trivially, $\cl(Y)\setminus Y\subset\cl(Y) \subset \pa B_r \cap \cl(E)$, while by definition of $Y$ and by $\cl(Y)\cap A=\emptyset$
\begin{eqnarray*}
E\cap(\cl(Y)\setminus Y)&=&\big(\cl(Y)\cap(E\cap\pa B_r)\big)\setminus Y=\cl(Y)\cap(E\cap\pa B_r)\cap\cl(A)
\\
&=&(E\cap\pa B_r)\cap\cl(Y)\cap\pa A\subset E\cap(\cl(A)\setminus A)=\emptyset\,,
\end{eqnarray*}
thanks to \eqref{cc2 AAA}. We have completed the claim. Next, by \eqref{cc2_keyforspanning}, \eqref{cc2_bdry}, \eqref{cc2 AAA}, and by $\H^n(\pa B_r\cap\pa E)=0$, we deduce \eqref{cup buccia stronger} and thus \eqref{cup buccia}, while $\Om \cap \pa F$ is $\C$-spanning $W$ thanks to \eqref{cup fuori da br chiusa}, Lemma \ref{lemma 10}, \eqref{cc2_keyforspanning}, and \eqref{cc2_restoftheworld}. Finally,
 \begin{eqnarray}\label{cup area total eta second}
    \H^n(\Om\cap\pa F)&\le&\H^n\big(\pa E\setminus B_r\big)+\H^n(\pa B_r\setminus E)
    \\\nonumber
    &&+\big(2+C(n)\,\eta\big)\,\H^n(E\cap\pa B_r\setminus A)
    +C(n)\,\eta\,\H^{n-1}\big(\pa E\cap\pa B_r\big)\,.
  \end{eqnarray}
Indeed, by $\H^n(\pa E \cap \pa B_r) = 0$, \eqref{cup fuori da br chiusa}, and \eqref{cc2_bdry}
\begin{eqnarray}
   \label{cc2_a_ext1}
\H^n(\Om \cap \pa F) & \leq& \H^n(\pa E \setminus B_r) + \H^n(\pa F \cap \cl(B_r))
\\\nonumber
&\leq& \H^n(\pa E \setminus B_r) + \H^n(\pa B_r \setminus A)+ \H^n(B_r \cap \pa F)
\\\nonumber
&\le& \H^n(\pa E \setminus B_r) + \H^n(\pa B_r\setminus E)+\H^n((E\cap\pa B_r)\setminus A)+\H^n(B_r \cap \pa F)\,;
\end{eqnarray}
by \eqref{cc2_in}, \eqref{cc2_Ybdry}, the $\H^{n-1}$-rectifiability of $\pa E \cap \pa B_r$, and the area formula
\begin{eqnarray}\label{cc2_a_ext2}
\H^n(B_r\cap\pa F)&\le&\H^n(B_r \cap \pa N_\eta(Y))
\\\nonumber
&\leq&(1 + C(n) \, \eta) \H^n(Y) + C(n) \, \eta \, \H^{n-1}(\pa E \cap \pa B_r)\,,
\end{eqnarray}
while \eqref{cc2 AAA} and $\H^n(\pa B_r\cap\pa E)=0$ give
\[
\H^n(Y)=\H^n((E\cap\pa B_r)\setminus\cl(A))=\H^n((E\cap\pa B_r)\setminus A)\,.
\]
We thus deduce \eqref{cup area total eta second}. As $\eta\to 0^+$ in \eqref{cup area total eta second} and in \eqref{cc2_a_ext2} we get \eqref{cup area totale} and \eqref{cup second br area}.
\end{proof}

In the following lemma we introduce the notion of exterior cup competitor. We set
\[
M_\eta(Y)=\Big\{y+t\,\nu_B(y):y\in Y\,,t\in(0,\eta)\Big\}\,,\qquad\eta>0\,,
\]
whenever $B$ is an open ball and $Y\subset\pa B$.

\begin{lemma}[Exterior cup competitor]
  \label{lemma cup exterior}
  Let $E\in\E$ be such that $\Om \cap \pa E$ is $\C$-spanning $W$, let $R>0$ be such that $W\cc B_R(0)$ and $\pa E\cap\pa B_R(0)$ is $\H^{n-1}$-rectifiable, and let $A$ be a connected component of $\pa B_R(0)\setminus\pa E$ such that $A\cap E=\emptyset$. For every $\eta \in \left(0,1 \right)$ there exists a set $F=F_\eta \in \E$ such that $\Om\cap\pa F$ is $\C$-spanning $W$ and
  \begin{eqnarray}
  \label{area of external cup competitors}
  \limsup_{\eta\to 0^+}\H^n(\Om\cap\pa F)\le \H^n\big(\Om\cap B_R(0)\cap\pa E)+2\,\H^n(\pa B_R(0)\setminus A)\,.
  \end{eqnarray}
\end{lemma}

\begin{proof}
The proof consists of a minor modification of step one and step two in the proof of Lemma \ref{lemma cup competitor first kind}. Precisely, the exterior cup competitor defined by $E$ and $A$ is given by
 \begin{equation}
    \label{cup exterior}
      F=\big(E\cap B_R(0)\big)\cup\,M_\eta(Z)\,,
  \end{equation}
where
\begin{eqnarray*}
Z&=&Y \cup \Big( U_\eta(S) \setminus \cl (E \cap \pa B_R(0)) \Big)\,,
\\
Y&=&\pa B_R(0)\setminus\Big(\cl(E\cap\pa B_R(0))\cup\cl(A)\Big)\,,
\\
U_\eta(S)&=&\pa B_R(0) \cap \{{\rm d}_S < \eta\}\,,
\\
S&=&\pa E \cap \cl (A) \setminus \Big( \cl (E \cap \pa B_R(0)) \cup \cl (Y)\Big)\,;
\end{eqnarray*}
   see
  \begin{figure}
  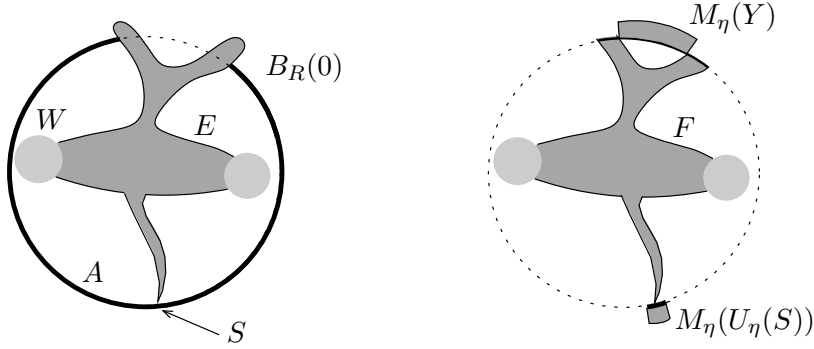\caption{{\small An exterior cup competitor. Notice that for $S$ to be non-emtpy, and non-disconnecting $A$, it must be $n\ge 2$.}}\label{fig cupext}
  \end{figure}
  Figure \ref{fig cupext}. If $\g\in\C$ is such that $\g\cap\pa E\cap \cl(B_R(0))=\emptyset$, then an adaptation of step one in the proof of Lemma \ref{lemma close by Lipschitz at boundary} shows that there exists a connected component of $\g\setminus B_R(0)$ which is diffeomorphic to an interval, and whose end-points belong to distinct connected components of $(\R^{n+1}\setminus B_R(0))\setminus\pa E$. Using this fact, and since $\pa F \cap B_R(0) = \pa E \cap B_R(0)$, we just need to show that $\pa B_R(0)\cap\pa F$ contains $\pa B_R(0) \cap \pa E$ as well as $\pa B_R(0) \setminus \cl (A)$ in order to show that $\Omega \cap \pa F$ is $\C$-spanning $W$. This is done by repeating with minor variations the considerations contained in step two of the proof of Lemma \ref{lemma cup competitor first kind}. The proof of \eqref{area of external cup competitors} is obtained in a similar way, and the details are omitted.
\end{proof}

\subsection{Slab competitors}\label{section slab} Bi-Lipschitz deformations of cup competitors can be used to generate new competitors thanks to Lemma \ref{statement spanning is close by Lipschitz maps}. We will crucially use this remark to replace balls with ``slabs'' (see Figures \ref{fig slab1}, \ref{fig slab2} and \ref{fig slab3}) and obtain sharp area concentration estimates in step five of the proof of Theorem \ref{thm lsc}, as well as in the proof of Theorem \ref{thm basic regularity}, see e.g. \eqref{the important remark citato}. Given $\tau\in(0,1)$, $x\in\R^{n+1}$, $r > 0$, and $\nu\in\SS^n$, we set
\[
S_{\tau,r}^{\,\nu}(x)=\big\{y\in B_r(x):|(y-x)\cdot\nu|<\tau\,r\big\}\,,
\]
and we claim the existence of a bi-Lipschitz map $\Phi:\R^{n+1}\to\R^{n+1}$ with
\[
  \{\Phi\ne\id\}\cc B_{2\,r}(x)\,,\qquad \Phi\big(B_{2\,r}(x)\big)= B_{2\,r}(x)\,,\qquad
  \Phi\big(\pa S_{\tau,t}^{\,\nu}(x)\big)=\pa B_t(x)\quad\forall t\in(0,r)\,,
\]
and such that $\Lip\,\Phi$ and $\Lip\,\Phi^{-1}$ depend only on $n$ and $\tau$.  Indeed, assuming without loss of generality that $x=0$, there is a convex, degree-one positively homogenous function $\vphi:\R^{n+1}\to[0,\infty)$ such that $S_{\tau,t}^\nu(0)=\{\vphi<t\}$ for every $t>0$. Taking $\eta_r:[0,\infty)\to[0,\infty)$ smooth, decreasing and such that $\eta=1$ on $[0,4r/3]$ and $\eta=0$ on $[5r/3,\infty)$, we set
\[
\Phi(x)=\eta_r(|x|)\,\frac{\vphi(x)}{|x|}\,x+(1-\eta_r(|x|))\,x\,.
\]
Noticing that $\Phi$ is a smooth interpolation between linear maps on each half-line $\{t\,x:t\ge0\}$, and observing that the slopes of these linear maps change in a Lipschitz way with respect to the angular variable, one sees that $\Phi$ has the required properties.

\begin{lemma}[Slab competitors]
  \label{lemma slab competitor}
  Let $E\in\E$ be such that $\Om\cap\pa E$ is $\C$-spanning $W$, and let $B_{2r}(x)\cc\Om$, $\nu\in\SS^n$, $\tau\in(0,1)$ with $\pa S_{\tau,r}^{\,\nu}(x)\cap\pa E$ $\H^{n-1}$-rectifiable. Let $A$ be an open connected component of $\pa S_{\tau,r}^{\,\nu}(x)\setminus\pa E$.
  Then for every $\eta\in(0,r/2)$, there exists $F\in\E$ such that $\Om\cap\pa F$ is $\C$-spanning $W$,
  \begin{eqnarray}
  \label{slab competitor exterior}
  &&    F\setminus\cl(S_{\tau,r}^{\,\nu}(x))=E\setminus\cl(S_{\tau,r}^{\,\nu}(x))\,,
      \\
  \label{slab competitor buccia}
&&\lim_{\eta \to 0^+} \H^n \left( (\pa F\cap\pa S_{\tau,r}^{\,\nu}(x)) \, \Delta \, (  \pa S_{\tau,r}^{\,\nu}(x)\setminus A   )   \right)  = 0\,,
  \end{eqnarray}
  and such that if $A\cap E=\emptyset$, then
  \begin{eqnarray}
  \label{area of slab competitors first 0}
  \limsup_{\eta\to 0^+}\H^n(S_{\tau,r}^{\,\nu}(x)\cap\pa F)
  \le\,C(n,\tau)\,\H^n\big(\pa S_{\tau,r}^{\,\nu}(x)\setminus (A\cup E)\big)\,;
  \end{eqnarray}
  while, if $A\subset E$, then
  \begin{eqnarray} \label{area of slab competitors second 0}
  \limsup_{\eta\to 0^+}\H^n(S_{\tau,r}^{\,\nu}(x)\cap\pa F)
  \le\,C(n,\tau)\,\H^n\big(E\cap\pa S_{\tau,r}^{\,\nu}(x)\setminus A\big)\,.
  \end{eqnarray}
\end{lemma}

\begin{proof}
  Let us set for brevity $S_r=S_{\tau,r}^{\,\nu}(x)$ and $B_r=B_r(x)$. By Lemma \ref{statement spanning is close by Lipschitz maps}, $\Phi(E)\in\E$ and $\Om\cap\pa\Phi(E)$ is $\C$-spanning $W$. Since $\Phi$ is an homeomorphism between $\pa S_r$ and $\pa B_r$, $\Phi(A)$ is an open connected component of $\pa B_r\setminus \pa \Phi(E)$. Depending on whether $A\cap E=\emptyset$ or $A\subset E$, and thus, respectively, depending on whether $\Phi(A)\cap\Phi(E)=\emptyset$ or $\Phi(A)\cap\Phi(E)\ne\emptyset$, we consider the cup competitor $G$ defined by $\Phi(E)$ and $\Phi(A)$, so that
  \[
  G=\big(\Phi(E)\setminus\cl(B_r)\big)\cup\,N_\eta(Z)\,,\qquad Z = Y \cup \Big( U_\eta(S) \setminus \cl (\Phi (E) \cap \pa B_r) \Big)\,,
  \]
  where
  \[
 Y=\pa B_r\setminus\big(\cl(\Phi(E)\cap\pa B_r)\cup\cl(\Phi(A))\big)\,, \qquad U_\eta(S) = \pa B_r \cap \{{\rm d}_S < \eta\}\,,
  \]
  with
  \[
  S = \pa \Phi(E) \cap \cl (\Phi (A)) \setminus \left[ \cl (\Phi (E) \cap \pa B_r) \cup \cl (Y) \right],
  \]
  if $A\cap E=\emptyset$, see \eqref{cup competitor first case difficult lemma}, and
  \[
  G=\big(\Phi(E)\cup B_r\big)\setminus\cl\big(N_\eta(Y)\big)\,,\qquad Y=\big(\Phi(E)\cap\pa B_r\big)\setminus\cl(\Phi(A))\,,
  \]
  if $A\subset E$, see \eqref{cup competitor second case lemma}. Finally, we set $F=\Phi^{-1}(G)$. Since $G\in\E$ and $\Om\cap\pa G$ is $\C$-spanning $W$, by Lemma \ref{statement spanning is close by Lipschitz maps} we find that $F\in\E$ and that $\Om\cap\pa F$ is $\C$-spanning $W$. By construction $G\setminus\cl(B_r)=\Phi(E)\setminus\cl(B_r)$, so that \eqref{slab competitor exterior} follows by
  \[
  F\setminus\cl(S_r)=\Phi^{-1}\big(G\setminus\cl(B_r)\big)=\Phi^{-1}\big(\Phi(E)\setminus\cl(B_r)\big)=E\setminus\cl(S_r)\,.
  \]
  By \eqref{cup buccia}, $\H^n \left( (\pa B_r\cap\pa G) \, \Delta \, (\pa B_r\setminus\Phi(A)) \right) \to 0$ as $\eta \to 0^+$, which gives \eqref{slab competitor buccia} by the area formula. Finally,  \eqref{area of slab competitors first 0} and \eqref{area of slab competitors second 0} are deduced by the area formula, \eqref{cup first br area} and \eqref{cup second br area}.
\end{proof}

\subsection{Cone competitors}\label{section cone}
  As customary in the analysis of area minimization problems, we want to compare $\H^n(B_r(x)\cap\pa E)$ with $\H^n(B_r(x)\cap\pa F)$, where $F$ is the cone spanned by $E\cap\pa B_r(x)$ over $x$,
  \begin{equation}
    \label{cone competitor}
      F=\big(E\setminus\cl(B_r(x))\big)\cup \big\{(1-t)\,x+t\,y:y\in E\cap \pa B_r(x)\,,t\in(0,1]\big\}\,.
  \end{equation}
  Following the terminology of \cite{DLGM}, given $K\in\Sc$, the cone competitor $K'$ of $K$ in $B_r(x)$ is similarly defined as
  \[
  K'=(K\setminus B_r(x) )\cup \big\{(1-t)\,x+t\,y:y\in K\cap\pa B_r(x)\,,t\in[0,1]\big\}\,,
  \]
  and is indeed $\C$-spanning $W$ (since $K$ was). However, for some values of $r$, $\pa F\cap B_r(x)$ may be strictly smaller than the cone competitor $K'$ defined by the choice $K=\Om\cap\pa E$ in $B_r(x)$, and thus it may fail to be $\C$-spanning; see
  \begin{figure}
     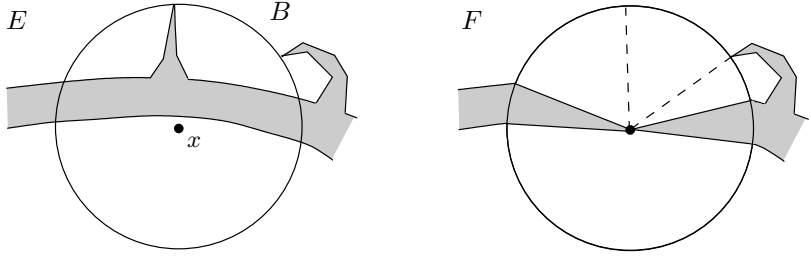\caption{{\small In this picture, the cone competitor $F$ defined by $E\cap\pa B_r$ as in \eqref{cone competitor} may fail to be $\C$-spanning $W$. Notice that the dashed lines are part of the cone competitor $K'$ defined by $K=\Om\cap\pa E$ in $B_r(x)$, which is indeed strictly larger than $\Om\cap\pa F$.
     }}\label{fig cone}
  \end{figure}
  Figure \ref{fig cone}. By Sard's lemma, if $E$ has smooth boundary in $\Om$ this issue can be avoided as, for a.e. $r$, $\pa E$ and $\pa B_r$ intersect transversally, and thus $\pa E\cap\pa B_r(x)$ is the boundary of $E\cap\pa B_r(x)$ relative to $\pa B_r(x)$; but working with smooth boundary leads to other difficulties when constructing cup competitors. We thus approximate $F$ (as defined in \eqref{cone competitor}) in energy by means of diffeomorphic images of $E$.

\begin{lemma}[Cone competitors]
  \label{lemma cone competitor}
  Let $E\in\E$ be such that $\Om\cap\pa E$ is $\C$-spanning $W$, and let $B=B_r(x)\cc\Om$ be such that $E \cap \pa B_r(x)$ is $\H^n$-rectifiable, $\pa E\cap\pa B_r(x)$ is $\H^{n-1}$-rectifiable and $r$ is a Lebesgue point of the maps $t \mapsto \H^{n}(E \cap \pa B_t(x))$ and $t \mapsto \H^{n-1}(\pa E \cap \pa B_t(x))$. Then for each $\eta\in(0,r/2)$ there exists $F\in\E$ such that $F \Delta E \subset B_r(x)$, $\Om\cap\pa F$ is $\C$-spanning $W$, and
  \begin{align}
    \label{cone competitor area inequality}
    \limsup_{\eta\to 0^+}\H^n(\Om\cap\pa F)&\le\H^n(\pa E\setminus B_r(x))+\frac{r}n\,\H^{n-1}(\pa E\cap\pa B_r(x))\,,\\
    \label{cone competitor volume inequality}
    \liminf_{\eta \to 0^+} |F| & \geq |E \setminus B_r(x)| + \frac{r}{n+1} \, \H^n(E \cap \pa B_r(x))    \,.
  \end{align}
\end{lemma}

\begin{proof}
  Let $x=0$, $r=1$, $B_r=B_r(0)$, and define a bi-Lipschitz map $f_\eta$ by $f_\eta(0)=0$ and $f_\eta(x)=u_\eta(|x|)\,\hat{x}$ if $x\ne 0$, where $\hat x = x / |x|$ and $u_\eta:\R\to[0,\infty)$ is given by
\begin{equation} \label{u_eta}
u_\eta(t) :=
\begin{cases}
\max\{0,\eta\,t\}\,, &\mbox{for $t \leq 1-\eta$}\,,
\\
\eta(1-\eta)+\frac{t-(1-\eta)}{\eta}\,\left(1-\eta(1-\eta)\right)\,, &\mbox{for $t \in [ 1-\eta,1 ]$}\,,
\\
t\,, & \mbox{for $t \geq 1$}\,,
\end{cases}
 \end{equation}
 so that $u_\eta(t)\le t$ for $t\ge0$. Clearly, $\{f_\eta\ne\id\} \subset B_1$ and $f_\eta(B_1)\subset B_1$. The open set $F=f_\eta(E)$ is such that $\Om\cap\pa F=f_\eta(\Om\cap\pa E)$, so that $\Om\cap\pa F$ is $\H^n$-rectifiable and, by Lemma \ref{statement spanning is close by Lipschitz maps},  $\C$-spanning $W$. Thanks to the area formula, \eqref{cone competitor area inequality} will follow by showing
  \begin{equation} \label{wanted estimate}
\limsup_{\eta\to 0^+}\int_{B_1\cap\pa E}J^{\pa E}f_\eta \,d\H^n\le\frac{1}{n}\,\H^{n-1}(\pa E\cap\pa B_1)\,.
  \end{equation}
Trivially, the integral over $B_{1-\eta}\cap\pa E$ is bounded by $C(n)\,\eta^n\,\H^n(\Om\cap\pa E)$. The integral over $B_1\setminus B_{1-\eta}$ is treated as in \cite[Step two, Theorem 7]{DLGM}; by the coarea formula,
\begin{equation} \label{ciccia}
\begin{split}
 \int_{(B_1 \setminus  B_{1-\eta} )\cap\pa E}J^{\pa E}f_\eta \,d\H^n = & \int_{1-\eta}^1 dt \int_{\pa B_t \cap \pa E \cap \{|\nu_E \cdot \hat x| < 1\}} \frac{J^{\pa E} f_\eta}{\sqrt{1-(\nu_E \cdot \hat x)^2}} \, d\H^{n-1} \\ &+ \int_{(B_1 \setminus B_{1-\eta}) \cap \pa E \cap \{|\nu_E \cdot \hat x| =1\}  } J^{\pa E}f_\eta \, d\H^n\,,
 \end{split}
\end{equation}
where $\nu_E(x) \in T_x(\pa E) \cap \mathbb{S}^n$ at $\H^n$-a.e. $x \in \pa E$. By
\begin{equation} \label{derivative}
\nabla f_\eta(x) = \frac{u_\eta(|x|)}{|x|} \, {\rm Id} + \left( u_\eta'(|x|) - \frac{u_\eta(|x|)}{|x|} \right) \, \hat x \otimes \hat x\,,
\end{equation}
if $|\nu_E(x) \cdot \hat x| = 1$, then $J^{\pa E}f_\eta =(u_\eta(|x|)/|x|)^n \leq   1$. Since
\begin{equation} \label{annulus shrinks}
\lim_{\eta \to 0^+} \H^n(\pa E \cap (B_1 \setminus B_{1-\eta})) = 0\,,
\end{equation}
the second term on the right-hand side of \eqref{ciccia} converges to $0$ as $\eta\to 0^+$. As for the first term, by \eqref{derivative}, we have, as explained later on,
\begin{equation}
  \label{all}
J^{\pa E}f_\eta(x) \leq 1 + \sqrt{1-(\nu_E(x) \cdot \hat x)^2} \, u_\eta'(|x|) \, \Big( \frac{u_\eta(|x|)}{|x|} \Big)^{n-1} \quad \mbox{for $\H^n$-a.e. $x \in \pa E$}\,.
\end{equation}
The term corresponding to $1$ in \eqref{all} converges to $0$ as $\eta\to 0^+$ by \eqref{annulus shrinks}. At the same time,
\[
\limsup_{\eta\to 0^+}\Big|\int_{1-\eta}^{1} \Big(\H^{n-1}(\pa E \cap \pa B_t)-\H^{n-1}(\pa E \cap \pa B_1)\Big) \, u_\eta' \, \Big( \frac{u_\eta}{t} \Big)^{n-1} \, dt\Big|=0
\]
since $t=1$ is a Lebesgue point of $t\mapsto\H^{n-1}(\pa B_t\cap\pa E)$, and since $u'_\eta(t)\le 1/\eta$ and $(u_\eta(t)/t)\leq 1$ for $t\ge0$. Finally,
\begin{eqnarray*}
  \int_{1-\eta}^{1} u_\eta' \, \Big( \frac{u_\eta}{t} \Big)^{n-1} \, dt\le\frac1{(1-\eta)^{n-1}}\,\frac{u_\eta(1)^n-u_\eta(1-\eta)^n}n=\frac1{(1-\eta)^{n-1}}\,\frac{1-\eta^n(1-\eta)^n}n\to \frac1n
\end{eqnarray*}
as $\eta\to 0^+$, thus completing the proof of \eqref{cone competitor area inequality}. The proof of \eqref{cone competitor volume inequality} follows an analogous argument. The goal is to show that
\begin{equation} \label{volume wanted estimate}
\liminf_{\eta \to 0^+} \int_{E \cap B_1} Jf_\eta \, dx \geq \frac{1}{n+1} \, \H^n(E \cap \pa B_1)\,,
\end{equation}
and by the coarea formula and \eqref{derivative} it is immediate to see that
\[
\int_{E \cap B_1} Jf_\eta \, dx \geq \int_{1-\eta}^1 u_\eta'(t) \, \left( \frac{u_\eta(t)}{t} \right)^n \, \H^n(E \cap \pa B_t) \, dt\,.
\]
The estimate in \eqref{volume wanted estimate} then readily follows using that $t=1$ is a Lebesgue point for the map $t \mapsto \H^n(E \cap \pa B_t)$, together with
\[
\int_{1-\eta}^1 u_\eta'(t) \, \left( \frac{u_\eta(t)}{t} \right)^n \, dt \geq \frac{1 -\eta^{n+1}(1-\eta)^{n+1} }{n+1} \to \frac{1}{n+1} \qquad \mbox{as $\eta \to 0^+$}\,.
\]
We finally explain how to deduce \eqref{all} from \eqref{derivative}. For $x\in\pa^*E$, let $\{\tau_i\}_{i=1}^n$ be an orthonormal basis of $T_x\pa^*E$ such that $\{\tau_i\}_{i=1}^{n-1}\subset x^\perp$. In this way, we can take
\[
\tau_n=\frac{\hat x-(\hat{x}\cdot\nu_E(x))\nu_E(x)}{\sqrt{1-(\hat{x}\cdot\nu_E(x))^2}}\,,
\]
and therefore compute by \eqref{derivative} that
\begin{eqnarray*}
\nabla^{\pa E}f_\eta(x)[\tau_i]&=& \frac{u_\eta(|x|)}{|x|}\,\tau_i\,,\qquad\forall i=1,...,n-1\,,
\\
\nabla^{\pa E}f_\eta(x)[\tau_n]&=&u_\eta'(|x|)\,\sqrt{1-(\hat{x}\cdot\nu_E)^2}\,\hat{x}-\frac{u_\eta(|x|)}{|x|}\,(\hat{x}\cdot\nu_E)\,
\frac{\nu_E-(\hat{x}\cdot\nu_E)\hat{x}}{\sqrt{1-(\hat{x}\cdot\nu_E)^2}}\,,
\end{eqnarray*}
where we have set for brevity $\nu_E$ in place of $\nu_E(x)$. Therefore
\begin{eqnarray*}
  J^{\pa E}f(x)^2&=&\Big|\bigwedge_{i=1}^n\nabla^{\pa E}f_\eta(x)[\tau_i]\Big|^2
  \\
  &=&\Big(\frac{u_\eta(|x|)}{|x|}\Big)^{2n}\,(\hat{x}\cdot\nu_E)^2\,\Big|\tau_1\wedge\cdots\wedge\tau_{n-1}\wedge\Big(\frac{\nu_E-(\hat{x}\cdot\nu_E)\hat{x}}{\sqrt{1-(\hat{x}\cdot\nu_E)^2}}\Big)\Big|^2
  \\
  &&+\Big(\frac{u_\eta(|x|)}{|x|}\Big)^{2(n-1)}\,u_\eta'(|x|)^2\,\big(1-(\hat{x}\cdot\nu_E)^2\big)\,
  \Big|\tau_1\wedge\cdots\wedge\tau_{n-1}\wedge\hat{x}\Big|^2
  \\
  &\le&1+\Big(\frac{u_\eta(|x|)}{|x|}\Big)^{2(n-1)}\,u_\eta'(|x|)^2\,\big(1-(\hat{x}\cdot\nu_E)^2\big)\,,
\end{eqnarray*}
from which \eqref{all} follows thanks to $\sqrt{1+a}\le1+\sqrt{a}$ for $a\ge 0$.
 \end{proof}

\subsection{Nucleation lemma}\label{section nucleation} The following nucleation lemma can be found, with slightly different statements, in \cite[VI(13)]{Almgren76} or in \cite[Lemma 29.10]{maggiBOOK}.

\begin{lemma}
   \label{statement nucleation} Let $\xi(n)$ be the constant of Besicovitch's covering theorem in $\R^{n+1}$. If $T$ is closed, $A=\R^{n+1}\setminus T$, $0<|E|<\infty$, $P(E;A)<\infty$, $\tau>0$, and
  \[
  \s=\min\Big\{\frac{|E\setminus I_\tau(T)|}{\tau\,P(E;A)},\frac{\xi(n)}{n+1}\Big\}>0
  \]
  then there exists $x\in E^{(1)}\setminus I_\tau(T)$ such that
  \[
  |E\cap B_\tau(x)|\ge\Big(\frac{\s}{2\xi(n)}\Big)^{n+1}\,\tau^{n+1}\,.
  \]
\end{lemma}

\begin{proof}
  By contradiction one assumes that
  \begin{equation}
    \label{nucl 1}
      |E\cap B_\tau(x)|<\Big(\frac{\s}{2\xi(n)}\Big)^{n+1}\,\tau^{n+1}\qquad\forall x\in E^{(1)}\setminus I_\tau(T)\,.
  \end{equation}
  Setting $\a=\xi(n)/\s$, so that $\a\ge n+1$, we claim that \eqref{nucl 1} implies the existence, for each $x\in E^{(1)}\setminus I_\tau(T)$, of $\tau_x\in(0,\tau)$ such that
  \begin{equation}
    \label{nucl 2}
    P(E;B_{\tau_x}(x))>\frac\a\tau\,|E\cap B_{\tau_x}(x)|\,.
  \end{equation}
  In turn \eqref{nucl 2} is in contradiction with \eqref{nucl 1}: indeed, by applying Besicovitch's theorem to $\{\cl(B_{\tau_x}(x)):x\in E^{(1)}\setminus I_\tau(T)\}$ we find an at most countable subset $I$ of $E^{(1)}\setminus I_\tau(T)$ such that $\{\cl(B_{\tau_x}(x))\}_{x\in I}$ is disjoint and
  \begin{eqnarray*}
  |E\setminus I_\tau(T)|&\le&\xi(n)\,\sum_{x\in I}|E\cap B_{\tau_x}(x)|
  <\frac{\xi(n)\,\tau}\a\,\sum_{x\in I}P(E;B_{\tau_x}(x))
  \\
  &\le&\frac{\xi(n)\,\tau\,P(E;A)}\a
  =\tau\,\s\,P(E;A)\le|E\setminus I_\tau(T)|\,,
  \end{eqnarray*}
  a contradiction. We show that \eqref{nucl 1} implies \eqref{nucl 2}: indeed, if \eqref{nucl 1} holds but \eqref{nucl 2} fails, then there exists $x\in E^{(1)}\setminus I_\tau(T)$ such that, setting $m(r)=|E\cap B_r(x)|$ for $r>0$,
  \begin{equation}
    \label{nucl 3}
      \mbox{$m>0$ on $(0,\infty)$}\,,\qquad m(\tau)<\Big(\frac{\tau}{2\a}\Big)^{n+1}
  \end{equation}
  and $(\a/\tau)\,m(r)\ge P(E;B_r(x))$ for every $r\in(0,\tau)$. Adding up $\H^n(\pa B_r(x)\cap E)$, which equals $m'(r)$ for a.e. $r>0$ by the coarea formula, we obtain
  \begin{equation}
    \label{nucl 4}
      m'(r)+\frac\a\tau\,m(r)\ge P(E\cap B_r(x))\ge m(r)^{n/(n+1)}\,,\qquad\mbox{for a.e. $r\in(0,\tau)$}\,.
  \end{equation}
  where in the last inequality we have used that $P(F)\ge |F|^{n/(n+1)}$ whenever $0<|F|<\infty$; see e.g. \cite[Proposition 12.35]{maggiBOOK}.  Since $m>0$ on $(0,\infty)$ we find
  \begin{equation}\nonumber
  \left\{
  \begin{split}
  &\frac\a\tau m(r)\le (1/2) m(r)^{n/(n+1)}
  \\
  &\forall r\in(0,\tau)
  \end{split}
  \right .
  \qquad\mbox{iff}\quad
  \left\{
  \begin{split}
  &m(r)\le (\tau/2\a)^{n+1}
  \\
  &\forall r\in(0,\tau)
  \end{split}
  \right .
  \qquad\mbox{if}\quad m(\tau)\le\Big(\frac{\tau}{2\a}\Big)^{n+1}\,,
  \end{equation}
  where the last condition holds by \eqref{nucl 3}. Thus \eqref{nucl 4} gives $m'(r)\ge(1/2) m(r)^{n/(n+1)}$ for a.e. $r\in(0,\tau)$, thus
  $m(\tau)\ge(\tau/2(n+1))^{n+1}\ge(\tau/2\a)^{n+1}$ as $\a\ge n+1$, a contradiction.
\end{proof}

\subsection{Isoperimetry, lower bounds and collapsing}\label{section spherical collapsing} Given an $L^1$-converging sequence of sets of finite perimeter $\{E_j\}_j$, the boundary of the $L^1$-limit set $E$ will be (in general) strictly included in $K=\spt\,\mu$, where $\mu$ is the weak-star limit of the Radon measures defined by the boundaries of the $E_j$'s. In the next lemma we show that, under some mild bounds on $\mu$ and $E_j$, if $\mu$ is absolutely continuous with respect to $\H^n \llcorner K$ then the Radon-Nikod\'ym density $\theta$ of $\mu$ is everywhere larger than $1$, and is actually larger than $2$ at a.e. point of $K\setminus\pa^*E$ (that is, a cancellation can happen only when boundaries are collapsing).

\begin{lemma}[Collapsing lemma]\label{lemma llb}
Let $K$ be a relatively compact and $\H^n$-rectifiable set in $\Om$, let $E\subset\Om$ be a set of finite perimeter with $\Om\cap\pa^*E\subset K$, and let $\{E_j\}_j\subset\E$ such that $E_j\to E$ in $L^1_{{\rm loc}}(\Om)$, and $\mu_j\weakstar\mu$ as Radon measures in $\Om$, where $\mu_j=\H^n\llcorner(\Om\cap\pa E_j)$ and $\mu=\theta\,\H^n\llcorner K$ for a Borel function $\theta$. If $\Om'\subset\Om$ and $r_*>0$ are such that for every $x\in K\cap\Om'$ and a.e. $r<r_*$ with $B_r(x)\cc\Om'$ we have
\begin{eqnarray}
  \label{llb1}
  \mu(B_r(x))&\ge&c(n)\,r^n\,,
  \\
  \label{llb2}
  \liminf_{j\to\infty}\H^n(B_r(x)\cap\pa E_j)&\le& C(n)\,\liminf_{j\to\infty}\,\H^n(\pa B_r(x)\setminus A_{r,j}^0)\,,
\end{eqnarray}
where $A_{r,j}^0$ denotes an $\H^n$-maximal connected component of $\pa B_r(x)\setminus\pa E_j$, then $\theta(x)\ge 1$ for $\H^n$-a.e. $x\in K\cap\Om'$, and $\theta(x)\ge 2$ for $\H^n$-a.e. $x\in (K\setminus\pa^*E)\cap\Om'$.
\end{lemma}

The bound $\theta\ge 1$ follows by arguing exactly as in \cite[Proof of Theorem 2, Step three]{DLGM}, and has nothing to do with the fact that the measures $\mu_j$ are defined by boundaries; the latter information is in turn crucial in obtaining the bound $\theta\ge 2$, and requires a new argument. For the sake of clarity, we also give the details of the $\theta\ge 1$ bound, which in turn is based on spherical isoperimetry.

\begin{lemma}[Spherical isoperimetry]
  \label{statement isoperimetry on spheres} Let $\Sigma\subset\R^{n+1}$ denote a spherical cap\footnote{That is, $\Sigma=\SS^n\cap H$ where $H$ is an open half-space of $\R^{n+1}$.} in the $n$-dimensional unit sphere $\SS^n$, possibly with $\Sigma=\SS^n$. If $K$ is a compact set in $\R^{n+1}$ and $\{A^h\}_{h=0}^\infty$ is the family of the open connected components of $\Sigma\setminus K$, ordered so to have $\H^n(A^h)\ge\H^n(A^{h+1})$, then
  \begin{equation}
    \label{spherical isoperimetry}
    \H^n(\Sigma\setminus A^0)\le C(n)\,\H^{n-1}(\Sigma\cap K)^{ n/(n-1)}\,.
  \end{equation}
  Moreover, if $\Sigma=\SS^n$, $\s_n=\H^n(\SS^n)$ and $\H^{n-1}(\SS^n\cap K)<\infty$, then each $A^h$ is a set of finite perimeter in $\SS^n$ and for every $\tau>0$ there exists $\s>0$ such that
  \begin{equation}
    \label{quant isop hp}
      \min\Big\{\H^n(A^0),\H^n(A^1)\Big\}=\H^n(A^1)\ge\frac{\s_n}2-\s
  \end{equation}
  implies
  \begin{equation}
    \label{quant isop tesis}
      \min\Big\{\H^{n-1}(\pa^* A^0),\H^{n-1}(\pa^* A^1)\Big\}\ge\s_{n-1}-\tau\,.
  \end{equation}
  Here $\pa^*A^h$ denotes the reduced boundary of $A^h$ in $\SS^n$.
\end{lemma}

\begin{proof}
  This is \cite[Lemma 9]{DLGM}. However, \eqref{quant isop tesis} is stated in a weaker form in \cite[Lemma 9]{DLGM}, so we give the details. Arguing by contradiction, we can find $\tau>0$ and $\{K_j\}_j$ such that, for $\a=0,1$, $\H^{n-1}(\pa^*A^\a_j)\le\s_{n-1}-\tau$ for every $j$, but $\H^n(A_j^\a)\to\s_n/2$ as $j\to\infty$. Since $\s_n=\H^n(\SS^n)$ and $A_j^0\cap A_j^1=\emptyset$, we find that, for $\a=0,1$, $A_j^\a\to A^\a$ in $L^1(\SS^n)$ where $A^0\cap A^1=\emptyset$ and $A^0\cup A^1$ is $\H^n$-equivalent to $\SS^n$. Therefore $\H^{n-1}(\pa^*A^0)=\H^{n-1}(\pa^*A^1)\le\s_{n-1}-\tau$, where we have used lower semicontinuity of perimeter. Since $\inf\,\H^{n-1}(\pa^*A)$ with $\H^n(A)=\s_n/2$ is equal to $\s_{n-1}$ we have reached a contradiction.
\end{proof}

\begin{proof}[Proof of Lemma \ref{lemma llb}] {\it Step one}: We fix $x\in K\cap\Om'$ such that $\H^n\llcorner(K-x)/r\weakstar \H^n\llcorner T_xK$ as $r\to 0^+$. Setting $\nu(x)^\perp=T_xK$ for $\nu(x)\in\SS^n$, by the lower density estimate \eqref{llb1} we easily find that for every $\s>0$ there exists $r_0=r_0(\s,x)\in(0,\min\{r_*,\dist(x,\pa\Om')\})$ such that $|(y-x)\cdot\nu(x)|<\s\,r$ for every $y\in K\cap B_r(x)$ and every $r<r_0$. In particular,
\[
\lim_{j \to \infty} \H^n(\pa E_j \cap \{ y \in B_r(x) \, \colon \, |(y-x) \cdot \nu(x)| > \s \, r\}) = 0 \qquad \mbox{for every $r \leq r_0$}\,,
\]
and thus by the coarea formula (see \cite[Equation (2.13)]{DLGM})
\begin{equation}
  \label{from dlgm}
  \lim_{j\to\infty}\H^{n-1}(\Sigma_{r,\s}^\pm\cap \pa E_j)=0\qquad\mbox{for a.e. $r\le r_0$}\,,
\end{equation}
where we have set
\begin{eqnarray*}
\Sigma_{r,\s}^+&=&\big\{y\in \pa B_r(x):(y-x)\cdot\nu(x)>\s\,r\big\}\,,
\\
\Sigma_{r,\s}^-&=&\big\{y\in \pa B_r(x):(y-x)\cdot\nu(x)<-\s\,r\big\}\,.
\end{eqnarray*}
Let $A_{r,j}^+$ be an $\H^n$-maximal connected component of $\Sigma_{r,\s}^+\setminus \pa E_j$, and define similarly $A_{r,j}^-$. Equations \eqref{from dlgm} and \eqref{spherical isoperimetry} imply that, for a.e. $r<r_0$,
\begin{equation}
  \label{from dlgm2}
  \lim_{j\to\infty}\H^n(A_{r,j}^\pm)=\H^n(\Sigma_{r,\s}^\pm)\,.
\end{equation}
Now let $\{A^h_{r,j}\}_{h=0}^\infty$ denote the open connected components of $\pa B_r(x)\setminus\pa E_j$, ordered by decreasing $\H^n$-measure. We claim that
\begin{equation}
  \label{llb3}
  \mbox{if \eqref{from dlgm2} holds, then either $A_{r,j}^+$ or $A_{r,j}^-$ is not contained in $A_{r,j}^0$}\,.
\end{equation}
Indeed, if for some $r$ we have $A_{r,j}^+\cup A_{r,j}^-\subset A_{r,j}^0$, then by \eqref{llb2} and \eqref{from dlgm2} we find
\begin{equation}
  \label{ias later}
  \mu(B_r(x))\le\liminf_{j\to\infty}\mu_j(B_r(x))\le C(n)\,\liminf_{j\to\infty}\H^n(\pa B_r\setminus A_{r,j}^0)\le C(n)\,r^n\,\s\,,
\end{equation}
a contradiction to \eqref{llb1} if $\s\le\s_0(n)$ for a suitable $\s_0(n)$. By \eqref{llb3} and \eqref{from dlgm2},
\begin{equation}
  \label{density 1 1}
  \min\Big\{\H^n(A^0_{r,j}),\H^n(A^1_{r,j})\Big\}\ge \Big(\frac{\sigma_n}2-C(n)\,\sigma\Big)\,r^n\qquad\mbox{for a.e. $r<r_0$}\,.
\end{equation}
By Lemma \ref{statement isoperimetry on spheres} and \eqref{density 1 1}, given $\tau>0$, if $\s$ is small enough in terms of $n$ and $\tau$, then
\begin{equation}
  \label{density 1 2}
  \min\Big\{\H^{n-1}(\pa^*A^0_{r,j}),\H^{n-1}(\pa^*A^1_{r,j})\Big\}\ge \big(\sigma_{n-1}-\tau\big)\,r^{n-1}\qquad\mbox{for a.e. $r<r_0$}\,,
\end{equation}
where $\pa^*A^{\a}_{r,j}$ is the reduced boundary of $A^{\a}_{r,j}$ as a subset of $\pa B_r(x)$. Since $A_{r,j}^0$ is a connected component of $\pa B_r(x)\setminus \pa E_j$ we have
\begin{equation}
  \label{llb4}
  \big(\sigma_{n-1}-\tau\big)\,r^{n-1}\le\H^{n-1}(\pa^*A^0_{r,j})\le \H^{n-1}(\pa B_r(x)\cap \pa E_j)\,.
\end{equation}
Now if $f_j(r)=\mu_j(B_r(x))$ and $f(r)=\mu(B_r(x))$ then by the coarea formula we easily find that $f_j\to f$ a.e. with $\liminf_{j\to\infty}f_j'(r)\le f'(r)\le Df$, where $Df$ denotes the distributional derivative of $f$. Hence, letting $j\to\infty$ and $\tau\to 0^+$ in \eqref{llb4} we obtain $Df\ge \sigma_{n-1}\,r^{n-1}\,dr$ on $(0,r_0)$. As $\om_n=n\,\sigma_{n-1}$, we conclude that $\theta(x)\ge 1$. We stress once more that so far we have just followed the argument of \cite[Proof of Theorem 2, Step three]{DLGM}.

\medskip

\noindent {\it Step two}: We use the boundary structure to show that $\theta\ge2$ $\H^n$-a.e. on $\Om'\cap(K\setminus\pa^*E)$. Since $\{E^{(0)}\,,E^{(1)}\,,\pa^*E\}$ is an $\H^n$-a.e. partition of $\R^{n+1}$, we can assume that $x\in (E^{(0)}\cup E^{(1)})\cap K\cap\Om'$.
We consider first the case $x\in E^{(0)}$. Given $\s>0$, up to decreasing $r_0$,
\begin{equation}
  \label{density 2 1}
  \s\, r_0^{n+1}\ge\lim_{j\to\infty}|E_j\cap B_{r_0}(x)|=\lim_{j\to\infty}\int_0^{r_0}\,\H^n(E_j\cap\pa B_r(x))\,dr\,.
\end{equation}
Let us consider the measurable set $I_j\subset(0, r_0)$
\[
I_j=\big\{r\in(0,r_0):A_{r,j}^0\cup A_{r,j}^1\subset\pa B_r(x)\setminus\cl(E_j)\big\}\,.
\]
We claim that
\begin{equation}
  \label{density 2 2}
\H^{n-1}(\pa^*A_{r,j}^0\cap \pa^*A_{r,j}^1)=0\qquad\forall r\in I_j\,.
\end{equation}
Indeed, if $r\in I_j$, then $A_{r,j}^0$, $A_{r,j}^1$, and $\pa B_r(x)\cap E_j$ are disjoint sets of finite perimeter in $\pa B_r(x)$, and in particular
\begin{eqnarray*}
  \nu_{A_{r,j}^0}&=&-\nu_{A_{r,j}^1}\,,\hspace{0.8cm}\qquad\mbox{$\H^{n-1}$-a.e. on $\pa^*A_{r,j}^0\cap\pa^*A_{r,j}^1$}\,,
  \\
  \nu_{A_{r,j}^0}&=&-\nu_{\pa B_r(x)\cap E_j}\,\qquad\mbox{$\H^{n-1}$-a.e. on $\pa^*A_{r,j}^0\cap\pa^*[\pa B_r(x)\cap E_j]$}\,,
  \\
  \nu_{A_{r,j}^1}&=&-\nu_{\pa B_r(x)\cap E_j}\,\qquad\mbox{$\H^{n-1}$-a.e. on $\pa^*A_{r,j}^1\cap\pa^*[\pa B_r(x)\cap E_j]$}\,.
\end{eqnarray*}
At the same time, since $\{A_{r,j}^h\}_{h=0}^\infty$ are connected components of $\pa B_r(x)\setminus \pa E_j$,
\[
\pa^*A_{r,j}^h\subset\pa^*[\pa B_r(x)\cap E_j]\qquad\mbox{modulo $\H^n$}
\]
and thus $\H^{n-1}$-a.e. on $\pa^*A_{r,j}^0\cap\pa^*A_{r,j}^1$ we have
\[
\nu_{\pa B_r(x)\cap E_j}=-\nu_{A_{r,j}^0}=\nu_{A_{r,j}^1}=-\nu_{\pa B_r(x)\cap E_j}
\]
a contradiction. By \eqref{density 1 2} and \eqref{density 2 2}, given $\tau>0$ and provided $\s$ is small enough in terms of $n$ and $\tau$, for a.e. $r\in I_j$ we find
\begin{eqnarray*}
  f_j'(r)&\ge&\H^{n-1}(\pa B_r(x)\cap \pa E_j)\ge\H^{n-1}(\pa^*A_{r,j}^0\cup\pa^*A_{r,j}^1)
  \\
  &=&\H^{n-1}(\pa^*A_{r,j}^0)+\H^{n-1}(\pa^*A_{r,j}^1)
  \ge2\,\big(\sigma_{n-1}-\tau\big)\,r^{n-1}\,.
\end{eqnarray*}
Hence,
\begin{eqnarray}\nonumber
  f_j( r_0)&\ge& 2\,\big(\sigma_{n-1}-\tau\big)\,\frac{ r_0^n}n-C(n)\int_{(0, r_0)\setminus I_j}r^{n-1}\,dr
  \\\label{density 2 3}
  &\ge&2\,\big(\sigma_{n-1}-\tau\big)\,\frac{ r_0^n}n-C(n)\, r_0^{1/n}\,\Big(\int_{(0, r_0)\setminus I_j}r^n\,dr\Big)^{(n-1)/n}\,.
\end{eqnarray}
We notice that for a.e. $r \in \left( 0, r_0 \right)  \setminus I_j$, \eqref{density 1 1} gives
\[
\H^n(E_j\cap\pa B_r(x))\ge\min\Big\{\H^n(A^0_{r,j}),\H^n(A^1_{r,j})\Big\}\ge \Big(\frac{\sigma_n}2-C(n)\,\s\Big)\,r^n\,,
\]
so that \eqref{density 2 1} implies
\begin{equation}
  \label{density 2 4}
  \s\, r_0^{n+1}\ge c(n)\,\limsup_{j\to\infty}\int_{(0, r_0)\setminus I_j}\,r^n\,dr\,.
\end{equation}
If we combine \eqref{density 2 3} and \eqref{density 2 4} and let $j\to\infty$, then we find
\[
f(r_0)=\lim_{j\to\infty}f_j( r_0)\ge 2\,\big(\sigma_{n-1}-\tau\big)\,\frac{ r_0^n}n-C(n)\, r_0^{1/n}\,\Big(\s\, r_0^{n+1}\Big)^{(n-1)/n}
\]
Dividing by $ r_0^n$ and letting $ r_0\to 0^+$, $\s\to 0^+$ and $\tau\to 0^+$ we find $\theta(x)\ge 2$ whenever $x\in E^{(0)}\cap K\cap\Om'$. The case when $x\in E^{(1)}$ is analogous and the details are omitted.
\end{proof}

\section{Existence of generalized minimizers: Proof of Theorem \ref{thm lsc}}\label{section existence of generalized minimizers} Given the length of the proof, we provide a short overview. In step one, we check that $\psi(\e)<\infty$ by using the open neighborhoods of a minimizer $S$ of $\ell$ as comparison sets for $\psi(\e)$. We remark that this is the only point of the proof where \eqref{hp on W and C} is used. It is important here to allow for sufficiently non-smooth sets in the competition class $\E$: indeed, minimizers of $\ell$ are known to be smooth only outside of a close $\H^n$-negligible set in arbitrary dimension. Once $\psi(\e)<\infty$ is established, we consider a minimizing sequence $\{E_j\}_j$ for $\psi(\e)$, so that $E_j\in\E$, $|E_j|=\e$, $\Om\cap\pa E_j$ is $\C$-spanning $W$ and
\begin{equation}
\label{minimizing sequence}
\H^n(\Om\cap\pa E_j)\le \H^n(\Om\cap\pa F)+\frac1j\qquad\forall F\in\E\,,\,|F|=\e\,,\,\mbox{$\Om\cap\pa F$ is $\C$-spanning $W$}\,.
\end{equation}
We want to apply \eqref{minimizing sequence} to the comparison sets constructed in section \ref{section five competitors}, but, in general, those local variations do not preserve the volume constraint. A family of volume-fixing variations acting uniformly on $\{E_j\}_j$ is constructed through the nucleation lemma (Lemma \ref{statement nucleation}) following some ideas introduced by Almgren in the existence theory of minimizing clusters \cite{Almgren76}; see steps two and three. In step four we exploit cup and cone competitors to show that, up to extracting subsequences, $\H^n\llcorner(\Om\cap\pa E_j)\weakstar\mu=\theta\,\H^n\llcorner K$ as Radon measures in $\Om$, and $E_j\to E$ in $L^1_{{\rm loc}}(\Om)$, for a pair $(K,E)\in\KK$ and for an upper semicontinuous function $\theta\ge1$ on $K$. An application of Lemma \ref{lemma llb} shows that $\theta\ge 2$ $\H^n$-a.e. on $K\setminus\pa^*E$, thus proving $\psi(\e)\ge\F(K,E)$. In order to show that $\psi(\e)=\F(K,E)$, and thus that $(K,E)$ is a generalized minimizer of $\psi(\e)$, we need to exclude that $\Om\cap\pa E_j$ concentrates area by folding against $K$, at infinity, or against the wire frame. By using slab competitors we prove that $\Om\cap\pa E_j$, in its convergence towards $K$, cannot fold at all near points in $\pa^*E$, and can fold at most twice near points in $K\cap(E^{(0)}\cup E^{(1)})$ (step five). In step six, concentration of area at the boundary is ruled out by a deformation argument based on Lemma \ref{lemma close by Lipschitz at boundary}. Finally, in step seven, we exclude area (and volume) concentration at infinity by using exterior cup competitors to construct a uniformly bounded minimizing sequence.

\begin{proof}[Proof of Theorem \ref{thm lsc}] \noindent {\it Step one}: We show that
\begin{equation}
  \label{psi eps basic bounds}
  \psi(\e)\le 2\,\ell+C(n)\,\e^{n/(n+1)}\qquad\forall\e>0\,.
\end{equation}
Let $S$ be a minimizer of $\ell$, and let $\eta_0>0$ be such that \eqref{hp on W and C} holds. If $\eta\in(0,\eta_0)$, then the open $\eta$-neighborhood $U_\eta(S)$ of $S$ is such that $\Om\cap\pa U_\eta(S)$ is $\C$-spanning $W$: otherwise we could find $\eta\in(0,\eta_0)$ and $\g\in\C$ such that $\g\cap \pa U_\eta(S)=\emptyset$. Since $\g$ is connected, we would either have $\g\subset\{x:\dist(x,S)>\eta\}$, against the fact that $S$ is $\C$-spanning; or we would have $\g\subset U_{\eta}(S)$, against \eqref{hp on W and C}. Hence $\Om\cap\pa U_\eta(S)$ is $\C$-spanning $W$.

As proved in \cite{DLGM}, $S$ is $\H^n$-rectifiable. Moreover, as shown in Theorem \ref{thm density for S fine} in the appendix, we have
\begin{equation}
  \label{mememe}
  \H^n(S\cap B_r(x))\ge c(n)\,r^n\qquad\forall x\in\cl(S)\,,r<\rho_0
\end{equation}
where $\rho_0$ depends on $W$, so that $\H^n(S)<\infty$ implies that $\cl(S)$ is compact. This density estimate has two more consequences: first, combined with \cite[Corollary 6.5]{maggiBOOK}, it implies $\H^n(\cl(S)\setminus S)=0$; second, it allows us to exploit \cite[Theorem 2.104]{AFP} to find
\begin{equation}
  \label{mink step1}
  |U_\eta(S)|=2\,\eta\,\H^n(\cl(S))+{\rm o}(\eta)=2\,\eta\,\H^n(S)+{\rm o}(\eta)\qquad\mbox{as $\eta\to 0^+$}\,.
\end{equation}
By the coarea formula for Lipschitz maps applied to the distance function from $S$, see \cite[Theorem 18.1, Remark 18.2]{maggiBOOK}, we have
\[
|U_{\eta}(S)\cap A|=\int_0^\eta\,P(U_t(S);A)\,dt=\int_0^\eta\,\H^n(A\cap\pa U_t(S))\,dt\,,\qquad\mbox{$\forall A\subset\R^{n+1}$ open}\,,
\]
so that $U_\eta(S)$ is a set of finite perimeter in $\R^{n+1}$ and $\H^n(\pa U_\eta(S)\setminus\pa^*U_\eta(S))=0$ for a.e. $\eta>0$. Summarizing, we have proved that, for a.e. $\eta\in(0,\eta_0)$,
\[
F_\eta=\Om\cap U_\eta(S)\in\E\,,\qquad\mbox{$\Om\cap\cl(\pa^*F_\eta)=\Om\cap\pa F_\eta$ is $\C$-spanning $W$}\,,
\]
and, by \eqref{mink step1},
\[
f(\eta)=|F_\eta|=\int_0^\eta P(F_t;\Om)\,dt=\int_0^\eta\,P(U_t(S);\Om)\,dt\le 2\,\eta\,\H^n(S)+{\rm o}(\eta)\,.
\]
Notice that $f(s)$ is absolutely continuous with $f(\eta)=\int_0^\eta\,f'(t)\,dt$ and $f'(t)=P(F_t;\Om)$ for a.e. $t\in(0,\eta)$. Hence, for every $\eta>0$ there exist $t_1(\eta),t_2(\eta)\in(0,\eta)$ such that $f'(t_1(\eta))\le f(\eta)/\eta\,\le f'(t_2(\eta))$. Setting $F_j=F_{t_1(\eta_j)}$ for a suitable $\eta_j\to 0^+$, we get
\[
  \limsup_{j\to\infty}P(F_j;\Om)\le2\,\ell\,,
\]
where $|F_j|\to 0^+$. Finally, given $\e>0$, we pick $j$ such that $|F_j|<\e$, and construct a competitor for $\psi(\e)$ by adding to $F_j$ a disjoint ball of volume $\e-|F_j|$. In this way, $\psi(\e)\le P(F_j;\Om)+C(n)\,\big(\e-|F_j|\big)^{n/(n+1)}$, and \eqref{psi eps basic bounds} is found by letting $j\to\infty$.

\medskip

Since $\psi(\e) < \infty$, we can now consider a minimizing sequence $\{E_j\}_{j=1}^{\infty}$ for $\psi(\e)$. Given that $P(E_j)\le\H^n(\pa\Om)+\H^n(\Om\cap\pa E_j)\le\H^n(\pa\Om)+\psi(\e)+1$ for $j$ large, and that $|E_j| = \e$ for every $j$, there exist a set of finite perimeter $E\subset\Om$ and a Radon measure $\mu$ in $\Om$ such that, up to extracting subsequences,
\begin{eqnarray}\label{almostthere}
  E_j\to E\quad\mbox{in $L^1_{{\rm loc}}(\Om)$}\,,\quad
  \mu_j=\H^n\llcorner(\Om\cap\pa E_j)\weakstar\mu\quad\mbox{as Radon measures on $\Om$}\,,
\end{eqnarray}
as $j\to\infty$, see e.g. \cite[Section 12.4]{maggiBOOK}. We consider the set, relatively closed in $\Om$, defined by
\[
K=\Om \cap \spt\mu=\big\{x\in\Om:\mu(B_r(x))>0\quad\forall r>0\big\}\,,
\]
and claim that
\begin{eqnarray}\label{proof that KE belongs to KK 1}
  \mbox{$K$ is $\C$-spanning $W$}\,,\qquad \Om\cap\pa^*E\subset K\,.
\end{eqnarray}
Indeed, the first claim in \eqref{proof that KE belongs to KK 1} is obtained by applying Lemma \ref{statement K spans} to $K_j=\Om\cap\pa E_j$; and if $x\in\Om\cap \pa^*E$ and $B_r(x)\subset\Om$, then
\[
0<P(E;B_r(x))\le\liminf_{j\to\infty}P(E_j;B_r(x))\le\liminf_{j\to\infty}\mu_j(B_r(x))\le\mu(\cl(B_{r}(x)))
\]
so that $x\in K$. Notice that, at this stage, we still do not know if $(K,E)\in\KK$: we still need to show that $K$ is $\H^n$-rectifiable and, possibly up to Lebesgue negligible modifications, that $E$ is open with $\Om\cap\cl(\pa^*E)=\Om\cap\pa E$. Moreover, we just have $|E|\le\e$ (possible volume loss at infinity), and we know nothing about the structure of $\mu$.

\bigskip

\noindent {\it Step two}: We show the existence of $\tau>0$ such that for every $E_j$ there exist $x_j^1,x_j^2\in\R^{n+1}$ such that $\{\cl(B_{2\tau}(x_j^1)),\cl(B_{2\tau}(x_j^2)),W\}$ is disjoint and
\begin{equation}\label{lsc good balls}
  |E_j\cap B_\tau(x_j^1)|=\k_1\,,\qquad  |E_j\cap B_\tau(x_j^2)|=\k_2\,,
\end{equation}
for some $\k_1,\k_2\in(0,|B_\tau|/2]$ depending on $n$, $\tau$, $\e$ and $\ell$ only. With $\tau_0$ as in \eqref{hp on W and C 0}, for $M\in\N \setminus \{0\}$ to be chosen later on, and by compactness of $W$, we can pick $\tau>0$ so that
\begin{equation}
  \label{lsc tau req}
  (M+1)\,\tau<\tau_0\,,\qquad |B_{M\,\tau}|<\frac\e4\,,\qquad |I_{(M+1)\tau}(W)\setminus W|<\frac\e2\,.
\end{equation}
The value $\s$ in Lemma \ref{statement nucleation} corresponding to $E_j$ and $T=I_{M\,\tau}(W)$ is given by
\[
\min\Big\{\frac{|E_j\setminus I_\tau(T)|}{\tau\,P(E_j;\R^{n+1}\setminus T)},\frac{\xi(n)}{n+1}\Big\}
\ge\min\Big\{\frac{\e/2}{\tau\,(\psi(\e)+1)},\frac{\xi(n)}{n+1}\Big\}>0\,,
\]
since $|E_j\setminus I_\tau(T)|\ge\e/2$ by \eqref{lsc tau req}, and since $P(E_j;\Om)\le\psi(\e)+1$. Therefore, setting
\[
\s_1 = \min\Big\{\frac{\e/2}{\tau\,(\psi(\e)+1)},\frac{\xi(n)}{n+1}\Big\}\,,
\]
an application of Lemma \ref{statement nucleation} yields $y_j\in \R^{n+1}\setminus I_{(M+1)\tau}(W)$ such that
\begin{eqnarray*}
|E_j\cap B_\tau(y_j)|\ge\min\Big\{\Big(\frac{\s_1}{2\xi(n)}\Big)^{n+1}\tau^{n+1}\,,\frac{|B_{\tau}|}2\Big\}=\k_1\,,
\end{eqnarray*}
so that $\k_1\in(0,|B_\tau|/2]$ depends on $n$, $\ell$, $\e$, and $\tau$ only (observe that this is a consequence of \eqref{psi eps basic bounds}). The continuous map $x\mapsto|E_j\cap B_\tau(x)|$ takes a value larger than $\k_1$ at $y_j\in \R^{n+1}\setminus I_{(M+1)\,\tau}(W)$; at the same time, by \eqref{hp on W and C 0}, $\R^{n+1}\setminus I_{(M+1)\,\tau}(W)$ is open and connected, therefore it is pathwise connected \cite[Corollary 5.6]{topa}, and $|E_j\cap B_\tau(x)|\to 0$ as $|x|\to\infty$ in $\R^{n+1}\setminus I_{(M+1)\,\tau}(W)$. Therefore we can find  $x_j^1\in \R^{n+1}\setminus I_{(M+1)\,\tau}(W)$ such that the first identity in \eqref{lsc good balls} holds and $\{\cl(B_{(M+1)\,\tau}(x_j^1)),W\}$ is disjoint. Setting $B=\cl(B_{(M-2)\tau}(x_j^1))$, the value $\s$ in Lemma \ref{statement nucleation} corresponding to $E_j$ and $T=I_{\tau}(W)\cup B$ is given by
\[
\min\Big\{\frac{|E_j\setminus I_\tau(T)|}{\tau\,P(E_j;\R^{n+1}\setminus T)},\frac{\xi(n)}{n+1}\Big\}
\ge\min\Big\{\frac{\e/4}{\tau\,(\psi(\e)+1)},\frac{\xi(n)}{n+1}\Big\}>0\,,
\]
so that, after setting
\[
\s_2 = \min\Big\{\frac{\e/4}{\tau\,(\psi(\e)+1)},\frac{\xi(n)}{n+1}\Big\}\,,
\]
 we can find $z_j\in\R^{n+1}\setminus( I_{2\tau}(W)\cup \cl(B_{(M-1)\,\tau}(x_j^1)))$ such that
\begin{eqnarray*}
|E_j\cap B_\tau(z_j)|\ge
\min\Big\{\Big(\frac{\s_2}{2\xi(n)}\Big)^{n+1}\tau^{n+1}\,,\frac{|B_{\tau}|}2\Big\}=\k_2\,,
\end{eqnarray*}
with $\k_2\in(0,|B_\tau|/2]$ depending on $n$, $\ell$, $\e$, and $\tau$ only.
Since $I_{2\tau}(W)$ and $\cl(B_{(M-1)\tau}(x_j^1))$ are disjoint and since $\R^{n+1}\setminus I_{2\tau}(W)$ is pathwise connected by \eqref{hp on W and C 0}, we easily check that $\R^{n+1}\setminus( I_{2\tau}(W)\cup \cl(B_{(M-1)\tau}(x_j^1)))$ is pathwise connected. By continuity,
\begin{equation}
  \label{lsc automatic}
\exists \,  x_j^2\in\R^{n+1}\setminus( I_{2\tau}(W)\cup \cl(B_{(M-1)\,\tau}(x_j^1)))
\end{equation}
such that the second identity in \eqref{lsc good balls} holds. Finally, \eqref{lsc automatic} implies that the family of sets $\{\cl(B_{(M-3)\tau}(x_j^1)),\cl(B_{2\tau}(x_j^2)),W\}$ is disjoint. We pick $M=5$ to conclude the proof.

\bigskip

\noindent {\it Step three}: In this step we show that \eqref{minimizing sequence} can be modified to allow for comparison with local variations $F_j$ of $E_j$ that do not necessarily preserve the volume constraint. More precisely, we prove the existence of positive constants $r_*$ and $C_*$ (depending on the whole sequence $\{E_j\}_j$, and thus uniform in $j$) such that if $x\in\Om$, $r<r_*$ and $\{F_j\}_j$ is an {\bf admissible local variation of $\{E_j\}_j$ in $B_r(x)$}, in the sense that
\begin{equation}
  \label{admissible variation}
  F_j\in\E\,,\qquad F_j\Delta E_j\cc B_r(x)\,,
\qquad
\mbox{$\Om\cap\pa F_j$ is $\C$-spanning $W$}\,,
\end{equation}
(notice that we do not require $B_r(x)\subset\Om$), then
\begin{equation}
  \label{volume fixing variation inequality}
  \H^n(\Om\cap\pa E_j)\le \H^n(\Om\cap\pa F_j)+C_*\,\Big||E_j|-|F_j|\Big|+\frac1j\,.
\end{equation}
We first claim that if $B_j \subset \Om$ is a ball with $\dist(B_j,B_r(x))>0$, $\zeta:\Om\to\Om$ is a diffeomorphism with $\zeta(B_j)\subset B_j$ and $\{\zeta\ne\id\}\cc B_j$, and if
\begin{equation}
  \label{form of Gj}
  G_j=\Big(F_j\cap B_r(x)\Big)\cup\Big(\zeta(E_j)\cap B_j\Big)\cup \Big(E_j\setminus (B_j\cup B_r(x))\Big)
\end{equation}
then $G_j\in\E$ and $\Om\cap\pa G_j$ is $\C$-spanning $W$. The fact that $G_j$ is open is obvious since $G_j$ is equal to $E_j$ in a neighborhood of $\Om\setminus (B_r(x)\cup B_j)$, to $F_j$ in a neighborhood of $B_r(x)$, and to $\zeta(E_j)$ in a neighborhood of $B_j$, where $E_j$, $F_j$ and $\zeta(E_j)$ are open, and where $\dist(B_j,B_r(x))>0$; this also shows that $\pa G_j$ is equal to $\pa E_j$ in a neighborhood of $\Om\setminus(B_r(x)\cup B_j)$, to $\pa F_j$ in a neighborhood of $B_r(x)$, and to $\pa\zeta(E_j)=\zeta(\pa E_j)$ in a neighborhood of $B_j$, so that $\Om\cap\pa G_j$ is $\H^n$-rectifiable and, thanks to \eqref{admissible variation} and Lemma \ref{statement spanning is close by Lipschitz maps}, that $\Om\cap\pa G_j$ is $\C$-spanning $W$. Having proved the claim, we only have to construct sets $G_j$ as in \eqref{form of Gj} and such that
\begin{equation}
  \label{lsc Gj properties 1}
  |G_j|=\e\,,\qquad
  \H^n(\Om\cap\pa G_j)\le\H^n(\Om\cap\pa F_j)+C_*\,\big||E_j|-|F_j|\big|\,,
\end{equation}
in order to deduce \eqref{volume fixing variation inequality} from \eqref{minimizing sequence}. To this aim, let $\{x_j^k\}_{k=1,2}$ be as  in step two: the sets $\{(E_j-x_j^k)\cap B_\tau(0)\}_j$ are bounded in $B_\tau(0)$, and have uniformly bounded perimeters, so that, up to extracting a subsequence, for each $k=1,2$ there exists a set of finite perimeter $E_*^k\subset B_\tau(0)$ such that $(E_j-x_j^k)\cap B_\tau(0)\to E_*^k$ in $L^1(\R^{n+1})$. The crucial point is that, by \eqref{lsc good balls} and since $\k_k\in(0,|B_\tau(0)|/2]$, we must have
\[
B_\tau(0)\cap\pa^*E_*^k\ne\emptyset\,.
\]
Hence, by arguing as in \cite[Section 29.6]{maggiBOOK}, we can find positive constants $C_*'$ and $\e_*$ such that for every set of finite perimeter $E'\subset B_\tau(0)$ with
\[
|E'\Delta E_*^k|<\e_*\,,
\]
there exists a $C^1$-map $\Phi_k:(-\e_*,\e_*)\times B_\tau(0)\to B_\tau(0)$ such that, for each $|v|<\e_*$: (i) $\Phi_k(v,\cdot)$ is a diffeomorphism with $\{\Phi_k(v,\cdot)\ne\Id\}\cc B_\tau(0)$; (ii) $|\Phi_k(v,E')|=|E'|+v$; (iii) if $\Sigma$ is an $\H^n$-rectifiable set in $B_\tau(0)$, then
\[
\Big|\H^n(\Phi_k(v,\Sigma))-\H^n(\Sigma)\Big|\le C_*'\,\H^n(\Sigma)\,|v|\,.
\]
By taking $E'=(E_j-x_j^k)\cap B_\tau(0)$ (for $j$ large enough), by composing the maps $\Phi_k$ with a translation by $x_j^k$, and then by extending the resulting maps as the identity map outside of $B_\tau(x_j^k)$, we prove the existence of $C^1$-maps $\Psi_k:(-\e_*,\e_*)\times \R^{n+1}\to\R^{n+1}$ such that, for each $|v|<\e_*$: (i) $\Psi_k(v,\cdot)$ is a diffeomorphism with $\{\Psi_k(v,\cdot)\ne\Id\}\cc B_\tau(x_j^k)$; (ii) $|\Psi_k(v,E_j)|=|E_j|+v$; (iii) if $\Sigma$ is an $\H^n$-rectifiable set in $\R^{n+1}$, then
\[
\Big|\H^n(\Psi_k(v,\Sigma))-\H^n(\Sigma)\Big|\le C_*'\,\H^n(\Sigma)\,|v|\,.
\]
Finally, we set
\[
r_*=\min\Big\{\tau,\Big(\frac{\e_*}{2\,\om_{n+1}}\Big)^{1/(n+1)}\Big\}\,,\qquad B_j=B_{\tau}(x_j^{k(j)})
\]
where $k=k(j)\in\{1,2\}$ is selected so that $\dist(B_r(x),B_j)>0$ (this is possible because $r_x
*\le\tau$ and $\{\cl(B_{2\tau}(x_j^1)),\cl(B_{2\tau}(x_j^2))\}$ are disjoint). We finally define $G_j$ by \eqref{form of Gj} with
\[
\zeta= \Psi_{k(j)}(v_j,\cdot)\,,\qquad v_j=|E_j\cap B_r(x)|-|F_j\cap B_r(x)|\,,
\]
as we are allowed to do since $E_j\Delta F_j\cc B_r(x)$ and thus $|v_j|\le \om_{n+1}\,r_*^{n+1}\le\e_*/2$. To prove \eqref{lsc Gj properties 1}: first, we have $G_j\Delta F_j\cc\Om\setminus\cl(B_r(x))$, while property (ii) of $\Psi_{k(j)}$ gives
\begin{eqnarray*}
|G_j|-|E_j|&=&|\Psi_{k(j)}(v_j,E_j)\cap B_j|+|F_j\cap B_r(x)|-|E_j\cap B_j|-|E_j\cap B_r(x)|
  \\
  &=&|\Psi_{k(j)}(v_j,E_j)\cap B_j|-v_j-|E_j\cap B_j|=0\,;
\end{eqnarray*}
second, property (iii) applied to the $\H^n$-rectifiable set $\S=B_j\cap\pa E_j$ gives
\begin{eqnarray*}
  &&\H^n(\Om\cap\pa G_j)-\H^n(\Om\cap\pa F_j)
  \\
  &=&\H^n\Big(\Psi_{k(j)}(v_j,B_j\cap\pa E_j)\Big)-\H^n(B_j\cap\pa E_j)\le C_*'\,|v_j|\,
  \H^n(B_j\cap\pa E_j)
\end{eqnarray*}
so that \eqref{lsc Gj properties 1} follows by taking $C_*=C_*'\,(\psi(\e)+1)$.

\bigskip

\noindent {\it Step four}: In this step we apply \eqref{volume fixing variation inequality} to the cup and cone competitors constructed in section \ref{section five competitors} and show that $K = \Om \cap \spt\mu$ is relatively compact in $\Om$ and $\H^n$-rectifiable, that $\mu=\theta\,\H^n\llcorner K$ with $\theta\ge1$ on $K$ and $\theta\ge 2$ $\H^n$-a.e. on $K\setminus\pa^*E$, and, finally, that $(K,E)\in\KK$. To this end, pick $x\in K$, set $d(x)=\dist(x,W)>0$, and let
\[
f_j(r)=\mu_j(B_r(x))=\H^n(B_r(x)\cap\pa E_j)\,,\qquad f(r)=\mu(B_r(x))\,,\qquad \mbox{for every $r\in(0,d(x))$}\,.
\]
Denoting by $Df$ the distributional derivative of $f$, and by $f'$ its classical derivative, the coarea formula (see \cite[Step one, proof of Theorem 2]{DLGM} and \cite[Theorem 2.9.19]{FedererBOOK}) gives
\begin{eqnarray}\label{fj f g}
  \mbox{$f_j\to f$ a.e. on $(0,d(x))$}\,,\quad
  Df_j\ge f_j'\,dr\,,\quad Df\ge f'\,dr\,,\quad f'\ge g=\liminf_{j\to\infty}f_j'\,,
\\
\label{lb fj'}
  f_j'(r)\ge\H^{n-1}(\pa B_r(x)\cap \pa E_j)\qquad\mbox{$\forall j$ and for a.e. $r\in(0,d(x))$}\,.\hspace{1cm}
\end{eqnarray}
Now let $\eta\in(0,r/2)$, let $A_j$ denote an $\H^n$-maximal open connected component of $\pa B_r(x)\setminus\pa E_j$, and let $F_j$ be the cup competitor defined by $E_j$ and $A_j$ as in Lemma \ref{lemma cup competitor first kind}. More precisely, when $E_j \cap A_j = \emptyset$, we let $\{\eta^j_k\}_{k=1}^\infty$ be the decreasing sequence with $\lim_{k \to \infty} \eta^j_k = 0$ defined in step two of the proof of Lemma \ref{lemma cup competitor first kind}, and setting, for $\eta^j_k$ such that $\eta \in \left( \eta^j_{k+1}, \eta^j_k \right]$,
\begin{eqnarray*}
 Y_j&=&\pa B_r(x)\setminus\big(\cl(E_j\cap\pa B_r(x))\cup\cl(A_j)\big)\,, \\
 S_j &=& \pa E_j \cap \cl(A_j) \setminus \left( \cl (E_j \cap \pa B_r (x)) \cup \cl (Y_j) \right)\,, \\
U_j &=& \pa B_r (x) \cap \{{\rm d}_{S_j} < \eta^j_k\}\,,
\end{eqnarray*}
we define
\begin{equation}
    \label{cup competitor first case}
    F_j=\big(E_j\setminus\cl(B_r(x))\big)\cup\,N_{\eta^j_k}(Z_j)\,, \qquad  Z_j = Y_j \cup \left( U_j \setminus \cl (E_j \cap \pa B_r (x)) \right)\,.
    \end{equation}
When $A_j \subset E_j$, instead, we define
    \begin{equation}     \label{cup competitor second case}
    F_j=\big(E_j\cup B_r(x)\big)\setminus\cl\big(N_\eta(Y_j)\big)\,, \qquad Y_j=(E_j\cap\pa B_r(x))\setminus\cl(A_j)\,;
\end{equation}
  see Figure \ref{fig cup00}. In both cases, $\{F_j\}_j$ is an admissible local variation of $\{E_j\}_j$ in $B_{r'}(x)$ for some $r'>r$, and by \eqref{cup area totale}, for a.e. $r<d(x)$ we have
  \[
  \limsup_{\eta\to 0^+}\H^n(\Om\cap\pa F_j)\le \H^n(\pa E_j\setminus B_r(x))+2\,\H^n(\pa B_r(x)\setminus A_j)
  \]
  so that, by \eqref{volume fixing variation inequality}, for a.e. $r<\min\{d(x),r_*\}$, we have
  \begin{equation}\label{cup analysis 1}
  f_j(r)\le2\,\H^n(\pa B_r(x)\setminus A_j)+C_*\,\limsup_{\eta\to 0^+}\big||E_j|-|F_j|\big|+\frac1j\,.
  \end{equation}
The estimate of $||E_j|-|F_j||$ is different depending on whether $F_j$ is given by \eqref{cup competitor first case} or by \eqref{cup competitor second case}. In both cases we make use of the Euclidean isoperimetric inequality
\[
(n+1)\,|B_1|^{1/(n+1)}\,|U|^{n/(n+1)}\le P(U)\,,\qquad\forall U\subset\R^{n+1}\,,
\]
and we also need the perimeter identities
\begin{equation}
  \label{choice of r 2}
  \begin{split}
    &P(E_j\cap B_r(x))=P(E_j;B_r(x))+\H^n(E_j\cap\pa B_r(x))\,,
    \\
    &P(B_r(x)\setminus E_j)=P(E_j;B_r(x))+\H^n(\pa B_r(x)\setminus E_j)\,,
  \end{split}
\end{equation}
which hold for a.e. $r>0$, with the exceptional set of $r$-values that can be made independent from $j$. We now take $F_j$ as in \eqref{cup competitor first case}: up to further decreasing the value of $r_*$ so to entail $C_*\,r_*/(n+1)\le 1/2$, and assuming that $r<r_*$, we have
\begin{eqnarray}\nonumber
  C_*\,\Big||E_j|-|F_j|\Big|&\le& C_*\,|E_j\cap B_r(x)|+C_*\,c(n)\,r^n\,\eta^j_k
  \\\nonumber
  &\le& C_*\,|B_1|^{1/(n+1)}\,r\,|E_j\cap B_r(x)|^{n/(n+1)}+C_*\,c(n)\,r^n\,\eta^j_k
  \\\nonumber
  &\le& \frac{C_*}{n+1}\,r_*\,P(E_j\cap B_r(x))+C_*\,c(n)\,r^n\,\eta^j_k
  \\\nonumber
  &\le&\frac12\Big\{P(E_j;B_r(x))+\H^n(E_j\cap\pa B_r(x))\Big\}+C_*\,c(n)\,r^n\,\eta^j_k
  \\\label{volume error estimate 1}
  &\le&\frac12\Big\{f_j(r)+\H^n(\pa B_r(x)\setminus A_j)\Big\}+C_*\,c(n)\,r^n\,\eta^j_k\,,
\end{eqnarray}
where in the last inequality we have used $\pa^*E_j\subset\pa E$ and $A_j\cap E_j=\emptyset$ (that is the assumption under which $F_j$ is chosen as in \eqref{cup competitor first case}). If instead we take $F_j$ as in \eqref{cup competitor second case}, then
\begin{eqnarray}\nonumber
  C_*\,\Big||E_j|-|F_j|\Big|&=&C_*\,\Big||E_j\cap B_r(x)|-|F_j\cap B_r(x)|\Big|=C_*\,\Big||B_r(x)\setminus E_j|-|B_r(x)\setminus F_j|\Big|
  \\\nonumber
  &\le&C_*\,|B_1|^{1/(n+1)}\,r\,|B_r(x)\setminus E_j|^{n/(n+1)}+C_*\,|N_\eta(\pa B_r(x)\cap E_j\setminus\cl(A_j))|
  \\\nonumber
  &\le&\frac12\Big\{P(E_j;B_r(x))+\H^n(\pa B_r(x)\setminus E_j)\Big\}+C_*\,c(n)\,r^n\,\eta
  \\\label{volume error estimate 2}
  &\le&\frac12\Big\{f_j(r)+\H^n(\pa B_r(x)\setminus A_j)\Big\}+C_*\,c(n)\,r^n\,\eta\,,
\end{eqnarray}
where in the last inequality we have used $\pa^*E_j\subset\pa E_j$ and $A_j\subset E_j$ (the assumption corresponding to \eqref{cup competitor second case}). By combining \eqref{cup analysis 1} with \eqref{volume error estimate 1} and \eqref{volume error estimate 2}, we conclude that
\begin{equation}
  \label{key inequality}
  \frac{f_j(r)}2\le 3\,\H^n(\pa B_r(x)\setminus A_j)+\frac1j\,,\qquad\mbox{for a.e. $r<\min\{r_*,d(x)\}$}\,.
\end{equation}
By the spherical isoperimetric inequality, Lemma \ref{statement isoperimetry on spheres}, and by \eqref{lb fj'}, for a.e. $r<d(x)$,
\[
\H^n(\pa B_r(x)\setminus A_j)\le C(n)\,\H^{n-1}(\pa B_r(x)\cap \pa E_j)^{n/(n-1)}\le C(n)\,f_j'(r)^{n/(n-1)}\,,
\]
which combined with \eqref{key inequality} and \eqref{fj f g}, allows us to conclude (letting $j\to\infty$), that
\begin{eqnarray}\label{cup analysis 3}
  f(r)\le C(n)\,f'(r)^{n/(n-1)}\,,\qquad\mbox{for a.e. $r<\min\{r_*,d(x)\}$}\,.
\end{eqnarray}
Since $x\in\spt\mu$, $f$ is positive, and thus \eqref{cup analysis 3} implies the existence of $\theta_0(n)>0$ such that
\begin{equation}
  \label{mu lower bound basic}
  \mu(B_r(x))\ge \theta_0\,\om_n\,r^n\qquad\forall x\in K\,,r<r_*\,,B_r(x)\cc\Om\,.
\end{equation}
Since $K=\Om\cap\spt\mu$, by \cite[Theorem 6.9]{mattila} and \eqref{mu lower bound basic} we obtain
\begin{equation}
  \label{mu lb HnK}
  \mu\ge\theta_0\,\H^n\llcorner K\qquad\mbox{on $\Om$}\,.
\end{equation}
As a consequence of $\mu(\Om)<\infty$ and of \eqref{mu lower bound basic} we deduce that $K$ is bounded, thus relatively compact in $\Om$. In turn, $\pa^* E\subset K$ implies the boundedness of $E$. Notice that we have not excluded $|E|<\e$ yet.

To further progress in the analysis of $\mu$, given $\eta\in(0,r/2)$ let use now denote by $F_j$ the set corresponding to $\eta$ constructed in Lemma \ref{lemma cone competitor}, so that, by \eqref{cone competitor area inequality}, for a.e. $r<d(x)$,
\begin{equation}
    \label{cone competitor area inequality proof}
    \limsup_{\eta\to 0^+}\H^n(\Om\cap\pa F_j)\le\H^n(\pa E_j\setminus B_r(x))+\frac{r}n\,\H^{n-1}(\pa E_j\cap\pa B_r(x))\,.
  \end{equation}
Using that $\{F_j\}_j$ is an admissible local variation of $\{E_j\}_j$ in $B_r(x)$, and combining \eqref{volume fixing variation inequality} and \eqref{cone competitor area inequality proof} with  $||E_j|-|F_j||\le C(n)\,r^{n+1}$, we find that
\[
\H^n(B_r(x)\cap\pa E_j)\le \frac{r}n\,f_j'(r)+C_*\,r^{n+1}+\frac1j\,,
\]
so that, as $j\to\infty$, $f(r)\le (r/n)\,f'(r)+C_*\,r^{n+1}$. By combining this last inequality with $Df\ge f'(r)\,dr$ and \eqref{mu lower bound basic} we find that
\begin{eqnarray*}
D\,(e^{\Lambda\,r}f(r)/r^n)&=&\frac{n\,e^{\Lambda\,r}}{r^{n+1}}\Big\{\frac{r}n\,Df+\Big(\frac{r\,\Lambda}n\,f(r)-\,f(r)\Big)\,dr\Big\}
\\
&\ge&
  \frac{n\,e^{\Lambda\,r}}{r^{n+1}}\Big\{f(r)-C_*\,r^{n+1}+\frac{r\,\Lambda}n\,f(r)-\,f(r)\Big\}\,dr
  \\
  &=&
    \frac{n\,e^{\Lambda\,r}}{r^n}\Big\{-C_*\,r^{n}+\frac{\Lambda\,f(r)}n\Big\}\,dr
  \ge
    n\,\,e^{\Lambda\,r}\Big\{-C_*+\frac{\Lambda\,\theta_0\,\om_n}n\Big\}\,dr
\end{eqnarray*}
so that, setting $\Lambda\ge n\,C_*/(\theta_0\om_n)$, we have proved
\begin{equation}
  \label{monotonicity}
  e^{\Lambda\,r}\,\frac{\mu(B_r(x))}{r^n}\qquad\mbox{is non-decreasing on $r<\min\{r_*,d(x)\}$}\,.
\end{equation}
By \eqref{monotonicity} and \eqref{mu lb HnK} we find that
\[
\theta(x)=\lim_{r\to 0^+}\frac{\mu(B_r(x))}{\om_n\,r^n}\quad\mbox{exists in $(0,\infty)$ for every $x\in K$}\,.
\]
By Preiss' theorem, $\mu=\theta\,\H^n\llcorner K^*$ for a Borel function $\theta$ and a countably $\H^n$-rectifiable set $K^*\subset\Om$. Since $K=\Om \cap \spt\mu$, we have $\H^n(K^*\setminus K)=0$, while \eqref{mu lb HnK} gives $\H^n(K\setminus K^*)=0$. Thus $K$ is countably $\H^n$-rectifiable and $\mu=\theta\,\H^n\llcorner K$. Moreover, $\theta$ is upper semicontinuous on $K$ thanks to \eqref{monotonicity}. Finally, consider the open set
\[
E^*=\big\{x\in\Om:\mbox{$\exists r>0$ s.t. $|B_r(x)|=|E\cap B_r(x)|$}\big\}\,.
\]
The topological boundary of $E^*$ is equal to
\[
\pa E^*=\big\{x\in\cl(\Om):0<|E\cap B_r(x)|<|B_r(x)|\quad\forall r>0\big\}\,,
\]
so that $\Om\cap\cl(\pa^*E)=\Om\cap\pa E^*$ by \cite[Proposition 12.19]{maggiBOOK}. Clearly $E^*\subset E^{(1)}$: moreover, if $x\in E^{(1)}\setminus E^*$, then $0<|E\cap B_r(x)|<|B_r(x)|$ for every $r>0$, and thus $x\in\pa E^*$. In particular,
\[
\Om\cap(E^{(1)}\setminus E^*)\subset\Om\cap\pa E^*=\Om\cap\cl(\pa^*E)\subset K\,,
\]
where $K$ is $\H^n$-rectifiable, and thus Lebesgue negligible. Since $\H^n(\pa\Om)<\infty$, we have proved $\H^n(E^{(1)}\setminus E^*)<\infty$, and thus $|E^{(1)}\Delta E^*|=0$. By the Lebesgue's points theorem, $E^*$ is equivalent to $E$, so that $\pa^*E=\pa^*E^*$. Replacing $E$ with $E^*$ we find $(K,E)\in\KK$. Finally, the lower bounds $\theta\ge 1$ $\H^n$-a.e. on $K$ and $\theta\ge2$ $\H^n$-a.e. on $K\setminus\pa^*E$ follow by applying Lemma \ref{lemma llb} with $\Om'=\Om$: notice indeed that assumptions \eqref{llb1} and \eqref{llb2} in Lemma \ref{lemma llb} hold by \eqref{mu lower bound basic} and by \eqref{key inequality}.

\medskip

\noindent {\it Step five}: We show that $\theta(x)\le1$ at every $x\in\Om\cap\pa^*E$ and that $\theta(x)\le2$ at every $x\in K\cap (E^{(0)}\cup E^{(1)})$ such that $K$ admits an approximate tangent plane at $x$ (thus, that $\theta\le 2$ $\H^n$-a.e. on $K\setminus\pa^*E$). We choose $\nu(x)\in\SS^n$ such that $T_xK=\nu(x)^\perp$ (notice that, necessarily, $\nu(x)=\nu_E(x)$ or $\nu(x)=-\nu_E(x)$ when, in addition, $x\in\pa^*E$), and let $B_{2\,r}(x)\cc\Om$. For $\tau\in(0,1)$ and $\s\in(0,\tau)$ we set
\begin{eqnarray}\label{all the young dudes}
S_{\tau,r}&=&\big\{y\in B_r(x):|(y-x)\cdot\nu(x)|<\tau\,r\big\}\,,
\\\nonumber
V_{\s,r}&=&\big\{y\in B_r(x):|(y-x)\cdot\nu(x)|<\s\,|y-x|\big\}\,\subset\, S_{\s,r}\,\subset\, S_{\tau,r}\,,
\\\nonumber
W_{\tau,\s,r}^\pm&=&\big(S_{\tau,r}\setminus\cl(V_{\s,r})\big)\cap\{y:(y-x)\cdot\nu_E(x)\gtrless0\}\,,
\\\nonumber
\Gamma_{\tau,\s,r}^\pm&=&\pa S_{\tau,r}\cap\pa W_{\tau,\s,r}^\pm\,,
\end{eqnarray}
that are depicted  in
\begin{figure}
  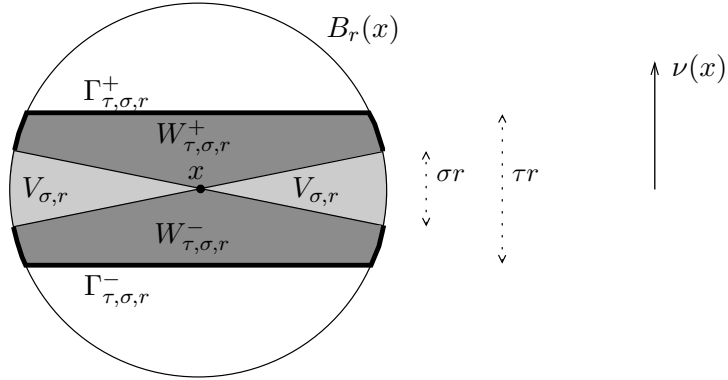\caption{{\small The sets defined in \eqref{all the young dudes}. Here $\s<\tau<1$, and $S_{\tau,r}$ is decomposed into a central open cone $V_{\s,r}$ of small amplitude $\s$, the upper and lower open regions $W_{\tau,\s,r}^\pm$, and the closed cone $S_{\tau,r}\cap\pa V_{\s,r}$. For $r\le r_0(\s,x)$, $B_r(x)\cap K$ lies inside $V_{\s,r}$ by approximate differentiability of $K$ at $x$ and by the density estimate \eqref{mu lower bound basic}. When $x\in\pa^*E$, if we choose $\nu(x)=\nu_E(x)$, then the divergence theorem implies that $E$ fills up the whole $W_{\tau,\s,r}^-$, and leaves empty $W_{\tau,\s,r}^+$.
  }}\label{fig step8geometry}
\end{figure}
Figure \ref{fig step8geometry}. By \eqref{mu lower bound basic} and since $\H^n\llcorner(K-x)/\rho\weakstar\H^n\llcorner T_xK$ as $\rho\to 0^+$,
the approximate tangent plane $T_xK$ is a classical tangent plane, and thus there exists $r_0=r_0(\s,x)>0$ such that $K\cap B_r(x)\subset S_{\s,r}$ for every $r<r_0$, or, equivalently,
\begin{equation}\label{theta 1 bordo nw}
  K\cap B_{r_0}(x)\subset V_{\s,r_0}\cup\{x\}\,.
\end{equation}
In particular
\begin{equation}
\label{theta 1 fine 00}
  \mu(S_{\tau,r})=\mu(B_r(x))\,,\qquad\forall r<r_0\,.
\end{equation}
We also notice that for a.e. value of $r$ we have
\begin{equation}\label{slab null}
  \mbox{$\pa S_{\tau,r}\cap\pa E_j$ is $\H^{n-1}$-rectifiable}\qquad\forall j\,.
\end{equation}
We now introduce the family of open sets
\begin{eqnarray}\nonumber
\A_{r,j}^{{\rm out}}&=&\Big\{A\subset \pa S_{\tau,r}:\mbox{$A$ is an open connected component}
 \\\nonumber
 &&\hspace{4cm}\mbox{of $\pa S_{\tau,r}\setminus\pa E_j$ and $A$ is {\bf disjoint} from $E_j$}\Big\}\,,
\\
\nonumber
\A_{r,j}^{{\rm in}}&=&\Big\{A\subset \pa S_{\tau,r}:\mbox{$A$ is an open connected component}
 \\\nonumber
 &&\hspace{4cm}\mbox{of $\pa S_{\tau,r}\setminus\pa E_j$ and $A$ is {\bf contained} in $E_j$}\Big\}\,,
\end{eqnarray}
and denote by $A_{r,j}^{{\rm out}}$ and $A_{r,j}^{{\rm in}}$ $\H^n$-maximal elements of $\A_{r,j}^{{\rm out}}$ and $\A_{r,j}^{{\rm in}}$ respectively. Finally, given $\eta\in(0,r/2)$, we let $F_j^\star$ be the slab competitor defined by $E_j$, $A_{r,j}^\star$ and $\tau$ in $B_{2r}(x)$ for $\star\in\{{\rm out},{\rm in}\}$ as in Lemma \ref{lemma slab competitor}: accordingly, $F_j^\star\in\E$, $\Om\cap\pa F_j^\star$ is $\C$-spanning $W$,
\begin{eqnarray}
  \label{new slabs 1}
  &&F_j^\star\setminus\cl(S_{\tau,r})=E_j\setminus\cl(S_{\tau,r})\,,
  \\
  \label{new slabs 2}
 && \lim_{\eta \to 0^+} \H^n \left( (\pa S_{\tau,r}\cap\pa F_j^\star) \, \Delta \, (\pa S_{\tau,r}\setminus A_{r,j}^\star) \right)  = 0\,,
\end{eqnarray}
and
\begin{eqnarray}
  \label{new slabs 3}
  \limsup_{\eta\to 0^+}\H^n(S_{\tau,r}\cap\pa F_j^\star)\le C(n,\tau)\,\left\{
  \begin{split}
    &\H^n\big(\pa S_{\tau,r}\setminus(A_{r,j}^{{\rm out}}\cup E_j)\big)\,,\hspace{1cm}\mbox{if $\star={\rm out}$}\,,
    \\
    &\H^n\big((E_j\cap\pa S_{\tau,r})\setminus A_{r,j}^{{\rm in}}\big)\,,\hspace{1.2cm}\mbox{if $\star={\rm in}$}\,;
  \end{split}
  \right .
\end{eqnarray}
see \eqref{slab competitor exterior}, \eqref{slab competitor buccia}, \eqref{area of slab competitors first 0} and \eqref{area of slab competitors second 0}. By \eqref{volume fixing variation inequality}, $\H^n(\pa S_{\tau,r}\cap\pa E_j)=0$ and \eqref{new slabs 1},
\[
 \H^n(S_{\tau,r}\cap\pa E_j)\le\H^n(\cl(S_{\tau,r})\cap\pa F_j^\star)+C_*\,c(n)\,r^{n+1}+\frac1j\,,\qquad\forall\star\in\{{\rm out},{\rm in}\}\,.
\]
By \eqref{new slabs 2} and \eqref{new slabs 3}, taking the limit first as $\eta\to 0^+$ and then as $j\to\infty$, and by taking also into account that $\mu_j\weakstar\mu$ and that \eqref{theta 1 fine 00} holds, we find, in the case $\star={\rm out}$, that
 \begin{eqnarray}\label{area of slab competitors first proof}
 \mu(B_r(x))&\le&\limsup_{j\to\infty}\H^n(E_j\cap\pa S_{\tau,r})
 \\\nonumber
 &&+C(n,\tau)\,\limsup_{j\to\infty}\H^n\big(\pa S_{\tau,r}\setminus(A_{r,j}^{{\rm out}}\cup E_j)\big)+C_*\,c(n)\,r^{n+1}\,,
 \end{eqnarray}
 and, in the case $\star={\rm in}$, that
  \begin{eqnarray}\label{area of slab competitors second proof}
 \mu(B_r(x))&\le&\limsup_{j\to\infty}\H^n(\pa S_{\tau,r}\setminus E_j)
 \\\nonumber
 &&+C(n,\tau)\,\limsup_{j\to\infty}\H^n\big((E_j\cap\pa S_{\tau,r})\setminus A_{r,j}^{{\rm in}}\big)+C_*\,c(n)\,r^{n+1}\,.
 \end{eqnarray}
 We now discuss the cases $x\in\pa^*E$, $x\in K\cap E^{(0)}$ and $x\in K\cap E^{(1)}$ separately.

 \medskip

 \noindent {\it The case $x\in\pa^*E$}: We claim that, in this case, for every $\s\in(0,\tau)$ and for a.e. $r<r_0(\s,x)$,
\begin{eqnarray}
  \label{theta 1 fine 0}
  \limsup_{j\to\infty}\H^n\Big(\pa S_{\tau,r}\setminus \big(A_{r,j}^{{\rm out}}\cup E_j\big)\Big)&\le&C(n)\,\s\,r^n\,,
  \\ \label{theta 1 fine}
  \limsup_{j\to\infty}\Big|\H^n\big(E_j\cap\pa S_{\tau,r}\big)-\om_n\,r^n\Big|&\le& C(n)\,\tau\,r^n\,;
\end{eqnarray}
see
\begin{figure}
  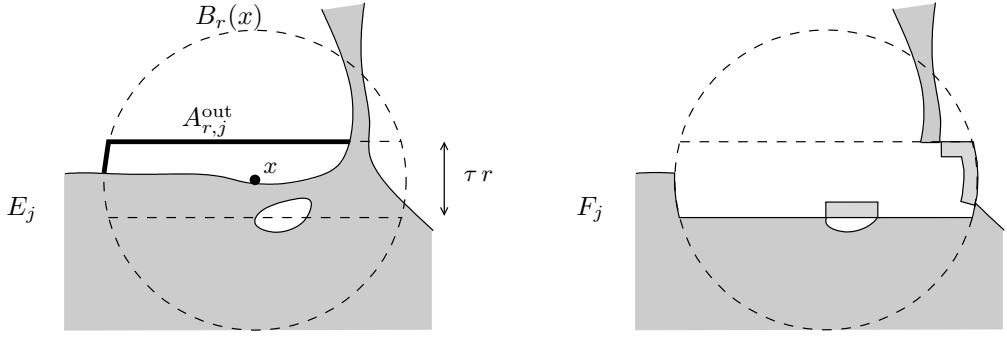\caption{{\small The slab competitor $F_j^{{\rm out}}$ is used in proving that $\theta(x)\le 1$. The fact that $x\in\pa^*E$ is used to show that $E_j\cap\pa S_{\tau,r}$ consists of a large connected component whose area is close to $\om_n\,r^n$ up to a ${\rm o}(r^n)$ error as $r\to 0^+$.}}\label{fig slab1}
\end{figure}
Figure \ref{fig slab1}. We notice that \eqref{theta 1 fine 0} and \eqref{theta 1 fine} combined with \eqref{area of slab competitors first proof} imply
\[
\frac{\mu(B_r(x))}{r^n}\le \om_n+C(n)\,\tau+ C(n,\tau)\,\s+C_*\,c(n)\,r\,,\qquad\mbox{for a.e. $r<r_0$}\,,
\]
which gives $\theta(x)\le 1$ by letting, in the order, $r\to 0^+$, $\s\to 0^+$ and then $\tau\to 0^+$. We now prove \eqref{theta 1 fine 0} and \eqref{theta 1 fine}. Since $x\in\pa^*E$, we can set $\nu(x)=\nu_E(x)$. As $\nu_E(x)$ is the outer normal to $E$, by $\pa^*E\subset K$, \eqref{theta 1 bordo nw} and the divergence theorem, we obtain
\[
|W_{\tau,\s,r_0}^-\setminus E|=|W_{\tau,\s,r_0}^+\cap E|=0\,.
\]
By $|W_{\tau,\s,r_0}^-\setminus E|=0$, the coarea formula and Fatou's lemma, we deduce
\begin{eqnarray*}
  0&=&\lim_{j\to\infty}|W_{\tau,\s,r_0}^-\setminus E_j|=\lim_{j\to\infty}\int_0^{r_0}\H^n\Big(\pa S_{\tau,r}\cap \big(W_{\tau,\s,r_0}^-\setminus E_j\big)\Big)\,dr
  \\
  &\ge&\int_0^{r_0}
  \liminf_{j\to\infty}\H^n\Big(\Gamma_{\tau,\s,r}^-\setminus E_j\Big)\,dr\,,
\end{eqnarray*}
and by arguing similarly with $|W_{\tau,\s,r_0}^+\cap E|=0$ we conclude that, for a.e. $r<r_0$,
\begin{eqnarray}
  \label{theta 1 bordo nw UP}
  &&\lim_{j\to\infty}\H^n\big(\Gamma_{\tau,\s,r}^+\cap E_j\big)=0\,,
  \\
  \label{theta 1 bordo nw DOWN}
  &&\lim_{j\to\infty}\H^n\big(\Gamma_{\tau,\s,r}^-\setminus E_j\big)=0\,.
\end{eqnarray}
By \eqref{theta 1 bordo nw UP}, \eqref{theta 1 bordo nw DOWN}, and since
\begin{equation}
  \label{slab decomp}
  \pa S_{\tau,r}=\Gamma_{\tau,\s,r}^+\cup\Gamma_{\tau,\s,r}^-\cup\big(\pa S_{\tau,r}\cap\pa S_{\s,r}\big)
\end{equation}
we find that, as $j\to\infty$,
\begin{eqnarray*}
  \big|\H^n(\pa S_{\tau,r}\cap E_j)-\om_n\,r^n\big|&\le&\H^n(\pa S_{\tau,r}\cap\pa S_{\s,r})
  +\big|\H^n(\Gamma_{\tau,\s,r}^-\cap E_j)-\om_n\,r^n\big|+{\rm o}(1)
  \\
  &\le&C(n)\,\s\,r^n+\big|\H^n(\Gamma_{\tau,\s,r}^-)-\om_n\,r^n\big|+{\rm o}(1)
  \\
  &\le&C(n)\,\tau\,r^n+{\rm o}(1)\,,
\end{eqnarray*}
that is \eqref{theta 1 fine}. At the same time, again by \eqref{theta 1 bordo nw} and by the coarea formula, assuming without loss of generality that $r_0=r_0(\s,x)$ also satisfies $\H^n(K\cap\pa B_{r_0}(x))=0$ in addition to \eqref{theta 1 bordo nw}, we get
\begin{eqnarray}\nonumber
  0&=&\mu(K\cap \cl(B_{r_0}(x))\setminus V_{\s,r_0})=\lim_{j\to\infty}\H^n\big(B_{r_0}(x)\cap\pa E_j\setminus V_{\s,r_0}\big)
  \\\nonumber
  &\ge&\lim_{j\to\infty}\H^n\big(S_{\tau,r_0}\cap\pa E_j\setminus V_{\s,r_0}\big)
  \\\nonumber
  &\ge&\lim_{j\to\infty}\int_0^{r_0}\,\H^{n-1}\big(\pa S_{\tau,r}\cap\pa E_j\setminus V_{\s,r_0}\big)\,dr\,,
\end{eqnarray}
that is
\begin{equation}
  \label{theta 1 bordo nw 1}
  \lim_{j\to\infty}\H^{n-1}\big(\pa S_{\tau,r}\cap\pa E_j\setminus V_{\s,r_0}\big)=0\qquad\mbox{for a.e. $r<r_0$}\,.
\end{equation}
Notice that \eqref{theta 1 bordo nw 1} implies in particular that
\begin{equation}
  \label{theta 1 bordo nw 2}
  \lim_{j\to\infty}\H^{n-1}\big(\Gamma_{\tau,\s,r}^+\cap\pa E_j\big)=0\qquad\mbox{for a.e. $r<r_0$}\,.
\end{equation}
Since $\Gamma_{\tau,\s,r}^+$ is a bi-Lipschitz image of a hemisphere, by Lemma \ref{statement isoperimetry on spheres},
\begin{equation}
  \label{isoperimetric on half cylinder}
  \H^{n-1}(\Gamma_{\tau,\s,r}^+\cap J)^{n/(n-1)}\ge c(n,\tau,\s)\,\H^n(\Gamma_{\tau,\s,r}^+\setminus A)\,,
\end{equation}
whenever $J$ is relatively closed in $\Gamma_{\tau,\s,r}^+$, and $A$ is an $\H^n$-maximal connected component of $\Gamma_{\tau,\s,r}^+\setminus J$. By \eqref{theta 1 bordo nw 2} and \eqref{isoperimetric on half cylinder} we find that, if
\[
\mbox{$A_{r,j}^+$ is a maximal $\H^n$-component of $\Gamma_{\tau,\s,r}^+\setminus\pa E_j$}\,,
\]
then
\begin{equation}
  \label{theta 1 bordo nw 3}
  \lim_{j\to\infty}\H^n(\Gamma_{\tau,\s,r}^+\setminus A_{r,j}^+)=0\,,\qquad\mbox{for a.e. $r<r_0$}\,.
\end{equation}
By connectedness, $A_{r,j}^+$ is either contained in $A_{r,j}^{{\rm out}}$, or in $E_j$, or in
\[
Y_{r,j}=\bigcup\big\{A:A\in\A_{r,j}^{{\rm out}}\,,A\ne A_{r,j}^{{\rm out}}\big\}\,.
\]
By combining \eqref{theta 1 bordo nw UP} with \eqref{theta 1 bordo nw 3} we find that for a.e. $r<r_0$, if $j$ is large enough, then
\[
A_{r,j}^+\cap E_j=\emptyset\,.
\]
Similarly, should there be a non-negligible set of values of $r$ such that for infinitely many value of $j$ the inclusion $A_{r,j}^+\subset Y_{r,j}$ holds, then by \eqref{theta 1 bordo nw DOWN} and \eqref{theta 1 bordo nw 3} there would be an element of $\A_{r,j}^{{\rm out}}$ different from $A_{r,j}^{{\rm out}}$ with $\H^n$-measure arbitrarily close to $\H^n(\Gamma_{\tau,\s,r}^+)$; thanks to \eqref{theta 1 bordo nw DOWN}, we would then have $\H^n(A_{r,j}^{{\rm out}})\to 0$, against the $\H^n$-maximality of $A_{r,j}^{{\rm out}}$ itself. In conclusion, it must be
\begin{equation}
  \label{theta 1 bordo nw 4}
  \mbox{$A_{r,j}^+\subset A_{r,j}^{{\rm out}}$ for a.e. $r<r_0$ and for $j$ large enough}\,.
\end{equation}
By combining \eqref{theta 1 bordo nw 4}  and \eqref{theta 1 bordo nw 3} we conclude that
\begin{equation}
  \label{theta 1 bordo nw 5}
  \lim_{j\to\infty}\H^n\big(\Gamma_{\tau,\s,r}^+\setminus A_{r,j}^{{\rm out}}\big)=0\,.
\end{equation}
By \eqref{slab decomp}, \eqref{theta 1 bordo nw DOWN} and \eqref{theta 1 bordo nw 5} we conclude that
\[
\limsup_{j\to\infty}\H^n\Big(\pa S_{\tau,r}\setminus \big(A_{r,j}^{{\rm out}}\cup E_j\big)\Big)
\le \H^n(\pa S_{\tau,r}\cap \pa S_{\s,r})\le C(n)\,\s\,r^n\,,
\]
that is \eqref{theta 1 fine 0}. This completes the proof of $\theta(x)\le1$ for $x\in\pa^*E$.

\medskip

\noindent {\it The case $x\in E^{(0)}$}: We claim that, in this case, for every $\s\in(0,\tau)$,
\begin{eqnarray}
  \label{coney island 1}
  \limsup_{j\to\infty}\H^n(E_j\cap \pa S_{\tau,r})\le C(n)\,\s\,r^n\,,
  \\
  \label{coney island 1 star}
  \limsup_{j\to\infty}\big|\H^n\big(\pa S_{\tau,r}\setminus E_j\big)-2\,\om_n\,r^n\big|\le C(n)\,\tau\,r^n\,,
 \end{eqnarray}
 for a.e. $r<r_0(\s,x)$, see
\begin{figure}
  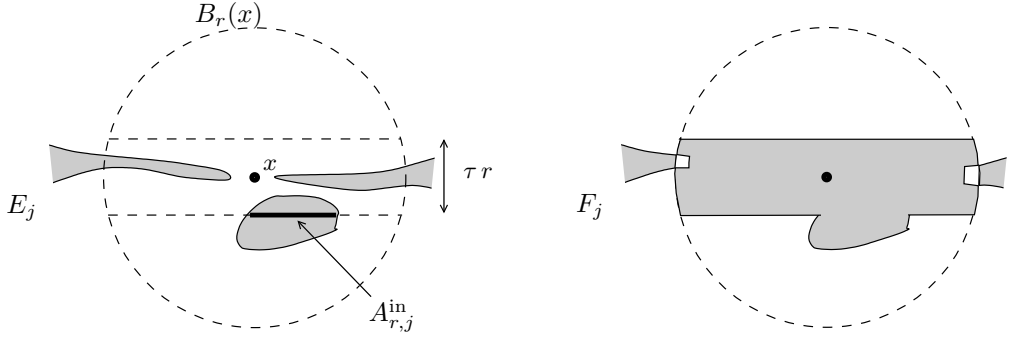\caption{{\small The slab competitor used in proving that $\theta(x)\le 2$ when $x\in E^{(0)}$ is the one defined by $A_{r,j}^{{\rm in}}$. Since $x\in E^{(0)}$ we can show that $E_j\cap\pa S_{\tau,r}$ is ${\rm o}(r^n)$ as $r\to 0^+$.}}\label{fig slab2}
\end{figure}
Figure \ref{fig slab2}. The idea is using the competitor defined by $A_{r,j}^{{\rm in}}$: indeed, \eqref{coney island 1}, \eqref{coney island 1 star}, and \eqref{area of slab competitors second proof} give
\begin{eqnarray*}
 \frac{\mu(B_r(x))}{r^n}
 &\le&\limsup_{j\to\infty}\frac{\H^n(\pa S_{\tau,r}\setminus E_j)}{r^n}
 \\\nonumber
 &&+C(n,\tau)\,\limsup_{j\to\infty}\frac{\H^n\big((E_j\cap\pa S_{\tau,r})\setminus A_{r,j}^{{\rm in}}\big)}{r^n}+C_*\,c(n)\,r
 \\
 &\le&2\,\om_n+C(n)\,\tau+C(n,\tau)\,\s+C_*\,c(n)\,r\,,
\end{eqnarray*}
and then $\theta(x)\le 2$ by letting, in the order, $r\to 0^+$, $\s\to 0^+$ and then $\tau\to 0^+$. The proof of \eqref{coney island 1} and \eqref{coney island 1 star} is simple: since $x\in E^{(0)}$ and $\pa^*E\subset K$, by \eqref{theta 1 bordo nw} and by the divergence theorem we find that
\[
|E\cap B_{r_0}(x)\setminus V_{\s,r_0}|=0\,.
\]
In particular, by the coarea formula we find that for a.e. $r<r_0$,
\[
0=\lim_{j\to\infty}\H^n\Big((E_j\setminus V_{\s,r_0})\cap\pa S_{\tau,r}\Big)=
\lim_{j\to\infty}\H^n\Big(E_j\cap\big(\Gamma_{\tau,\s,r}^+\cup \Gamma_{\tau,\s,r}^-\big)\Big)\,,
\]
so that, by \eqref{slab decomp},
\begin{eqnarray*}
  \H^n(E_j\cap\pa S_{\tau,r})=\H^n(\pa S_{\tau,r}\cap\pa S_{\s,r})+{\rm o}(1)\le C(n)\,\s\,r^n+{\rm o}(1)\,,
\end{eqnarray*}
as $j\to\infty$, that is \eqref{coney island 1}, and
\begin{eqnarray*}
  \big|\H^n(\pa S_{\tau,r}\setminus E_j)-2\,\om_n\,r^n\big|&\le&\H^n(\pa S_{\tau,r}\cap\pa S_{\s,r})+
    \big|\H^n(\Gamma_{\tau,\s,r}^+\cup \Gamma_{\tau,\s,r}^-)-2\,\om_n\,r^n\big|+{\rm o}(1)
    \\
    &\le&C(n)\,\tau\,r^n+{\rm o}(1)
\end{eqnarray*}
as $j\to\infty$, that is \eqref{coney island 1 star}.

\medskip

\noindent {\it The case $x\in E^{(1)}$}: We claim that for every $\s\in(0,\tau)$,
\begin{eqnarray}
  \label{coney island 2}
  \limsup_{j\to\infty}\big|\H^n(E_j\cap \pa S_{\tau,r})-2\,\om_n\,r^n\big|\le C(n)\,\tau\,r^n\,,
  \\
  \label{coney island 2 star}
  \limsup_{j\to\infty}\H^n\big(\pa S_{\tau,r}\setminus E_j\big)\le C(n)\,\s\,r^n\,,
 \end{eqnarray}
 for a.e. $r<r_0(\s,x)$, see
 \begin{figure}
  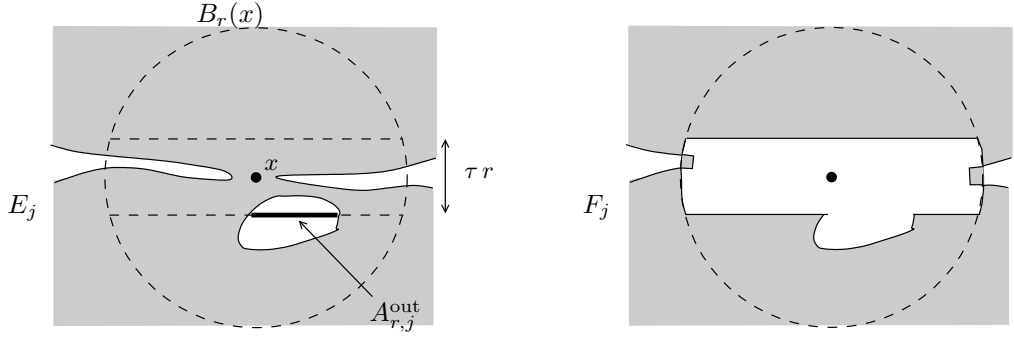\caption{{\small The slab competitor used in proving that $\theta(x)\le 2$ when $x\in E^{(1)}$ is the one defined by $A_{r,j}^{{\rm out}}$.}}\label{fig slab3}
 \end{figure}
 Figure \ref{fig slab3}. Indeed, by using as in the case $x\in\pa^*E$ the competitor defined by $A_{r,j}^{{\rm out}}$, \eqref{coney island 2}, \eqref{coney island 2 star} are combined with \eqref{area of slab competitors first proof} to obtain
 \begin{eqnarray}\nonumber
 \frac{\mu(B_r(x))}{r^n}&\le&\limsup_{j\to\infty}\frac{\H^n(E_j\cap\pa S_{\tau,r})}{r^n}
 \\\nonumber
 &&+C(n,\tau)\,\limsup_{j\to\infty}\frac{\H^n\big(\pa S_{\tau,r}\setminus(A_{r,j}^{{\rm out}}\cup E_j)\big)}{r^n}+C_*\,c(n)\,r
 \\
 &\le& 2\,\om_n+C(n)\,\tau+C(n,\tau)\,\s+C_*\,c(n)\,r\,,
 \end{eqnarray}
 which gives $\theta(x)\le 2$ by letting once again $r\to 0^+$, $\s\to 0^+$ and finally $\tau\to 0^+$. To prove \eqref{coney island 2} and \eqref{coney island 2 star}, we notice that by $x\in E^{(1)}$, $\pa^*E\subset K$, \eqref{theta 1 bordo nw} and the divergence theorem, we have
\[
\big|B_{r_0}(x)\setminus\big(V_{\s,r_0}\cup E\big)\big|=0\,.
\]
By the coarea formula, for a.e. $r<r_0$ we find
\[
0=\lim_{j\to\infty}\H^n\big(\big(\Gamma_{\tau,\s,r}^+\cup \Gamma_{\tau,\s,r}^-\big)\setminus E_j\big)\,,
\]
and conclude as in the previous case by exploiting \eqref{slab decomp}.

\medskip

\noindent {\it Remark}: We make an important remark on the constructions of step five, which will be needed in the proof of Theorem \ref{thm basic regularity}. We claim that, under the assumptions on $x$ considered in step five, for a.e. $r<r_0(\s,x)$ we have
\begin{eqnarray}
\label{the important remark}
  &&\limsup_{\eta\to 0^+}
  \Big|\H^n\Big(\Big\{y\in \cl(S_{\tau,r})\cap\pa F_j^\star:T_y(\pa F_j^{\star})=T_xK\Big\}\Big)-\theta(x)\,\om_n\,r^n\Big|
  \\\nonumber
  &&\hspace{3cm}\le C(n)\,\tau\,r^n+C(n,\tau)\,\s\,r^n+{\rm o}(1)\,,  \qquad\mbox{as $j\to\infty$}\,.
\end{eqnarray}
Here $\star={\rm out}$ if $x\in\pa^*E\cup(K\cap E^{(1)})$, $\star={\rm in}$ if $x\in K\cap E^{(0)}$, and $\theta(x)=1$ if $x\in\pa^*E$ and $\theta(x)=2$ if $x\in K\cap(E^{(0)}\cup E^{(1)})$. Consider, for example, the case when $x\in \pa^*E$. By \eqref{new slabs 2}, $\pa S_{\tau,r}\cap\pa F_j^{{\rm out}} \subset (\pa S_{\tau,r} \setminus A_{r,j}^{{\rm out}}) \cup N_j$ with $\lim_{\eta\to 0^+}\H^n (N_j) = 0$: thus, by taking into account that
\[
T_y(\pa F_j^{{\rm out}})=T_y(\pa S_{\tau,r})\qquad\mbox{$\H^n$-a.e. on $\pa F_j^{{\rm out}}\cap \pa S_{\tau,r}$}
\]
and that
\[
\big\{y\in\pa S_{\tau,r}:T_y(\pa S_{\tau,r})=T_xK\big\}= \pa S_{\tau,r}\setminus\pa B_r(x)\,,
\]
(recall that $T_xK=\nu(x)^\perp$), we have
\begin{eqnarray*}
  &&\Big|\H^n\Big(\Big\{y\in \cl(S_{\tau,r})\cap\pa F_j^{{\rm out}}:T_y(\pa F_j^{{\rm out}})=T_xK\Big\}\Big)-\om_n\,r^n\Big|
  \\
  &\le&\Big|\H^n\Big(\Big\{y\in \pa S_{\tau,r}\cap\pa F_j^{{\rm out}}  :T_y(\pa F_j^{{\rm out}})=T_xK\Big\}\Big)-\om_n\,r^n\Big|
  +\H^n(S_{\tau,r}\cap\pa F_j^{{\rm out}})
  \\
  &\le&\Big|\H^n\Big(\Big\{y\in \pa S_{\tau,r}\setminus A_{r,j}^{{\rm out}}:T_y(\pa S_{\tau,r})=T_xK\Big\}\Big)-\om_n\,r^n\Big|
  +\H^n (N_j) + \H^n(S_{\tau,r}\cap\pa F_j^{{\rm out}})
  \\
  &=&\Big|\H^n\big(\pa S_{\tau,r}\setminus(\pa B_r(x)\cup A_{r,j}^{{\rm out}})\big)-\om_n\,r^n\Big|
  +\H^n (N_j) + \H^n(S_{\tau,r}\cap\pa F_j^{{\rm out}})
\end{eqnarray*}
so that, by \eqref{new slabs 3}, \eqref{theta 1 fine 0}, and $\H^n(\pa S_{\tau,r}\cap\pa B_r(x))\le C(n)\,\tau\,r^n$,
\begin{eqnarray*}
&&
  \limsup_{\eta\to 0^+}\Big|\H^n\Big(\Big\{y\in \cl(S_{\tau,r})\cap\pa F_j^{{\rm out}}:T_y(\pa F_j^{{\rm out}})=T_xK\Big\}\Big)-\om_n\,r^n\Big|
  \\
  &\le&\Big|\H^n(\pa S_{\tau,r}\cap E_j)-\om_n\,r^n\Big|+ C(n,\tau)\H^n\big(\pa S_{\tau,r}\setminus(A_{r,j}^{{\rm out}}\cup E_j)\big)+C(n)\,\tau\,r^n\,.
\end{eqnarray*}
By \eqref{theta 1 fine 0} and \eqref{theta 1 fine} we deduce \eqref{the important remark} when $x\in\pa^*E$. The case when $x\in K\cap(E^{(0)}\cup E^{(1)})$ is treated analogously and the details are omitted.

\medskip

\noindent {\it Step six}: We exclude area concentration near $\pa\Om$, by showing that
\begin{equation}
  \label{no area concentration at the boundary}
  \limsup_{\eta\to 0^+}\limsup_{j\to\infty}\mu_j(\Om\cap U_\eta(\pa\Om))=0\,.
\end{equation}
Exploiting the smoothness and boundedness of $\pa\Om$, we can find $r_0>0$ such that Lemma \ref{lemma close by Lipschitz at boundary} holds, and such that for every $x\in\pa\Om$ there exists an open set $\Om'$ with $\Om\subset\Om'$ and a homeomorphism $f:\cl(\Om)\to\cl(\Om')=f(\cl(\Om))$ with $f(\pa\Om)=\pa\Om'$, $\{f\ne\id\}\cc B_{r_0}(x)$, $f(B_{r_0}(x)\cap\cl(\Om))=B_{r_0}(x)\cap\cl(\Om')$, which is a diffeomorphism $f:\Om\to\Om'$, and such that
\begin{equation}
  \label{boundary diffeo f}
  f\Big(\Om\cap U_\eta(\pa\Om)\cap B_{r_0/2}(x)\Big)\subset \Om'\setminus\Om\,,\qquad \|f-\id\|_{C^1(\Om)}\le C\,\eta\,;
\end{equation}
see
\begin{figure}
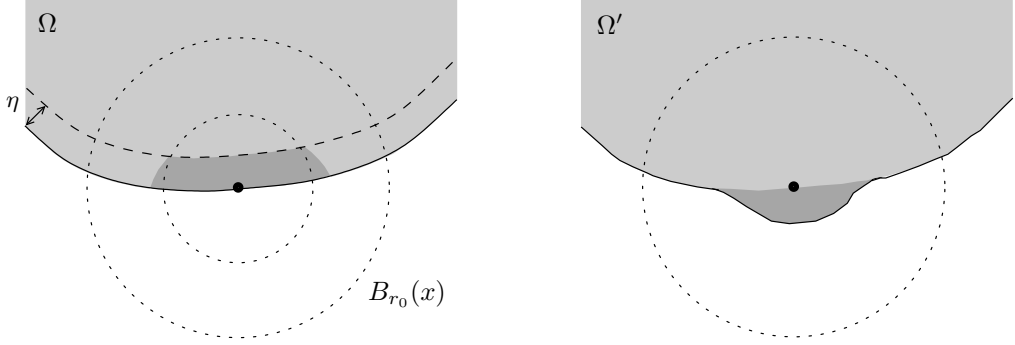
\caption{{\small The boundary diffeomorphism $f$ pushes out $\Om$ into a larger open set $\Om'$. Regions depicted with the same color are mapped one into the other. Notice that the dark region on the left contains $\Om\cap U_\eta(\pa\Om)\cap B_{r_0/2}(x)$, and is mapped outside of $\Om$. The diffeomorphism $f$ can be formally constructed by exploiting the local graphicality of $\Om$, and the simple details are omitted.}}
  \label{fig push}
\end{figure}
Figure \ref{fig push}. Let $\Om^*=f^{-1}(\Om)$ and let $F_j=f(E_j\cap\Om^*)=f(E_j)\cap\Om$. Clearly $F_j\in\E$, and $f(\pa\Om^*)=\pa\Om$ and $\Om^*\cap\pa(E_j\cap\Om^*)=\Om^*\cap\pa E_j$ give
\[
\Om\cap\pa F_j=f(\Om^*)\cap\,f\big(\pa(E_j\cap\Om^*)\big)=f\big(\Om^*\cap\pa E_j\big)\,,
\]
so that $\Om\cap\pa F_j$ is $\C$-spanning $W$ by Lemma \ref{lemma close by Lipschitz at boundary}. Assuming without loss of generality that $r_0<r_*$, by \eqref{volume fixing variation inequality}, $\{f\ne\id\}\cc B_{r_0}(x)$ and $f(B_{r_0}(x)\cap\cl(\Om))=B_{r_0}(x)\cap\cl(\Om')$ we have
\begin{eqnarray*}
\H^n(\Om\cap B_{r_0}(x)\cap \pa E_j)&\le&\H^n\big(f(B_{r_0}(x)\cap\Om^*\cap\pa E_j)\big)+C_*\,\big||F_j|-|E_j|\big|+\frac1j
\\
&\le&\big(1+C\,\eta\big)\,\H^n\big(B_{r_0}(x)\cap\Om^*\cap\pa E_j\big)+C_*\,\big||F_j|-|E_j|\big|+\frac1j\,,
\end{eqnarray*}
where
\begin{eqnarray*}
  \big||F_j|-|E_j|\big|\le\big||E_j\cap\Om^*|-|E_j|\big|+\int_{E_j\cap\Om^*}|Jf-1|
  \le\big|\Om\setminus\Om^*\big|+C\,\e\,\eta\le C\,\eta\,,
\end{eqnarray*}
so that
\begin{eqnarray*}
\H^n(\Om\cap B_{r_0}(x)\cap \pa E_j\setminus\Om^*)\le C\,\eta\,\Big\{\H^n\big(\Om\cap\pa E_j\big)+1\Big\}+\frac1j
\le C\,\eta\,\Big\{\psi(\e)+2 \Big\}+\frac1j\,.
\end{eqnarray*}
Since $\Om\cap U_\eta(\pa\Om)\cap B_{r_0/2}(x)\subset\Om\setminus\Om^*$, by letting $j\to\infty$ we conclude that
\[
\mu\big(B_{r_0/2}(x)\cap U_\eta(\pa\Om)\big)\le C\,\eta\,,\qquad\forall x\in\pa\Om\,.
\]
By a covering argument we find $\mu(\Om\cap U_\eta(\pa\Om))\le C\,\eta$, and thus \eqref{no area concentration at the boundary} follows.

\bigskip

\noindent {\it Step seven}: Let us now pick $R>0$ such that $W\cup K\cup E\cc B_R(0)$. If $E_j\subset B_{R+1}(0)$ for infinitely many values of $j$, then $|E|=\e$ and $\mu_j(\Om\setminus B_{R+1}(0))=0$, which combined with \eqref{no area concentration at the boundary} implies $\mu_j(\Om)\to\mu(\Om)=\F(K,E)$ as $j\to\infty$, and thus $\psi(\e)=\F(K,E)$ with $(K,E)\in\KK$ and $|E|=\e$: thus $(K,E)$ is a generalized minimizer of $\psi(\e)$, as desired. We now assume without loss of generality that $|E_j\setminus B_{R+1}(0)|>0$ for every $j$. By \eqref{almostthere},
\[
\limsup_{j\to\infty}|E_j\cap(B_{R+1}(0)\setminus B_R(0))|=\limsup_{j\to\infty}\H^n((B_{R+1}(0)\setminus B_R(0))\cap\pa E_j)=0\,.
\]
By the coarea formula, this implies that for a.e. $s\in(R,R+1)$,
\begin{equation}
  \label{infinity 1}
\limsup_{j\to\infty}\H^n(E_j\cap\pa B_s(0))=\limsup_{j\to\infty}\H^{n-1}(\pa E_j\cap\pa B_s(0))=0\,.
\end{equation}
We fix a value of $s$ such that \eqref{infinity 1} holds, and we let $A_j$ denote an $\H^n$-maximal connected component of $\pa B_s(0)\setminus\pa E_j$. It must be $A_j\cap E_j=\emptyset$: for, otherwise, by the spherical isoperimetric inequality, $A_j\subset E_j$ would imply
\begin{eqnarray*}
C(n)\,\H^{n-1}(\pa B_s(0)\cap \pa E_j)^{n/(n-1)}&\ge&\H^n(\pa B_s(0)\setminus A_j)\ge\H^n(\pa B_s(0)\setminus E_j)
\\
&\ge& c(n)\,R^n-\H^n(E_j\cap\pa B_s(0))\,,
\end{eqnarray*}
a contradiction to \eqref{infinity 1}. Since $A_j\cap E_j=\emptyset$, we can consider the exterior cup competitor defined by $E_j$ and $A_j$. More precisely, for every $j$ there exists a decreasing sequence $\{\eta^j_k\}_{k=1}^\infty$ with $\lim_{k \to \infty} \eta^j_k = 0$ such that, setting
\begin{eqnarray*}
Y_j = \pa B_s(0) \setminus \cl ((E_j \cap \pa B_s (0)) \cup A_j)\,, &\quad&  S_j = \pa E_j \cap \cl (A_j) \setminus \left( \cl ((E_j \cap \pa B_s(0)) \cup Y_j) \right)\,,\\
 U_{j,k} = \pa B_s (0) \cap \{{\rm d}_{S_j} < \eta^j_k\}\,,& \quad & Z_{j,k} = Y_j \cup \left( U_{j,k} \setminus \cl (E_j \cap \pa B_s (0)) \right)\,,
\end{eqnarray*}
the sets
\[
F_{j,k}=\big(E_j\cap B_s(0)\big)\cup M_{\eta^j_k}(Z_{j,k})
\]
satisfy $F_{j,k} \in \E$, with $\Om\cap\pa F_{j,k}$ $\C$-spanning $W$, $F_{j,k}\subset B_{R+1}$ and
\begin{eqnarray}  \label{infinity 2}
  \limsup_{k \to \infty}\H^n(\Om\cap\pa F_{j,k})&\le&\H^n(\Om\cap B_s(0)\cap\pa E_j)+2\,\H^n(\pa B_s(0)\setminus A_j)
  \\
  &\le&\nonumber
  \H^n(\Om\cap B_s(0)\cap\pa E_j)+C(n)\,\H^{n-1}(\pa B_s(0)\cap\pa E_j)^{n/(n-1)}\,.
  \end{eqnarray}
Since $|E_j\setminus B_{R+1}(0)|>0$ for every $j$, we can select $k(j)$ sufficiently large so that
\begin{equation} \label{infinity fix}
\H^n (\Om \cap \pa F_{j,k(j)}) \leq \H^n(\Om\cap B_s(0)\cap\pa E_j)+C(n)\,\H^{n-1}(\pa B_s(0)\cap\pa E_j)^{n/(n-1)} + \frac{1}{j}\,,
\end{equation}
as well as $|E_j\setminus B_s(0)|>|M_{\eta^j_{k(j)}}(Z_{j,k(j)})|$; then, after setting $F_j = F_{j,k(j)}$, define $\rho_j>0$ by the equation
\[
|B_{\rho_j}|=|E_j|-|F_j|=|E_j\setminus B_s(0)|-|M_{\eta^j_{k(j)}}(Z_{j,k(j)})|\,.
\]
In particular, $|B_{\rho_j}|\le\e$, so that we can find $x\in\Om$ such that $\cl(B_{\rho_j}(x))\cap\cl(F_j)=\emptyset$ and
\[
E_j^*=F_j\cup B_{\rho_j}(x)\subset B_{R+1+C(n)\,\e^{1/(n+1)}}(0)\qquad\forall j\,.
\]
We notice that $E_j^*\in\E$ with $|E_j^*|=\e$ and $\Om\cap\pa F_j\subset\Om\cap\pa E_j^*$, so that $\Om\cap\pa E_j^*$ is $\C$-spanning $W$: in particular, $\psi(\e)\le\H^n(\Om\cap\pa E_j^*)$. By the Euclidean isoperimetric inequality, and since $|B_{\rho_j}|\le |E_j\setminus B_s(0)|$ by definition of $\rho_j$, we have
\[
P(B_{\rho_j})\le P(E_j\setminus B_s(0))=\H^n(\pa E_j\setminus B_s(0))+\H^n(E_j\cap\pa B_s(0))\,,
\]
so that by \eqref{infinity 1} and \eqref{infinity fix} we get
\begin{eqnarray*}
\psi(\e)&\le&  \limsup_{j\to\infty}\H^n(\Om\cap\pa E_j^*)
\le\limsup_{j\to\infty}\H^n(\Om\cap\pa F_j)+P(B_{\rho_j})
\\
&\le&\limsup_{j\to\infty}\H^n(\Om\cap \pa E_j)+2\,C(n)\,\limsup_{j\to\infty}\H^{n-1}(\pa B_s(0)\cap\pa E_j)^{n/(n-1)}=\psi(\e)\,.
\end{eqnarray*}
We have thus proved that $\{E_j^*\}_j$ is a minimizing sequence for $\psi(\e)$, with $E_j^*\subset B_{R^*}(0)$ for some $R^*$ depending only on $R$, $n$ and $\e$. By repeating the argument of the first six steps with $E_j^*$ in place of $E_j$ we see that $E_j^*\to E^*$ in $L^1(\Om)$ and $\mu_j^*=\H^n\llcorner(\Om\cap\pa E_j^*)\weakstar \mu^*$ where $\mu^*=2\,\H^n\llcorner (K^*\setminus\pa^*E^*)+\H^n\llcorner\pa^*E^*$, and where $(K^*,E^*)\in\KK$ with $|E^*|=\e$ and with
\[
\limsup_{\eta\to 0^+}\limsup_{j\to\infty}\mu_j^*(\Om\cap\ U_\eta(\pa\Om))=0\,.
\]
Therefore $\mu_j^*(\Om)\to\mu^*(\Om)=\F(K^*,E^*)$ and in conclusion
\[
\F(K^*,E^*)=\mu^*(\Om)=\lim_{j\to\infty}\mu_j^*(\Om)=\psi(\e)
\]
so that, by $|E^*|=\e$, $(K^*,E^*)$ is indeed a generalized minimizer of $\psi(\e)$. This concludes the proof of the theorem.
\end{proof}

\section{The Euler-Lagrange equation: Proof of Theorem \ref{thm basic regularity}}\label{section theorem basic regularity}

\begin{proof}
  [Proof of Theorem \ref{thm basic regularity}] Let $(K,E)$ be a generalized minimizer of $\psi(\e)$ and $f:\Om\to\Om$ be a diffeomorphism such that $|f(E)|=|E|$. We want to prove that
 \begin{equation}
   \label{basic tesi}
   \F(K,E)\le\F(f(K),f(E))\,.
 \end{equation}
 Let $K'$ denote the set of points of approximate differentiability of $K$, so that $\H^n(K\setminus K')=0$, and for $x\in K'$ denote by $T_x=T_xK=\nu_x^\perp$ the approximate tangent plane to $K$ at $x$, where $\nu_x\in\SS^n$ is chosen so that $\nu_x=\nu_E(x)$ if $x\in \pa^*E$. As in step five of the proof of Theorem \ref{thm lsc}, for every $\s>0$ we introduce $r_0=r_0(\s,x)$ such that
 \begin{equation}
   \label{basic striscia 0}
    K\cap B_r(x)\subset S_{\s,r}^x=\Big\{y\in B_r(x):|(y-x)\cdot\nu_x|<\s\,r\Big\}\qquad\forall r<r_0(\s,x)\,,
 \end{equation}
 see \eqref{theta 1 bordo nw}. In fact, by Egoroff's theorem, we can find a compact set $K^*\subset K'$ with $\H^n(K\setminus K^*)<\s$ such that $r_*(\s)=\max\{r_0(\s,x):x\in K^*\}\to 0^+$ as $\s\to 0^+$, that is, such that \eqref{basic striscia 0} holds uniformly on $K^*$,
 \begin{equation}
   \label{basic striscia}
    K\cap B_r(x)\subset S_{\s,r}^x\qquad\forall x\in K^*\,,\forall r<r_*(\s)\,.
 \end{equation}
 Similarly, if $G_n$ denotes the family of the $n$-planes in $\R^{n+1}$, endowed with a distance $d$, by Lusin's theorem and up to further decreasing the size of $K^*$ while keeping $\H^n(K\setminus K^*)<\s$, we can make sure that
 \begin{equation}
   \label{basic omega}
    \sup_{x,y\in K^*\,,|x-y|<r}d(T_x,T_y)+\sup_{x,y\in K^*\,,|y-x|<r}|\nabla f(x)-\nabla f(y)|\le \om_*(r)\,,
 \end{equation}
 for a function $\om_*(r)\to 0^+$ as $r\to 0^+$. Finally, since
 \begin{equation*}
 \hspace{2.1cm}\left\{
 \begin{split}
 &\H^n(B_r(x)\cap\pa^*E)={\rm o}(r^n)\,,
 \\
 &\H^n\big(B_r(x)\cap(K\setminus\pa^*E)\big)=\om_n\,r^n+ {\rm o}(r^n)\,,\qquad\mbox{for $\H^n$-a.e. $x\in K\setminus\pa^*E$}\,,
 \end{split}
 \right .
 \end{equation*}
 \begin{equation*}
 \left\{
 \begin{split}
 &\H^n(B_r(x)\cap\pa^*E)=\om_n\,r^n+{\rm o}(r^n)\,,
 \\
 &\H^n\big(B_r(x)\cap(K\setminus\pa^*E)\big)= {\rm o}(r^n)\,, \qquad\mbox{for $\H^n$-a.e. $x\in \pa^*E$}\,,
 \end{split}
 \right .
 \end{equation*}
 as $r\to 0^+$, by Egoroff's theorem, up to decreasing $K^*$ and increasing $\om_*$, we can also entail
 \begin{eqnarray}\label{basic density 2}
   \sup_{x\in K^*\setminus\pa^*E}\H^n(B_r(x)\cap\pa^*E)+\Big|\H^n\big(B_r(x)\cap(K\setminus\pa^*E)\big)-\om_n\,r^n\Big|\le\om_*(r)\,r^n\,,
   \\\label{basic density 1}
   \sup_{x\in K^*\cap \pa^*E}\Big|\H^n(B_r(x)\cap\pa^*E)-\om_n\,r^n\Big|+\H^n\big(B_r(x)\cap(K\setminus\pa^*E)\big)\le\om_*(r)\,r^n\,,
 \end{eqnarray}
 while still keeping $\H^n(K\setminus K^*)<\s$ and $\om_*(r)\to 0^+$ as $r\to 0^+$.

 \medskip

 Let $\{E_j\}_j$ be a minimizing sequence for $\psi(\e)$ converging to $(K,E)$ as in \eqref{mininizing seq conv to gen minimiz}, and consider a point $x\in K^*$. Given $\tau\in(0,1)$ and $\s\in(0,\tau)$, for a.e. $r<r_*(\s)$ such that $B_{2r}(x)\cc\Om$, we have that $\pa S_{\tau,r}^x\cap\pa E_j$ is $\H^{n-1}$-rectifiable for every $j$ (with the exceptional set depending on $x$). For such values of $r$ and for every $\eta\in(0,r/2)$, we can set
 \[
 F_j^x=\left\{
 \begin{split}
   & F_j^{{\rm out}}\,,\qquad\mbox{if $x\in\pa^*E\cup(K^*\cap E^{(1)})$}\,,
   \\
   & F_j^{{\rm in}}\,,\qquad\hspace{0.2cm}\mbox{if $x\in K^*\cap E^{(0)}$}\,,
 \end{split}
 \right .
 \]
 with $F_j^{{\rm out}}$ and $F_j^{{\rm in}}$ defined as in step five of the proof of Theorem \ref{thm lsc}. In particular, $F_j^x\in\E$, $\Om\cap\pa F_j^x$ is $\C$-spanning $W$, $F_j^x\setminus\cl(S_{\tau,r}^x)=E_j\setminus\cl(S_{\tau,r}^x)$ and, as proved in \eqref{the important remark}, for a.e. $r<r_*(\s)$ we have
 \begin{eqnarray}
 \label{the important remark citato}
  &&\limsup_{\eta\to 0^+}
  \Big|\H^n\Big(\Big\{y\in \cl(S_{\tau,r}^x)\cap\pa F_j^x:T_y(\pa F_j^x)=T_x\Big\}\Big)-\theta(x)\,\om_n\,r^n\Big|
  \\\nonumber
  &&\hspace{3cm}\le C(n)\,\tau\,r^n+C(n,\tau)\,\s\,r^n+{\rm o}(1)  \qquad\mbox{as $j\to\infty$}\,,
 \end{eqnarray}
 where $\theta(x)=1$ if $x\in\pa^*E$ and $\theta(x)=2$ if $x\in K\cap (E^{(0)}\cup E^{(1)})$, as well as
 \begin{equation}
   \label{poca roba}
   \limsup_{j\to\infty}\limsup_{\eta\to 0^+}\H^n(S_{\tau,r}^x\cap\pa F_j^x)\le C(n,\tau)\,\s\,r^n\,,
 \end{equation}
 see \eqref{new slabs 3}, \eqref{theta 1 fine 0}, \eqref{coney island 1}, and \eqref{coney island 2 star}. By Besicovitch-Vitali's covering theorem and by Federer's theorem \eqref{federers theorem}, we can find a {\it finite} disjoint family of closed balls $\{B_i=\cl(B_{r_i}(x_i))\}_i$ such that $B_i\cc\Om$ and
 \begin{equation}
   \label{basic covering}
   \H^n\Big(K^*\setminus\bigcup B_{r_i}(x_i)\Big)<\s\,,\qquad x_i\in K^*\cap \big(E^{(0)}\cup E^{(1)}\cup\pa^*E\big)\,,\qquad r_i<r_*(\s)\,.
 \end{equation}
 We let $\eta<\min_i\{r_i/2\}$, define $F_j^{x_i}$ accordingly, and set
 \[
 S_i=S_{\tau,r_i}^{x_i}\cc B_i\,,\qquad T_i=T_{x_i}\,,\qquad F_j^i=F_j^{x_i}\,.
 \]
 Correspondingly, we define a sequence $\{F_j\}_j\subset\E$ with $\Om\cap\pa F_j$ $\C$-spanning $W$ by setting
 \begin{equation}
   \label{basic Fj Ej fuori da Bi}
 F_j\setminus \bigcup_i B_i=E_j\setminus\bigcup_i B_i\,,\qquad F_j\cap B_i=F_j^i\cap B_i\,.
 \end{equation}
 Since $F_j^i\setminus \cl(S_i)=E_j\setminus\cl(S_i)$ we find that
 \begin{equation}
   \label{basic Fj Ej fuori da Si}
    F_j\setminus \bigcup_i \cl(S_i)=E_j\setminus\bigcup_i\cl(S_i)\,,
 \end{equation}
 and, setting,
 \begin{equation}
   \label{theta i def}
    \theta_i=1\quad
 \mbox{if $x_i\in\pa^*E$}\,,\qquad
 \theta_i=2\quad \mbox{if $x_i\in E^{(0)}\cup E^{(1)}$}
 \end{equation}
 we deduce from \eqref{the important remark citato} and \eqref{poca roba} that, for each $i$,
 \begin{eqnarray}
 \label{basic Fj has right tangent plane}
 &&\limsup_{\eta\to 0^+}\Big|\H^n\big(\big\{y\in \cl(S_i)\cap\pa F_j:T_y(\pa F_j)=T_i\big\}\big)-\theta_i\,\om_n\,r_i^n\Big|
 \\\nonumber
 &&\hspace{5cm}\le C(n)\,\tau\,r_i^n+C(n,\tau)\,\s\,r_i^n+{\rm o}(1)
 \\\nonumber
 \\
 \label{basic Fj xi first kind}
 &&\hspace{1.4cm}\limsup_{\eta\to 0^+}\H^n(S_i\cap\pa F_j)\le C(n,\tau)\,\s\,r_i^n+{\rm o}(1)
\end{eqnarray}
as $j\to\infty$. Now let $C_*$ and $\e_*$ the volume-fixing variation constants defined by $f(E)$. By the monotonicity formula \eqref{monotonicity}, which can be applied to $B_{r_i}(x_i)$ as $x_i\in K$, we have
\begin{equation}\label{basic sum rin 0}
  e^{-\Lambda\,r_*(\s)}\,\theta_i\,\om_n\,r_i^n\le e^{-\Lambda\,r_i}\,\theta_i\,\om_n\,r_i^n\le \mu(B_{r_i}(x_i))=\mu(S_i)\,,
\end{equation}
where in the last identity we have used \eqref{basic striscia}, and where $\Lambda$ depends on $E$. By \eqref{basic sum rin 0}, $\theta_i \geq 1$, and $\mu = \theta \, \H^n \llcorner K$ with $\theta \leq 2$,
\begin{equation}
  \label{basic sum rin}
  \sum_i\,r_i^n\le C(n,E)\,\sum_i\H^n(K\cap B_i)\le C(n,E)\, \H^n(K)=C(n,E,K)\,,
\end{equation}
so that, by \eqref{basic Fj Ej fuori da Si}, $|S_i|\le C(n)\,\tau\,r_i^{n+1}$ and $r_i\le r_*(\s)\le 1$, we find
\[
|F_j\Delta E_j|\le\sum_i|S_i|\le C(n,E,K)\,\tau\,.
\]
Therefore,
\begin{eqnarray*}
  |f(F_j)\Delta f(E)|\le C\big(n,E,\Lip (f),\H^n(K)\big)\,\Big\{\tau+|E_j\Delta E|\Big\}<\e_*\,,
\end{eqnarray*}
provided $j$ is large enough and $\tau$ is small enough depending on $\e_*$. By the volume-fixing variations construction, for each $j$ large enough there exists a smooth map $\Phi_j:(-\e_*,\e_*)\times\R^{n+1}\to\R^{n+1}$, such that, for every $|v|<\e_*$, $\Phi_j(v,\cdot)$ is a diffeomorphism with $\Phi_j(v,\Om)=\Om$ and
\[
|\Phi_j(v,f(F_j))|=v+|f(F_j)|\,,\qquad \H^n\big(\Phi_j(v,\Sigma)\big)\le\H^n(\Sigma)+ C_*\,|v|\,\H^n(\Sigma)\,,
\]
for every $\H^n$-rectifiable set $\Sigma\subset\Om$. In particular, if we set
\[
G_j=\Phi_j(v_j,f(F_j))\,,\qquad v_j=|f(E)|-|f(F_j)|=|E|-|f(F_j)|\,,
\]
then we find that $G_j\in\E$, $|G_j|=|E|=\e$ and
\[
\H^n(\Om\cap\pa G_j)\le  \left( 1  +   C\big(n,E,\Lip(f),\H^n(K)\big)\,\Big\{\tau+|E_j\Delta E|\Big\} \right)    \H^n(\Om\cap\pa f(F_j))\,.
\]
Since $\Om\cap\pa F_j$ is $\C$-spanning $W$, so is $\Om\cap\pa G_j$ thanks to Lemma \ref{statement spanning is close by Lipschitz maps}, so that the minimizing sequence property of $E_j$ implies
\begin{equation}
  \label{basic Ej prima}
  \H^n(\Om\cap \pa E_j)\le \left( 1  +   C\,\Big\{\tau+|E_j\Delta E|\Big\} \right)    \H^n(\Om\cap\pa f(F_j)) + \frac1j\,,
\end{equation}
where, here and for the rest of the proof, $C$ is a generic constant depending on $K$, $E$, $f$ and $n$. We now claim that
\begin{equation}
  \label{basic Fj stima}
  \limsup_{\s\to 0^+}\limsup_{j\to\infty}\,\limsup_{\eta\to0^+}\H^n(\Om\cap\pa f(F_j))\le\F(f(K),f(E))+C\,\tau\,.
\end{equation}
Notice that by combining \eqref{basic Ej prima} and \eqref{basic Fj stima}, and by finally letting $\tau\to 0^+$, we complete the proof of \eqref{basic tesi}.

\medskip

To prove \eqref{basic Fj stima}, we notice that $f(\Om)=\Om$, $\Om\cap\pa f(F_j)=f(\Om\cap\pa F_j)$, and \eqref{basic Fj Ej fuori da Si} yield
\begin{eqnarray*}
  \H^n(\Om\cap \pa f(F_j))\le
  \H^n\Big(f\Big(\Om\cap\pa E_j\setminus\bigcup_i\cl(S_i)\Big)\Big)
 +\sum_i\int_{\cl(S_i)\cap\pa F_j}J^{\pa F_j}f\,d\H^n\,,
\end{eqnarray*}
where
\[
\limsup_{j \to \infty} \, \limsup_{\eta \to 0^+}\H^n\Big(f\Big(\Om\cap\pa E_j\setminus\bigcup_i \cl(S_i)\Big)\Big)
\le C\,\H^n\Big(K\setminus\bigcup_i S_i\Big)\le C\,\s
\]
by \eqref{basic striscia}, \eqref{basic covering}, and $\H^n(K\setminus K^*)<\s$. Hence, as
 \begin{eqnarray}\label{basic 1}
  \H^n(\Om\cap \pa f(F_j))\le\sum_i\int_{\cl(S_i)\cap\pa F_j}J^{\pa F_j}f\,d\H^n
  +C\,\s + {\rm o}(1)\,,
\end{eqnarray}
where ${\rm o}(1) \to 0^+$ if we let first $\eta \to 0^+$ and then $j \to \infty$.\\

If we set
\[
Z_i=\big\{y\in \pa S_i\cap\pa F_j:T_y(\pa F_j)=T_i\big\}\,,
\]
then by \eqref{basic Fj has right tangent plane} and \eqref{basic Fj xi first kind} we find
\begin{eqnarray*}
\H^n\big((\cl(S_i)\cap\pa F_j)\Delta Z_i\big)\le C(n)\,\tau\,r_i^n+C(n,\tau)\,\s\,r_i^n+{\rm o}(1)\,,
\\
\big|\H^n(Z_i)-\theta_i\,\om_n\,r_i^n\big|  \le C(n)\,\tau\,r_i^n+C(n,\tau)\,\s\,r_i^n+{\rm o}(1)\,,
\end{eqnarray*}
where ${\rm o}(1)\to 0^+$ if we let first $\eta\to 0^+$ and then $j\to\infty$. Also, it follows from \eqref{basic sum rin 0}, the characterization of $\mu$, and \eqref{basic density 1} that
\begin{equation} \label{basic sum rin 0 nuova}
e^{-\Lambda \, r_*(\s)} \, \theta_i \, \om_n \, r_i^n \leq \theta_i \, \H^n(S_i \cap K) + \om_*(r_i) \, r_i^n\,.
\end{equation}

By \eqref{basic omega}, \eqref{basic sum rin 0 nuova}, and $r_i < r_*(\s)$, we thus find
\begin{eqnarray}\nonumber
\int_{\cl(S_i)\cap\pa F_j}J^{\pa F_j}f&\le&
\int_{Z_i}J^{T_i}f+(\Lip\,f)^n\big\{C(n)\,\tau+C(n,\tau)\,\s\big\}\,r_i^n+{\rm o}(1)
\\\nonumber
&\le&
\theta_i\,\om_n\,r_i^n\,\big\{J^{T_i}f(x_i)+C(n)\,\om_*(r_i)\big\}
+C\,\big\{\tau+C(n,\tau)\,\s\big\}\,r_i^n+{\rm o}(1)
  \\\nonumber
  &\le&\Big\{J^{T_i}f(x_i)+\,C\,\Big(\om_*(r_*(\s))+\tau+C(n,\tau)\,\s\Big)\,\Big\}\times
  \\\nonumber
  &&\times\,\left( \theta_i \, \H^n(S_i\cap K)  + \om_*(r_*(\s)) \, r_i^n  \right)\,e^{\Lambda\,r_*(\s)}+{\rm o}(1)
  \\\label{basic 2}
  &=&J^{T_i}f(x_i)\big(\theta_i \, \H^n(S_i\cap K^*)+\a_i + \om_*(r_*(\s)) \, r_i^n\big)\,e^{\Lambda\,r_*(\s)}
  \\\nonumber
  &&+C\,\Big\{\om_*(r_*(\s))+\tau+C(n,\tau)\,\s\Big\}\,\left( \H^n(S_i\cap K)    +   \om_*(r_*(\s)) \, r_i^n   \right)\,e^{\Lambda\,r_*(\s)}
\\\nonumber
  &&+{\rm o}(1)\,,
\end{eqnarray}
where we have set
\begin{equation}
  \label{basic sum ai}
  \a_i=\theta_i \,\H^n(S_i\cap(K\setminus K^*))\quad\mbox{so that}\quad
  \sum_i\a_i<2\,\s\,.
\end{equation}
Now, again by \eqref{basic omega} we see that
\begin{eqnarray*}
\theta_i\,J^{T_i}f(x_i)\,\H^n(S_i\cap K^*)&\le&\theta_i\int_{S_i\cap K^*}J^Kf\,d\H^n+C(n)\,\om_*(r_i)\,\H^n(S_i\cap K^*)
\\
&=&\theta_i\,\H^n(f(S_i\cap K^*))+C(n)\,\om_*(r_i)\,\H^n(S_i\cap K^*)\,.
\end{eqnarray*}
By combining this last relation with \eqref{basic sum rin}, \eqref{basic 1}, \eqref{basic 2} and $r_i<r_*(\s)$, we find that
 \begin{eqnarray}\label{basic 3}
  \H^n(\Om\cap \pa f(F_j))&\le&e^{\Lambda\,r_*(\s)}\,\sum_i\theta_i\,\H^n(f(S_i\cap K^*))
  \\\nonumber
  &&+
C\,\Big\{\om_*(r_*(\s))+\tau+C(n,\tau)\,\s\Big\}\,e^{\Lambda\,r_*(\s)}+{\rm o}(1)\,,
\end{eqnarray}
with ${\rm o}(1)\to 0$ as first $\eta\to 0^+$ and then $j\to\infty$. If $x_i\in K^*\setminus\pa^*E$, then $\theta_i=2$ and by \eqref{basic density 2} we have
\begin{eqnarray*}
\theta_i\,\H^n(f(S_i\cap K^*))&\le& 2\,\H^n\big(f\big(S_i\cap (K^*\setminus\pa^*E)\big)\big)+2\,\Lip(f)^n\,\om_*(r_i)\,r_i^n
\\
&\le&2\H^n\big(f\big(S_i\cap (K\setminus\pa^*E)\big)\big)+C\,\om_*(r_*(\s))\,r_i^n\,;
\end{eqnarray*}
if, instead, $x_i\in \pa^*E$, then $\theta_i=1$ and \eqref{basic density 1} give
\begin{eqnarray*}
\theta_i\,\H^n(f(S_i\cap K^*))&\le& \H^n\big(f\big(S_i\cap K^*\cap \pa^*E\big)\big)+\Lip(f)^n\,\om_*(r_i)\,r_i^n
\\
&\le&\H^n\big(f(S_i\cap \pa^*E)\big)+C\,\om_*(r_*(\s))\,r_i^n\,;
\end{eqnarray*}
combining these last two estimates with \eqref{basic sum rin}, we find
\begin{eqnarray*}
\sum_i\theta_i\,\H^n(f(S_i\cap K^*))&\le& \sum_i\,2\,\H^n\big(f\big(S_i\cap (K\setminus\pa^*E)\big)\big)
+\H^n\big(f(S_i\cap \pa^*E)\big)
\\
&&+C\,\om_*(r_*(\s))\,\sum_i\,r_i^n
\\
&\le& \F\Big(f(K),f(E);\bigcup_if(S_i)\Big)+C\,\om_*(r_*(\s))\,,
\end{eqnarray*}
where $f(\pa^*E)=\pa^*f(E)$ by Lemma \ref{lemma redb}. Combining this last estimate with \eqref{basic 3} we find
 \[
 \H^n(\Om\cap \pa f(F_j))\le
 \,e^{\Lambda\,r_*(\s)}\,\Big\{\F(f(K),f(E))+
C\,\big\{\om_*(r_*(\s))+\tau+C(n,\tau)\,\s\big\}\Big\}+{\rm o}(1)\,,
\]
where ${\rm o}(1)\to 0$ as first $\eta\to 0^+$ and then $j\to\infty$; in particular, \eqref{basic Fj stima} holds.

\medskip

We conclude the proof. As explained, \eqref{basic Fj stima} implies \eqref{basic tesi}. By a classical first variation argument, see Appendix \ref{memme}, we deduce the existence of $\l\in\R$ such that
  \begin{equation}
    \label{basic stationary main}
    \l\,\int_{\pa^*E}X\cdot\nu_E\,d\H^n=\int_{\pa^*E}\Div^K\,X\,d\H^n+2\,\int_{K\setminus\pa^*E}\Div^K\,X\,d\H^n\,,
  \end{equation}
  for every $X\in C^1_c(\R^{n+1};\R^{n+1})$ with $X\cdot\nu_\Om=0$ on $\pa\Om$. Let us now consider the integer rectifiable varifold $V$ supported on $K$, with density $2$ on $K\setminus\pa^*E$ and $1$ on $\pa^*E$. By \eqref{basic stationary main}, we can compute the first variation of $V$ as
  \[
  \de V(X)=\int\,\vec{H}\cdot X\,d\,\|V\|\qquad\forall X\in C^1_c(\Om;\R^{n+1})
  \]
  where $\vec{H}=0$ on $K\setminus\pa^*E$ and $\vec{H}=\l\,\nu_E$ on $\pa^*E$. In particular, $\vec{H}\in L^\infty(\|V\|)$, and by Allard's regularity theorem \cite[Chapter 5]{SimonLN}, we have $K=\Sigma\cup{\rm Reg}$, where $\Sigma\subset K$ is closed and has empty interior in $K$, and where
 for every $x\in {\rm Reg}$ there exists a $C^{1,\a}$-function $u$ defined on $\R^n$ such that
  \begin{equation}
    \label{allard local graph}
      B_{r_x/2}(x)\cap K=B_{r_x/2}(x)\cap{\rm Reg}=B_{r_x/2}(x)\cap{\rm graph}(u)\,.
  \end{equation}
  By the divergence theorem, if $x\in{\rm Reg}\cap\pa E$, then, by \eqref{allard local graph} and by $\Om\cap\pa E\subset K$,
  \begin{eqnarray}\label{allard local epigraph}
    &&\mbox{$E={\rm epigraph}(u)$}\qquad\hspace{0.6cm}\mbox{inside $B_{r_x/2}(x)$}\,,
    \\\label{allard local red bound}
    &&\mbox{$K=\pa E={\rm graph}(u)$}\qquad\mbox{inside $B_{r_x/2}(x)$}\,,
  \end{eqnarray}
  which imply ${\rm Reg}\cap\pa E\subset\Om\cap\pa^*E$. Viceversa, if $x\in\Om\cap\pa^*E$, then $\H^n(B_r(x)\cap(K\setminus\pa^*E))={\rm o}(r^n)$ and $\H^n(B_r(x)\cap\pa^*E)=\om_n\,r^n+{\rm o}(r^n)$ as $r\to 0^+$, so that Allard's regularity theorem implies $\Om\cap\pa^*E\subset{\rm Reg}\cap\pa E$. Thus ${\rm Reg}\cap\pa E=\Om\cap\pa^*E$, and, in particular, $\Om\cap(\pa E\setminus\pa^*E)\subset\Sigma$, so that $\Om\cap(\pa E\setminus\pa^*E)$ has empty interior in $K$. Moreover, by \eqref{allard local epigraph}, \eqref{basic stationary main} implies that the graph of $u$ has constant mean curvature in $B_{r_x/2}(x)$, and thus that $\pa^*E$ is a smooth hypersurface, see e.g. \cite[Section 8.2]{GiMa}. Finally, \eqref{basic stationary main} implies that $K\setminus\pa E$ is the support of a multiplicity one stationary varifold in the open set $\Om\setminus\pa E$, so that $K\setminus(\Sigma\cup\pa E)$ is a smooth hypersurface with zero mean curvature, and $\H^n(\Sigma\setminus\pa E)=0$. The proof of Theorem \ref{thm basic regularity} is complete.
\end{proof}

\section{Convergence to Plateau's problem: Proof of Theorem \ref{thm convergence as eps goes to zero}}\label{section convergence to plateau} This section is devoted to showing that $\psi(\e)\to 2\,\ell$ as $\e\to 0^+$ and that a sequence $\{(K_h,E_h)\}_h$ of generalized minimizers for $\psi(\e_h)$ with $\e_h\to 0^+$ as $h\to\infty$ has to converge to a minimizer $S$ for Plateau's problem $\ell$ counted with multiplicity $2$ in the sense of Radon measures. If one could prove the latter assertion directly, then the former would follow at once by lower semicontinuity of weak-star converging Radon measures and by the upper bound $\psi(\e)\le 2\,\ell+C\,\e^{n/(n+1)}$ proved in \eqref{psi eps basic bounds}. A possible direct approach to the convergence of $(K_h,E_h)$ to a minimizer of Plateau's problem may be tried using White's compactness theorem \cite{whiteMULT2limit}. That would require proving an $L^1$-bound on the first variations of the varifolds $V_h$ supported on $K_h$ with density $1$ on $\Om\cap\pa^*E_h$ and with density $2$ on $K_h\setminus\pa^*E_h$. The validity of such bound is supported by the analysis of simple examples like  Example \ref{example two points} and Example \ref{example tripe sing}. However, Example \ref{example tripe sing} also indicates that when singularities are present in the limit Plateau minimizers $S$, then an $L^1$-bound for the mean curvatures of the varifolds $V_h$ would result from a quantitative balance between the rate of divergence towards $-\infty$ of the constant mean curvatures of the reduced boundaries $\pa^*E_h$, and the rate of vanishing of the areas $\H^n(\Om\cap\pa^*E_h)$. Validating a quantitative analysis of this kind in some generality would be of course very interesting per se as a way to describe the behavior of generalized minimizers; nonetheless, completing this analysis has so far eluded our attempts. Coming back to the proof of Theorem \ref{thm convergence as eps goes to zero}, we adopt a different approach. We prove directly that $\psi(\e)\to 2\,\ell$ as $\e\to 0^+$ by exploiting the same ``compactness-by-comparison'' strategy adopted in the proof of Theorem \ref{thm lsc}. An interesting point here is that because $|E_h|=\e_h\to 0^+$, we do not have a limit set that we can use to uniformly adjust volumes among local competitors of the elements of the minimizing sequence, and have to use a sort of ``absolute minimality at vanishing volumes'' of any sequence $\{(K_h,E_h)\}_h$ of generalized minimizers such that $\lim_{h\to\infty}\F(K_h,E_h)$ is equal to $\liminf_{\e\to 0^+}\psi(\e)$.

\begin{proof}[Proof of Theorem \ref{thm convergence as eps goes to zero}]
  {\it Step one}: We start proving that $\psi$ is lower semicontinuous on $(0,\infty)$. Given $\e_0>0$, let $\e_j\to\e_0>0$ as $j\to\infty$ be such that
  \[
  \lim_{j\to\infty}\psi(\e_j)=\liminf_{\e\to\e_0}\psi(\e)\,,
  \]
  and let $E_j\in\E$ be such that $|E_j|=\e_j$ and $\H^n(\Om\cap\pa E_j)\le\psi(\e_j)+1/j$. By \eqref{psi eps basic bounds}, $\psi(\e_j)$ is bounded in $j$, and thus by the compactness criteria for sets of finite perimeter and for Radon measures we have that, up to extracting subsequences, $\mu_j=\H^n\llcorner(\Om\cap\pa E_j)\weakstar \mu$ as Radon measures in $\Om$ and $E_j\to E$ in $L^1_{{\rm loc}}(\Om)$, where $\mu$ is a Radon measure in $\Om$, and where $E\subset\Om$ is a set of finite perimeter. We now repeat the proof of Theorem \ref{thm lsc}, with the only difference that while $|E_j|$ was constant in that proof, we know have that $|E_j|=\e_j\to\e_0$ for some $\e_0>0$. The modifications are minimal. In step two (nucleation of the sequence $E_j$), we repeat {\it verbatim} the argument, using the facts that $|E_j|\ge\e_0/2$ and that $\H^n(\Om\cap\pa E_j)\le 2\,\ell+C\,\e_0^{n/(n-1)}+1$ in place of $|E_j|=\e$ and $\H^n(\Om\cap\pa E_j)\le\psi(\e)+1$. Based on step two, in step three we construct volume-fixing variations with uniform constant $\e_*$ and $C_*$, and then repeat the rest of the argument without modifications. As a consequence, we can show that $\mu=\theta\,\H^n\llcorner K$ and $(K,E)\in\KK$ is a generalized minimizer of $\psi(\e_0)$, with
  \[
  \psi(\e_0)=\mu(\Om)=\lim_{j\to\infty}\mu_j(\Om)\leq\lim_{j\to\infty}\psi(\e_j)=\liminf_{\e\to\e_0}\psi(\e)\,,
  \]
  as claimed. The key information here is of course that $|E_j|\ge\e_0/2$ where $\e_0>0$. If $\e_0=0$, then the nucleation lemma is inconsequential, and the argument cannot be used.

  \medskip

  \noindent {\it Step two}: Thanks to \eqref{psi eps basic bounds}, to prove $\psi(\e)\to 2\,\ell$ as $\e\to 0^+$ we just need to show that
  \begin{equation}
    \label{psi eps lb 2 ell}
  \liminf_{\e\to 0^+}\psi(\e)\ge2\,\ell\,.
  \end{equation}
  To this end, we pick a sequence $\e_h\to 0^+$ such that
  \begin{eqnarray*}
    \label{ias1}
    \liminf_{\e\to 0^+}\psi(\e)=\lim_{h\to\infty}\psi(\e_h)\,.
  \end{eqnarray*}
  Notice that, in this way, given an arbitrary sequence $\s_h\to 0^+$, we have
  \begin{equation}
    \label{ias2}
      \limsup_{h\to\infty}\left[\psi(\e_h)-\psi(\s_h)\right]\le 0\,.
  \end{equation}
  Let $\{E_{h,j}\}_j$ be a minimizing sequence in $\psi(\e_h)$. By Theorem \ref{thm lsc}, there exists a generalized minimizer $(K_h,E_h)$ in $\psi(\e_h)$ such that, up to extracting subsequences,
  \begin{eqnarray*}
  &&E_{h,j}\to E_h\qquad\hspace{5.8cm}\mbox{in $L^1(\Om)$ as $j\to\infty$}\,,
  \\
  &&\mu_{h,j}=\H^n\llcorner(\Om\cap\pa E_{h,j})\weakstar \mu_h\qquad\hspace{2.8cm}\mbox{as Radon measures in $\Om$ as $j\to\infty$}\,,
  \\
  &&\mbox{$|E_{h,j}|=\e_h$ and}\,\, \H^n(\Om\cap\pa E_{h,j})\le \psi(\e_h)+\frac1j\,,\qquad\forall j\in\N\,,
  \end{eqnarray*}
  where, by \eqref{psi eps basic bounds} and up to extracting a further subsequence,
  \begin{eqnarray}\label{ias def muh}
    \mu_h=2\,\H^n\llcorner(K_h\setminus\pa^*E_h)+\H^n\llcorner(\Om\cap\pa^*E_h)\weakstar\mu\qquad\mbox{as Radon measures in $\Om$}
  \end{eqnarray}
  for some Radon measure $\mu$ in $\Om$. Given $x\in\Om \cap \spt\,\mu$, we set $d(x)=\dist(x,\pa\Om)$, and let
  \begin{equation}
    \label{ias5}
    H_{x,r}=\big\{h\in\N:|E_h\setminus B_r(x)|>0\big\}\,,\qquad I_x=\big\{r\in(0,d(x)):\mbox{$H_{x,r}$ is infinite}\big\}\,.
  \end{equation}
  We now look at local variations $F_{h,j}$ of $E_{h,j}$ such that $|F_{h,j}|$ has a positive limit volume $\s_h$ as $j\to\infty$, which in turn satisfies $\s_h\to 0^+$ as $h\to\infty$. The idea is that, by \eqref{ias2}, we will be able to use such variations to gather information on $\mu$.

  \medskip

  \noindent {\it Claim}: for every $r\in I_x$, if $\{F_{h,j}\}_{h\in H_{x,r},\,j \in \mathbb{N}}\subset\E$ is such that $\Om\cap\pa F_{h,j}$ is $\C$-spanning $W$ and $F_{h,j}\Delta E_{h,j}\subset\cl(B_r(x))$ for every $h\in H_{x,r}$ and every $j\in\N$, and if
  \begin{equation}
    \label{ias sh to zero}
    \exists\,\,\s_h=\lim_{j\to\infty}|F_{h,j}|>0\,,\qquad\mbox{and}\quad\lim_{h\in H_{x,r}\,,h\to\infty}\s_h=0\,,
  \end{equation}
  then
  \begin{equation}
    \label{ias7}
    \mu(B_r(x))\le \liminf_{h\in H_{x,r}\,,h\to\infty}\liminf_{j\to\infty}\H^n\big(\cl(B_r(x))\cap\pa F_{h,j}\big)\,.
  \end{equation}

  \medskip

  \noindent To prove this claim, we first notice that, for every $h\in H_{x,r}$,
  \begin{equation}
  \label{ias5starstar}
  \s_h=\lim_{j\to\infty}|F_{h,j}|\ge|E_h\setminus B_r(x)|>0\,.
  \end{equation}
  In particular, for $j$ large enough, $|F_{h,j}|>0$, $\psi(|F_{h,j}|)$ is well-defined, and $F_{h,j}$ is a competitor for $\psi(|F_{h,j}|)$, so that
  \begin{eqnarray}\nonumber
    \psi\big(|F_{h,j}|\big)&\le&\H^n(\Om\cap\pa F_{h,j})=\H^n(\cl(B_r(x))\cap\pa F_{h,j})+\H^n(\pa E_{h,j}\cap\Om\setminus\cl(B_r(x)))
    \\
    \nonumber
    &\le&\H^n(\cl(B_r(x))\cap\pa F_{h,j})+\psi(\e_h)+\frac1{j}-\H^n(\pa E_{h,j}\cap B_r(x))
  \end{eqnarray}
  which can be recombined into
  \[
  \mu_{h,j}(B_r(x))\le \H^n(\cl(B_r(x))\cap\pa F_{h,j})+\psi(\e_h)-\psi\big(|F_{h,j}|\big)+\frac1{j}\,.
  \]
  Letting $j\to\infty$, by $\mu_{h,j}\weakstar\mu_h$, $|F_{h,j}|\to \s_h>0$, and the lower semicontinuity of $\psi$ on $(0,\infty)$, we find that
  \[
  \mu_h(B_r(x))\le\liminf_{j\to\infty}\H^n(\cl(B_r(x))\cap\pa F_{h,j})+\psi(\e_h)-\psi(\s_h)\,.
  \]
  Since $\s_h\to 0^+$ as $h\to\infty$ with $h\in H_{x,r}$, by $\mu_h\weakstar\mu$ and \eqref{ias2} we deduce \eqref{ias7}, and thus prove the claim.

  \medskip

  \noindent {\it Step three}: We now fix $x\in\spt\mu$, set $f(r)=\mu(B_r(x))$, and prove that, for a.e. $r\in I_x$,
  \begin{eqnarray}\label{ias late 1}
    &&\left\{\begin{split}
    &\mbox{either $f'(r)\ge c(n)\,r^{n-1}$}\,,
    \\
    &\mbox{or $(f^{1/n})'(r)\ge c(n)$}\,,
    \end{split}
    \right .
    \\\label{ias9}
    &&f(r)\le \frac{r}n\,f'(r)\,.
  \end{eqnarray}
  By using the coarea formula together with $|E_h|\to 0$ as $h\to\infty$ and $E_{h,j}\to E_h$ as $j\to\infty$, we find that for a.e. $r<d(x)$,
  \begin{eqnarray}\label{ias rect}
    \mbox{$\pa E_{h,j}\cap\pa B_r(x)$ is $\H^{n-1}$-rectifiable}\,,
    \\\label{ias needed for}
    \lim_{j\to\infty}\H^n(E_{h,j}\cap\pa B_r(x))=\H^n(E_h\cap\pa B_r(x))\,,
    \\\label{ias needed for Ahj to zero}
    \lim_{h\to\infty}\lim_{j\to\infty}\H^n(E_{h,j}\cap\pa B_r(x))=0\,,
  \end{eqnarray}
  for every $h,j\in\N$. Moreover, if we set
  \[
  f_{h,j}(r)=\mu_{h,j}(B_r(x))\,,\qquad f_h(r)=\mu_h(B_r(x))\,.
  \]
  then, again by the coarea formula and by Fatou's lemma, for a.e. $r<d(x)$ we find
  \begin{equation}\label{ias derivate}
  \begin{split}
    \H^{n-1}(\pa E_{h,j}\cap\pa B_r(x))&\le f_{h,j}'(r)\,,
    \\
    g_h(r)=\liminf_{j\to\infty}f_{h,j}'(r)&\le f_h'(r)\,,
    \\
    g(r)=\liminf_{h\in H_{x,r},h\to\infty}f_h'(r)&\le f'(r)\,,
  \end{split}
  \end{equation}
  for every $h,j\in\N$. We first prove \eqref{ias late 1}. Let $r\in I_x$ be such that \eqref{ias rect}, \eqref{ias needed for}, \eqref{ias needed for Ahj to zero} and \eqref{ias derivate} hold, and let $A_{h,j}$ denote an $\H^n$-maximal connected component of $\pa B_r(x)\setminus\pa E_{h,j}$.  If $A_{h,j}\subset E_{h,j}$, then, by spherical isoperimetry, by \eqref{ias derivate}, and since the relative boundary to $A_{h,j}$ in $\pa B_r(x)$ is contained in $\pa B_r(x)\cap\pa E_{h,j}$, we find
  \begin{eqnarray*}
  f_{h,j}'(r)\ge c(n)\,\H^n(\pa B_r(x)\setminus A_{h,j})^{(n-1)/n}\,,
  \end{eqnarray*}
  where the lower bound converges to $c(n)\,r^{n-1}$ if we let first $j\to\infty$ and then $h\to\infty$ thanks to \eqref{ias needed for Ahj to zero}; hence, if $A_{h,j}\subset E_{h,j}$, the first alternative in \eqref{ias late 1} holds. We now assume that $A_{h,j}\cap E_{h,j}=\emptyset$, and consider the corresponding cup competitor $F_{h,j}$ as defined in Lemma \ref{lemma cup competitor first kind} starting from $E_{h,j}$, $A_{h,j}$. More precisely, if $\{\eta^{h,j}_k\}_{k=1}^\infty$ denotes the corresponding sequence as in \eqref{fix:coarea trick}, we choose $k(h,j)$ so that, setting
  \begin{eqnarray*}
  Y_{h,j} &=& \pa B_r(x) \setminus \cl ((E_{h,j} \cap \pa B_r (x)) \cup A_{h,j})\,, \\
 S_{h,j} &=& \pa E_{h,j} \cap \cl (A_{h,j}) \setminus \left( \cl ((E_{h,j} \cap \pa B_r(x)) \cup Y_{h,j}) \right)\,,
  \end{eqnarray*}
  we have that $\eta_j = \eta^{h,j}_{k(h,j)}$ satisfies $\eta_j\le r/2j$, with
  \begin{eqnarray}
  \label{sequence2}
  &&\H^n (\pa B_r (x) \cap \{{\rm d}_{S_{h,j}} \leq \eta_j\}) \leq \frac1{j}\,,
  \\\label{sequence3}
  &&\eta_j \,\H^{n-1} (\pa B_r (x) \cap \{{\rm d}_{S_{h,j}} = \eta_j\}) \leq\frac{1}j\,.
  \end{eqnarray}
  Then, with the usual notation
  \begin{eqnarray*}
 U_{h,j} = \pa B_r (x) \cap \{{\rm d}_{S_{h,j}} < \eta_j\}\,, &\qquad&  Z_{h,j} = Y_{h,j} \cup \left( U_{h,j} \setminus \cl (E_{h,j}\cap \pa B_r (x)) \right)\,,
\end{eqnarray*}
  we define
  \[
  F_{h,j}=\big(E_{h,j}\setminus\cl(B_r(x))\big)\cup N_{\eta_j}(Z_{h,j})\,.
  \]
  By Lemma \ref{lemma cup competitor first kind}, $F_{h,j}\in\E$, $\Om\cap\pa F_{h,j}$ is $\C$-spanning $W$ and $E_{h,j}\Delta F_{h,j}\subset \cl(B_r(x))$. Since $\eta_j\to 0$ as $j\to\infty$, we find
  \[
  \s_h=\lim_{j\to\infty}|F_{h,j}|=\lim_{j\to\infty}|E_{h,j}\setminus B_r(x)|=|E_h\setminus B_r(x)|\,,
  \]
  so that $\s_h>0$ if $h\in H_{x,r}$, and $\s_h\to 0^+$ if we let $h\to\infty$. Thus $F_{h,j}$ satisfies \eqref{ias sh to zero}, and we can apply \eqref{ias7} to $F_{h,j}$. To estimate the upper bound in \eqref{ias7}, we look back at \eqref{cup buccia eta level}, \eqref{fix rect 1}, \eqref{fix:est2}, and \eqref{fix:est3}, and find that
  \begin{equation}\label{ias festa}
  \begin{split}
  \H^n(\cl(B_r(x))\cap\pa F_{h,j})&\le(2+C(n)\,\eta_j)\,\H^n(\pa B_r(x)\setminus A_{h,j})
  \\
&+ (2 + C(n)\,\eta_j)\, \H^n (\pa B_r (x) \cap \{{\rm d}_{S_{h,j}} \leq  \eta_j\}) \\
  &+C(n)\,\eta_j\,\left(\H^{n-1}(\pa B_r(x)\cap\pa E_{h,j}) + \H^{n-1}(\pa B_r(x) \cap \{{\rm d}_{S_{h,j}} = \eta_j\})\right)\,.
  \end{split}
  \end{equation}
  By \eqref{ias7}, \eqref{sequence2}, \eqref{sequence3}, and \eqref{ias festa} we deduce that
  \begin{eqnarray}\nonumber
     f(r)=\mu(B_r(x))&\le&\liminf_{h\in H_{x,r}\,,h\to\infty}\liminf_{j\to\infty}\H^n\big(\cl(B_r(x))\cap\pa F_{h,j}\big)
     \\\label{ias played by}
     &\le&\liminf_{h\in H_{x,r}\,,h\to\infty}\liminf_{j\to\infty}
     2\,\H^n(\pa B_r(x)\setminus A_{h,j})
     \\\nonumber
     &\le&C(n)\,\liminf_{h\in H_{x,r}\,,h\to\infty}\liminf_{j\to\infty}
     f_{h,j}'(r)^{n/(n-1)}\le C(n)\,f'(r)^{n/(n-1)}\,,
  \end{eqnarray}
  We have thus proved that the second alternative in \eqref{ias late 1} holds, as claimed. We now prove \eqref{ias9}: let us now denote by $F_{h,j}$ the set defined by Lemma \ref{lemma cone competitor} as approximation of the cone competitor corresponding to $E_{h,j}$ in $B_r(x)$ with $\eta = \eta_j = r/2j$. We have that $F_{h,j} \in \E$ and that $\Om \cap \pa F_{h,j}$ is $\C$-spanning $W$; furthermore,
  by \eqref{cone competitor volume inequality} and \eqref{ias needed for} we find
  \[
  \s_h=\lim_{j\to\infty}|F_{h,j}|\geq|E_h\setminus B_r(x)|+\frac{r}{n+1}\,\H^n(E_h\cap \pa B_r(x))
  \]
  (in particular, $\s_h>0$ if $h\in H_{x,r}$) and, by \eqref{ias needed for Ahj to zero}, $\s_h\to 0^+$ as $h\to\infty$. Thus \eqref{ias sh to zero} holds, and we can deduce from \eqref{ias7} and \eqref{cone competitor area inequality} that
  \[
  f(r)=\mu(B_r(x))\le\liminf_{h\in H_{x,r}\,,h\to\infty}\liminf_{j\to\infty}\frac{r}n\,\H^{n-1}(\pa E_{h,j}\cap\pa B_r(x))\le \frac{r}n\,f'(r)\,,
  \]
  that is \eqref{ias9}.

  \medskip

  \noindent {\it Step four}: We now define a function $g:\Om\to(0,\infty)\cup\{-\infty\}$ by letting
  \begin{eqnarray*}
  g(x)&=&\sup\Big\{s>0:(0,s)\subset I_x\Big\}
  \\
  &=&\sup\Big\{t>0:\mbox{if $s<t$, then $|E_h\setminus B_s(x)|>0$ for infinitely many $h$}\Big\}\,.
  \end{eqnarray*}
  We notice that
  \begin{eqnarray}
    \label{g 1}
    &&\mbox{$g$ is lower semicontinuous on $\Om$}\,,
    \\
    \label{g 2}
    &&\mbox{$\{g=-\infty\}$ contains at most one point}\,.
  \end{eqnarray}
  (Notice that $\{g=-\infty\}$ may indeed contain one point: this is the case of the singular point of a triple junction, see Figure \ref{fig example23}-(b)). To prove \eqref{g 1}: if $g(x)\ne-\infty$, then $g(x)>0$, and for every $s\in(0,g(x))$, $|E_h\setminus B_s(x)|>0$ for infinitely many $h$. Thus, if $\eta\in(0,g(x))$ and $m_\eta$ is such that $|x-x_m|<\eta$ for every $m\ge m_\eta$, then, for every $m\ge m_\eta$ and $s\in(0,g(x)-\eta)$,
  \[
  |E_h\setminus B_s(x_m)|\ge |E_h\setminus B_{s+\eta}(x)|>0\,,\qquad\mbox{for intinitely many $h$}\,,
  \]
  that is $g(x)-\eta\le g(x_m)$ for every $m\ge m_\eta$; this proves \eqref{g 1}. Next, if $g(x_1)=g(x_2)=-\infty$, then for every $s>0$ there exists $h(s)$ such that
  \[
  |E_h\setminus B_s(x_1)|=  |E_h\setminus B_s(x_2)|=0\qquad\forall h\ge h(s)\,.
  \]
  If $x_1\ne x_2$ we can take $s=|x_1-x_2|/2$ and deduce $|E_h|=0$; thus \eqref{g 2} holds. Let us now consider the open set $\{g>s\}\subset\Om$, $s>0$, and set
  \[
  Z(s)=\spt\mu\cap\{g>s\}\,,\qquad Z=\spt\mu\cap\{g>0\}\,.
  \]
  We claim that if $x\in Z(s)$, then
  \begin{equation}
    \label{ias replaced}
      f(r)\ge c_0(n)\,r^n\qquad\forall r\in(0,s)\,,\qquad\mbox{$r^{-n}\,f(r)$ is increasing over $r\in(0,s)$}\,.
  \end{equation}
  The second assertion is immediate from \eqref{ias9}. To prove the first one, set
  \[
  L_1=\big\{r\in(0,s):f'(r)\ge c(n)\,r^{n-1}\big\}\,,\qquad L_2=(0,s)\setminus L_1\,,
  \]
  with $c(n)$ as in \eqref{ias late 1}. If $x\in Z(s)$ is such that $\H^1(L_1)\ge s/2$, then for every $r\in(0,s)$
  \begin{eqnarray*}
    f(r)\ge\int_{L_1\cap(0,r)}f'\ge c(n) \int_{L_1\cap(0,r)}t^{n-1}\,dt\ge c(n)\,\int_0^{\min\{r,s/2\}}\,t^{n-1}\,dt\ge \frac{c(n)}{n\,2^n}\,r^n\,;
  \end{eqnarray*}
  if instead $\H^1(L_2)\ge s/2$, then for every $r\in(0,s)$,
  \begin{eqnarray*}
    f(r)^{1/n}\ge\int_{L_2\cap(0,r)}(f^{1/n})'\ge c(n)\, \H^1(L_2\cap(0,r))\ge c(n)\,\min\big\{r,\frac{s}2\big\}\ge\frac{c(n)}2\,r\,,
  \end{eqnarray*}
  where we have used the fact that, by \eqref{ias late 1}, we have $(f^{1/n})'\ge c(n)$ on $L_2$. Thanks to \eqref{ias replaced} we are in the position of using \cite[Theorem 6.9]{mattila} and Preiss' theorem (as done in step four of the proof of Theorem \ref{thm lsc}) on each $Z(s)$, to find that $Z$ is $\H^n$-rectifiable with
  \begin{equation}
    \label{ias101}
      \mu\llcorner Z=\theta\,\H^n\llcorner Z\,,
  \end{equation}
  where the density
  \[
  \theta(x)=\lim_{r\to 0^+}\frac{\mu(B_r(x))}{\om_n\,r^n}\quad\mbox{exists in $[c_0(n),\infty)$ for every $x\in Z$}\,.
  \]
  Moreover, by \eqref{g 2},
  \begin{equation}
    \label{ias10}
    \H^0(\spt\mu\setminus Z)\le 1\,.
  \end{equation}
 By combining \eqref{ias101} and \eqref{ias10} we find that $K=\spt\mu$ is $\H^n$-rectifiable and such that $\mu=\theta\,\H^n\llcorner K$. Since $K_h=\spt\,\mu_h$ is $\C$-spanning $W$ and $\mu_h\weakstar\mu$, by Lemma \ref{statement K spans} we find that $K$ is $\C$-spanning $W$, and thus  admissible in $\ell$, so that
  \begin{equation}
    \label{ias12}
      \liminf_{\e\to 0^+}\psi(\e)
  =\lim_{h\to\infty}\mu_h(\Om)\ge\mu(\Om)=\int_K\,\theta\,d\H^n\ge\,\min_K\,\theta\,\H^n(K)\ge\,\ell\,\min_K\,\theta\,.
  \end{equation}
  Thus, to complete the proof of \eqref{psi eps lb 2 ell} we just need to show that
  \begin{equation}
    \label{ias11}
    \mbox{$\theta\ge2$ $\H^n$-a.e. on $K$}\,.
  \end{equation}
  Since $\mu_{h,j}=\H^n\llcorner(\Om\cap\pa E_{h,j})\weakstar\mu_h$ as $j\to\infty$, with $\mu_h\weakstar \theta\,\H^n\llcorner K$ as $h\to\infty$, we can extract a diagonal subsequence $j=j(h)$ so that, denoting $E_h^* = E_{h,j(h)}$, $\{E_h^*\}_h\subset\E$, $\Om\cap\pa E_h^*$ $\C$-spanning $W$, and
  \[
  \mu_h^*=\H^n\llcorner(\Om\cap\pa E_h^*)\weakstar \theta\,\H^n\llcorner K\,,\qquad\mbox{as $h\to\infty$}\,.
  \]
  Moreover, $\mu(B_r(x))\ge c(n)\,r^n$ for every $r\in(0,s)$ if $x\in K\cap\{g>s\}$ and, thanks to \eqref{ias festa},
  \[
  \liminf_{h\to\infty}\H^n(B_r(x)\cap\pa E_h^*)\le C(n)\,\liminf_{h\to\infty}\,\H^n(\pa B_r(x)\setminus A_{r,h}^0)\,,
  \]
  where $A_{r,h}^0$ denotes an $\H^n$-maximal connected component of $\pa B_r(x)\setminus\pa E_h^*$, this time for every $x\in K$ and $B_r(x)\cc\Om$. We can thus apply Lemma \ref{lemma llb} with the open set $\Om'=\{g>s\}$ to deduce that
  \[
  \mbox{$\theta\ge 2$ $\H^n$-a.e. on $\{g>s\}\cap K\setminus \pa^*E^*$}
  \]
  where $E^*=\emptyset$ is the $L^1$-limit of the sets $E_h^*$. Since $\pa^*E^*=\emptyset$, taking the union over $s>0$ and recalling \eqref{ias10}, we conclude that \eqref{ias11} holds.

  \medskip

  \noindent {\it Step five}: Now that $\psi(\e)\to 2\,\ell$ as $\e\to 0^+$ has been proved, let $(K_h,E_h)$ be a sequence of generalized minimizers of $\psi(\e_h)$ for an arbitrary sequence $\e_h\to 0^+$. Since the limit of $\psi(\e)$ as $\e\to0^+$ exists, $\e_h$ automatically satisfies \eqref{ias1}, and the arguments of step two to four can be repeated {\it verbatim}. Correspondingly, up to extracting subsequences, \eqref{ias def muh} holds with $\mu=\theta\,\H^n\llcorner K$, $\theta\ge 2$ $\H^n$-a.e. on $K$, and $K$ a relatively compact subset of $\Om$, $\H^n$-rectifiable, and $\C$-spanning $W$. By plugging $\psi(\e)\to 2\,\ell$ as $\e\to 0^+$ in \eqref{ias12}, we find that $\theta=2$ $\H^n$-a.e. on $K$, $2\,\H^n(K)=2\,\ell$, so that $K$ is a minimizer of $\ell$, and thus, looking back at \eqref{ias def muh}, we conclude that \eqref{weak star convergence of gen minimizers} holds.
  \end{proof}

\appendix

\section{A technical fact on sets of finite perimeter}

\begin{lemma}\label{lemma redb}
  If $\Om$ is an open set in $\R^{n+1}$, $E$ is a set of finite perimeter in $\Om$, and $f:\Om\to\Om$ is a diffeomorphism, then $f(E)$ is a set of finite perimeter in $\Om$ with $\pa^*f(E)=f(\pa^*E)$ and
  \begin{equation}
    \label{redb 0}
      \nu_{f(E)}(f(x))=\frac{(\nabla g(f(x)))^T\nu_E(x)}{|(\nabla g(f(x)))^T\nu_E(x)|}\qquad\forall x\in\pa^*E\,,
  \end{equation}
  where $g=f^{-1}$.
\end{lemma}

\begin{proof}
  In \cite[Proposition 17.1, Remark 17.2]{maggiBOOK} it is shown that $f(E)$ is a set of finite perimeter with
  \[
  \mu_{f(E)}=f_\#\Big(Jf\,(\nabla g(f))^T\,\mu_E\Big)\,.
  \]
  and that mapping by $f$ preserves essential boundaries (thus just the $\H^n$-equivalence of $\pa^*f(E)$ and $f(\pa^*E)$ is deduced there). In order to prove $\pa^*f(E)=f(\pa^*E)$, we pick a ball $B_r(f(x))\cc\Om$, and look at
  \begin{eqnarray}\nonumber
    \frac{\mu_{f(E)}(B_r(f(x)))}{|\mu_{f(E)}|(B_r(f(x)))}&=&
    \frac{\int_{g(B_r(f(x)))\cap\pa^*E}Jf\,\nabla g(f)^T\nu_E\,d\H^n}
    {\int_{g(B_r(f(x)))\cap\pa^*E}Jf\,|\nabla g(f)^T\nu_E|\,d\H^n}
    \\
    &=&\label{redb 1}
    \frac{\int_{(\pa^*E-x)/r}u_r(z)\nu_E(x+r\,z)\,d\H^n_z}
    {\int_{(\pa^*E-x)/r}|u_r(z)\,\nu_E(x+r\,z)|\,d\H^n_z}\,.
  \end{eqnarray}
  where we have set
  \[
  F_r=\frac{g(B_r(f(x)))-x}r\,,\qquad u_r(z)=1_{F_r}(z)\,Jf(x+r\,z)\,\nabla g(f(x+r\,z))^T\,.
  \]
  If we set $F=L(B_1(0))$ for the linear map $L=\nabla g(f(x))$, then for every $\e>0$ we have
  \begin{eqnarray*}
  L(B_{1-\e}(0))\subset F_r\subset L(B_{1+\e}(0))\qquad\mbox{provided $r$ is small enough}\,,
  \end{eqnarray*}
  and thus, as $r\to 0^+$,
  \[
  1_{F_r}\to 1_F\qquad \mbox{uniformly on $\R^{n+1}\setminus X_\e$}\,,
  \]
  where we have set
  \[
  X_\e= L(B_{1+\e}(0)\setminus B_{1-\e}(0))\,.
  \]
  Since $F_r,F\subset B_{\Lip g}(0)$, and since for every $R>0$
  \[
  Jf(x+r\,z)\,\nabla g(f(x+r\,z))^T\to Jf(x)\,\nabla g(f(x))^T\qquad\mbox{uniformly on $|z|\le R$}\,,
  \]
  as $r\to 0^+$, we conclude that
  \begin{equation}
    \label{lemma pf uniform ur to u}
      u_r(z)\to u(z):=1_F(z)\,Jf(x)\,\nabla g(f(x))^T\qquad\mbox{uniformly on $\R^{n+1}\setminus X_\e$}\,.
  \end{equation}
  We now decompose the integrals over $(\pa^*E-x)/r$ appearing in \eqref{redb 1} through $X_\e$. By \eqref{lemma pf uniform ur to u},
  \begin{eqnarray*}
    &&\Big|\int_{[(\pa^*E-x)/r]\setminus X_\e}u_r(z)\nu_E(x+r\,z)\,d\H^n_z-\int_{[(\pa^*E-x)/r]\setminus X_\e}u(z)\nu_E(x+r\,z)\,d\H^n_z\Big|
    \\
    \\
    &\le&\om(r)\,\H^n(B_{\Lip g}(0)\cap[(\pa^*E-x)/r]\setminus X_\e)
    \\
    &\le&\om(r)\,P(E;B_{r\,\Lip g}(x))\to 0
  \end{eqnarray*}
  as $r\to 0^+$, while $x\in\pa^*E$ gives
  \[
  \lim_{r\to 0^+}\int_{[(\pa^*E-x)/r]\setminus X_\e}u(z)\nu_E(x+r\,z)\,d\H^n_z
  =\int_{T_x(\pa^*E)\setminus X_\e}u(z)\nu_E(x)\,d\H^n_z\,.
  \]
  At the same time, since $|u_r|\le C$ for a constant $C$ independent of $r$, we have
  \[
  \Big|\int_{X_\e\cap[(\pa^*E-x)/r]}u_r(z)\nu_E(x+r\,z)\,d\H^n_z\Big|\le C\,\H^n\Big(X_\e\cap\frac{\pa^*E-x}r\Big)\to C\,\H^n(X_\e\cap T_x(\pa^*E))
  \]
  as $r\to 0^+$. Combining the above estimates with $|u|\le C$ we finally find
  \begin{eqnarray*}
    &&\limsup_{r\to 0^+}\Big|\int_{(\pa^*E-x)/r}u_r(z)\nu_E(x+r\,z)\,d\H^n_z-\int_{T_x(\pa^*E)}u(z)\nu_E(x)\,d\H^n_z\Big|
    \\
    &\le&C\,\H^n(X_\e\cap T_x(\pa^*E))\,,\qquad\forall\e>0\,.
  \end{eqnarray*}
  Letting $\e\to 0^+$, we find $\H^n(X_\e\cap T_x(\pa^*E))\to \H^n(X\cap T_x(\pa^*E))$ where
  \[
  X=L(\pa B_1(0))\,.
  \]
  Since $L$ is invertible, $L(\pa B_1(0))$ intersects transversally any plane through the origin, and in particular $T_x(\pa^*E)$. Therefore $\H^n(X\cap T_x(\pa^*E))=0$ and we have proved
  \begin{eqnarray*}
  \lim_{r\to 0^+}\int_{(\pa^*E-x)/r}u_r(z)\nu_E(x+r\,z)\,d\H^n_z&=&\int_{T_x(\pa^*E)}u(z)\,\nu_E(x)\,d\H^n_z
  \\
  &=&Jf(x)\,L^T\,\nu_E(x)\,\H^n(F\cap T_x(\pa^*E))\,.
  \end{eqnarray*}
  An analogous argument shows
  \[
  \lim_{r\to 0^+}\int_{(\pa^*E-x)/r}|u_r(z)\nu_E(x+r\,z)|\,d\H^n_z
  =Jf(x)\,|L^T\,\nu_E(x)|\,\H^n(F\cap T_x(\pa^*E))\,,
  \]
  and finally we conclude that if $x\in\pa^*E$, then
  \begin{eqnarray*}
    \lim_{r\to 0^+}\frac{\mu_{f(E)}(B_r(f(x)))}{|\mu_{f(E)}|(B_r(f(x)))}=\frac{L^T\,\nu_E(x)}{|L^T\,\nu_E(x)|}\in\SS^n\,.
  \end{eqnarray*}
  In particular, $f(x)\in\pa^*f(E)$ and \eqref{redb 0} holds.
\end{proof}

\section{Boundary density estimates for the Harrison--Pugh minimizers}\label{appendix hp lower density} In this appendix we prove that when $\pa W$ is smooth and $\ell<\infty$, then every minimizer $S$ of $\ell$ satisfies uniform lower density estimates up to the boundary of $\Omega$.

\begin{theorem}\label{thm density for S fine}
  If $\ell<\infty$, $\pa W$ is smooth, and $S$ is a minimizer of $\ell$, then
  \begin{equation}
    \label{hp S lower density estimates}
    \H^n(B_r(x)\cap S)\ge c(n)\,r^n\,,\qquad\forall x\in\cl(S)\,,r\in(0,r_0)\,,
  \end{equation}
  for a value of $r_0$ depending on $W$.
\end{theorem}

\begin{proof}
  By Lemma \ref{lemma close by Lipschitz at boundary}, and since $S$ minimizes $\H^n$ with respect to every relatively closed subset of $\Om$ which is $\C$-spanning $W$, recall \eqref{plateau problem}, we have
  \begin{equation}
    \label{hp diffeo min}
    \H^n(S)\le \H^n(f(S))
  \end{equation}
  whenever $f:\cl(\Om)\to\cl(\Om)$ is a homeomorphism with $f(\pa\Om)=\pa\Om$, $\{f\ne\id\}\cc B_{r_0}(x)$ for $x\in\pa \Om$, and $f(B_{r_0}(x)\cap\cl(\Omega))=B_{r_0}(x)\cap\cl(\Omega)$ for $r_0$ depending on $W$. We immediately deduce from \eqref{hp diffeo min}, that
  \begin{equation}
    \label{hp S stazionario al bordo}
    \int_S\Div^SX\,d\H^n=0
  \end{equation}
  for every $X\in C^1_c(B_{r_0}(x);\R^{n+1})$ with $X\cdot\nu_\Om=0$ on $\pa\Om$. Since $S$ is an Almgren minimizer in $\Omega$,   \eqref{hp S stazionario al bordo} also holds for every $X\in C^1_c(\Om;\R^{n+1})$. Finally, we deduce the validity of \eqref{hp S stazionario al bordo} for every $X\in C^1_c(\R^{n+1};\R^{n+1})$ with $X\cdot\nu_\Om=0$ on $\pa\Om$ by a standard covering argument.

  \medskip

  The validity of \eqref{hp S stazionario al bordo} for every $X\in C^1_c(\R^{n+1};\R^{n+1})$ with $X\cdot\nu_\Om=0$ on $\pa\Om$ is a distributional formulation of Young's law, which has been extensively studied in the classical work of Gr\"uter and Jost \cite{gruterjost}, and has been recently extended to arbitrary contact angles by Kagaya and Tonegawa \cite{kagayatone}. The main consequence of \eqref{hp S stazionario al bordo} we shall need here is an adapted monotonicity formula which takes care of the local geometry of $\pa W$. We now introduce this tool and then complete the proof.

  \medskip

  Let $r_0$ be sufficiently small, so that $I_{r_0}(\pa W)$ admits a well-defined nearest point projection map $\Pi \colon I_{r_0}(\pa W) \to \pa W$ of class $C^1$. By \cite[Theorem 3.2]{kagayatone}, there exists a constant $C = C(n,r_0)$ such that for any $x \in I_{r_0/6}(\pa W) \cap {\rm cl}(\Om)$ the map
\begin{equation} \label{monotonicity_formula}
r \in \left( 0, r_0/6 \right) \mapsto \frac{\H^n(S\cap B_r(x)) +\H^n(S\cap \tilde B_r(x))}{\omega_n \, r^n} \,e^{C\,r}
\end{equation}
is increasing, where
\begin{equation} \label{reflection}
\tilde B_r(x) := \left\lbrace y \in \R^{n+1} \, \colon \, \tilde y \in B_r(x) \right\rbrace \,, \qquad \tilde y := \Pi(y) + (\Pi(y) - y)
\end{equation}
denotes a sort of nonlinear reflection of $B_r(x)$ across $\pa W$. In particular, the limit
\begin{equation} \label{boundary_density}
\s(x)=\lim_{r\to 0^+} \frac{\H^n(S\cap B_r(x)) +\H^n(S\cap \tilde B_r(x))}{\omega_n \, r^n}
\end{equation}
exists for every $x \in I_{r_0/6}(\pa W) \cap {\rm cl}(\Om)$ , and the map $x \mapsto \sigma(x)$ is upper semicontinuous in there; see \cite[Corollary 5.1]{kagayatone}.

 \medskip

 Next, we recall from \cite[Lemma 4.2]{kagayatone} a simple geometric fact: if $x \in I_{r_0}(\pa W)$, and $\rho > 0$ is such that $\dist(x,\pa W) \leq \rho$ and $B_{\rho}(x) \subset I_{r_0}(\pa W)$, then
\begin{equation} \label{ovvio2}
\tilde B_{\rho}(x) \subset B_{5\rho}(x)\,.
\end{equation}
We are now in the position to prove \eqref{hp S lower density estimates}. First of all we recall that, since $S$ defines a multiplicity one stationary varifold in $\Omega$, we have
\begin{equation} \label{interior_dlb}
\H^n(S\cap B_r(x))\ge \om_n\,r^n\,,\qquad\forall x\in S\,,B_r(x)\cc\Om\,.
\end{equation}
In particular, \eqref{hp S lower density estimates} holds with $c=\omega_n$ for all $x \in S \setminus I_{r_0/6}(\pa W)$ as soon as $r <r_0/6$. Therefore we can assume that
\begin{equation}
  \label{x once}
x \in \cl(S) \cap I_{r_0/6}(\pa W)\,.
\end{equation}
We first notice that we have $\s(x)\ge 1$: by upper semicontinuity of $\s$ on $\cl(S) \cap I_{r_0/6}(\pa W)$ we just need to show this when, in addition to \eqref{x once}, we have  $x\in S$, and indeed in this case,
\[
\s(x) \geq \lim_{\rho\to 0^+} \frac{\H^n(S\cap B_\rho(x))}{\omega_n\, \rho^n}\ge1
\]
thanks to \eqref{interior_dlb}; this proves $\s(x)\ge1$. Now we fix $r<5\,r_0/6$ and distinguish two cases depending on the validity of
\begin{equation}
  \label{x validity}
  \dist(x,\pa W)> \frac{r}5\,.
\end{equation}
If \eqref{x validity} holds, then by \eqref{interior_dlb}
\[
\H^n(S \cap B_r(x)) \geq \H^n(S \cap B_{r/5}(x)) \ge \omega_n \, \Big(\frac{r}{5} \Big)^n\,,
\]
thus proving \eqref{hp S lower density estimates}. If $\dist(x,\pa W)\le r/5$, then, thanks to the obvious inclusion $B_r(x) \subset I_{r_0}(\pa W)$,  we can apply \eqref{ovvio2} with $\rho=r/5$ to find $\tilde B_{r/5}(x) \subset B_r(x)$. In this way, by exploiting $\s(x)\ge1$ and \eqref{monotonicity_formula}, we get
\begin{eqnarray*}
  c(n)\,r^n&\le&\s(x)\,\om_n\,\Big(\frac{r}5\Big)^n
  \\
  &\le& \Big(\H^n(S\cap B_{r/5}(x)) +\H^n(S\cap \tilde B_{r/5}(x))\Big) \,e^{C\,r/5}
  \\
  &\le& 2\,\H^n(S\cap B_r(x)) \,e^{C\,r_0}\le 4\,\H^n(S\cap B_r(x))
\end{eqnarray*}
up to further decreasing $r_0$.
\end{proof}

\section{A classical variational argument}\label{memme} Let $(K,E)$ be a generalized minimizer of $\psi(\e)$. In Theorem \ref{thm basic regularity}, we have proved that if $f:\Om\to\Om$ is a diffeomorphism such that $|f(E)|=|E|$, then
 \begin{equation}
   \label{basic tesi app}
   \F(K,E)\le\F(f(K),f(E))\,.
 \end{equation}
Here we show how to deduce from \eqref{basic tesi app} the existence of $\l\in\R$ such that
  \begin{equation}
    \label{basic stationary main app}
    \l\,\int_{\pa^*E}X\cdot\nu_E\,d\H^n=\int_{\pa^*E}\Div^K\,X\,d\H^n+2\,\int_{K\setminus\pa^*E}\Div^K\,X\,d\H^n\,,
  \end{equation}
  for every $X\in C^1_c(\R^{n+1};\R^{n+1})$ with $X\cdot\nu_\Om=0$ on $\pa\Om$. This is proved following a classical argument, see e.g. \cite[Theorem 17.20]{maggiBOOK}. We first treat the case when we also have
  \begin{equation}
    \label{corso cv per m}
      \int_{\pa^*E}X\cdot\nu_E\, d\H^n=0\,.
  \end{equation}
  In this case, let $Y\in C^1_c(\Om;\R^{n+1})$ be such that
  \[
  \int_{\pa^*E}Y\cdot\nu_E\,d\H^n=1\,,
  \]
  and set
  \[
  f_{t,s}(x)=x+t\,X(x)+s\,Y(x)\,,\qquad x\in\Omega\,.
  \]
  Given that $X\cdot\nu_\Om=0$ on $\pa\Om$ and that $\pa\Om$ is smooth, it is easily seen that for $t$ and $s$ sufficiently small, $f_{t,s}$ is a diffeomorphism from $\Omega$ to $\Omega$. In particular, the map
  \[
  \vphi(t,s)=|f_{t,s}(E)|
  \]
  is such that $\vphi(0,0)=|E|$, $(\pa\vphi/\pa t)(0,0)=0$ by \eqref{corso cv per m} and $(\pa\vphi/\pa s)(0,0)=1$ by the assumption on $Y$, so that, by the implicit function theorem we have $\vphi(t,s(t))=|E|$ for every $t$ sufficiently small and for $s(t)={\rm O}(t^2)$. Setting $g_t=f_{t,s(t)}$, by \eqref{basic tesi app}, we find that
  \begin{eqnarray*}
  m(t)=2\,\H^n(g_t(K)\setminus \pa^* g_t(E))+\H^n(\Om\cap \pa^*g_t(E))
  \end{eqnarray*}
  has a minimum at $t=0$. By Lemma \ref{lemma redb}, we can write
  \[
  m(t)=2\,\H^n(g_t(K\setminus\pa^*E))+\H^n(g_t(\Om\cap \pa^* E))\,.
  \]
  By the area formula, and since $s(t)={\rm O}(t^2)$ gives $(\pa g_t/\pa t)|_{t=0}=X$, we deduce the validity of \eqref{basic stationary main app} when \eqref{corso cv per m} holds. Let us now consider two fields $X_k\in C^1_c(\R^{n+1};\R^{n+1})$, $k=1,2$, with $X_k\cdot\nu_\Om=0$ on $\pa\Om$ and set
  \[
  X=X_1-\frac{\int_{\pa^*E}X_1\cdot\nu_E\,d\H^n}{\int_{\pa^*E}X_2\cdot\nu_E\,d\H^n}\,X_2\,.
  \]
  In this way $X$ satisfies \eqref{corso cv per m}, and thus \eqref{basic stationary main app}; as a consequence the quantity
  \[
\frac{\int_{\pa^*E}\Div^K\,X_k\,d\H^n+2\,\int_{K\setminus\pa^*E}\Div^K\,X_k\,d\H^n}{\int_{\pa^*E}X_k\cdot\nu_E\,d\H^n}
  \]
  has the same value for $k=1,2$.

\bibliographystyle{alpha}
\bibliography{references_mod}
\end{document}

%% file: span.pstex_t
\begin{picture}(0,0)%
\includegraphics{span.eps}%
\end{picture}%
\setlength{\unitlength}{3947sp}%
\begingroup\makeatletter\ifx\SetFigFont\undefined%
\gdef\SetFigFont#1#2#3#4#5{%
  \reset@font\fontsize{#1}{#2pt}%
  \fontfamily{#3}\fontseries{#4}\fontshape{#5}%
  \selectfont}%
\fi\endgroup%
\begin{picture}(5189,948)(225,-426)
\put(2030, 19){\makebox(0,0)[lb]{\smash{{\SetFigFont{11}{13.2}{\rmdefault}{\mddefault}{\updefault}{\color[rgb]{0,0,0}$E$}%
}}}}
\put(3410,-327){\makebox(0,0)[lb]{\smash{{\SetFigFont{10}{12.0}{\rmdefault}{\mddefault}{\updefault}{\color[rgb]{0,0,0}$\g$}%
}}}}
\put(4557, 79){\makebox(0,0)[lb]{\smash{{\SetFigFont{11}{13.2}{\rmdefault}{\mddefault}{\updefault}{\color[rgb]{0,0,0}$E$}%
}}}}
\put(240,367){\makebox(0,0)[lb]{\smash{{\SetFigFont{11}{13.2}{\rmdefault}{\mddefault}{\updefault}{\color[rgb]{0,0,0}$(a)$}%
}}}}
\put(3314,362){\makebox(0,0)[lb]{\smash{{\SetFigFont{11}{13.2}{\rmdefault}{\mddefault}{\updefault}{\color[rgb]{0,0,0}$(b)$}%
}}}}
\put(777,-158){\makebox(0,0)[lb]{\smash{{\SetFigFont{10}{12.0}{\rmdefault}{\mddefault}{\updefault}{\color[rgb]{0,0,0}$I_\de(\Gamma)$}%
}}}}
\end{picture}%

%% file: hpc.pstex_t
\begin{picture}(0,0)%
\includegraphics{hpc.eps}%
\end{picture}%
\setlength{\unitlength}{3947sp}%
\begingroup\makeatletter\ifx\SetFigFont\undefined%
\gdef\SetFigFont#1#2#3#4#5{%
  \reset@font\fontsize{#1}{#2pt}%
  \fontfamily{#3}\fontseries{#4}\fontshape{#5}%
  \selectfont}%
\fi\endgroup%
\begin{picture}(5251,1527)(296,-971)
\put(2088,212){\makebox(0,0)[lb]{\smash{{\SetFigFont{10}{12.0}{\rmdefault}{\mddefault}{\updefault}{\color[rgb]{0,0,0}$\g_1$}%
}}}}
\put(2083,-414){\makebox(0,0)[lb]{\smash{{\SetFigFont{10}{12.0}{\rmdefault}{\mddefault}{\updefault}{\color[rgb]{0,0,0}$\g_2$}%
}}}}
\put(4130,410){\makebox(0,0)[lb]{\smash{{\SetFigFont{10}{12.0}{\rmdefault}{\mddefault}{\updefault}{\color[rgb]{0,0,0}$\g_3$}%
}}}}
\put(3074,372){\makebox(0,0)[lb]{\smash{{\SetFigFont{11}{13.2}{\rmdefault}{\mddefault}{\updefault}{\color[rgb]{0,0,0}$(b)$}%
}}}}
\put(311,367){\makebox(0,0)[lb]{\smash{{\SetFigFont{11}{13.2}{\rmdefault}{\mddefault}{\updefault}{\color[rgb]{0,0,0}$(a)$}%
}}}}
\end{picture}%

%% file: example23.pstex_t
\begin{picture}(0,0)%
\includegraphics{example23.eps}%
\end{picture}%
\setlength{\unitlength}{3947sp}%
\begingroup\makeatletter\ifx\SetFigFont\undefined%
\gdef\SetFigFont#1#2#3#4#5{%
  \reset@font\fontsize{#1}{#2pt}%
  \fontfamily{#3}\fontseries{#4}\fontshape{#5}%
  \selectfont}%
\fi\endgroup%
\begin{picture}(5449,1667)(470,-1384)
\put(5162,104){\makebox(0,0)[lb]{\smash{{\SetFigFont{11}{13.2}{\rmdefault}{\mddefault}{\updefault}{\color[rgb]{0,0,0}$I_\de(\Gamma)$}%
}}}}
\put(3786,122){\makebox(0,0)[lb]{\smash{{\SetFigFont{11}{13.2}{\rmdefault}{\mddefault}{\updefault}{\color[rgb]{0,0,0}$(b)$}%
}}}}
\put(1497,-363){\makebox(0,0)[lb]{\smash{{\SetFigFont{11}{13.2}{\rmdefault}{\mddefault}{\updefault}{\color[rgb]{0,0,0}$E$}%
}}}}
\put(5068,-281){\makebox(0,0)[lb]{\smash{{\SetFigFont{11}{13.2}{\rmdefault}{\mddefault}{\updefault}{\color[rgb]{0,0,0}$K\setminus\pa E$}%
}}}}
\put(642,-905){\makebox(0,0)[lb]{\smash{{\SetFigFont{11}{13.2}{\rmdefault}{\mddefault}{\updefault}{\color[rgb]{0,0,0}$I_\de(\Gamma)$}%
}}}}
\put(485,128){\makebox(0,0)[lb]{\smash{{\SetFigFont{11}{13.2}{\rmdefault}{\mddefault}{\updefault}{\color[rgb]{0,0,0}$(a)$}%
}}}}
\put(4501,-686){\makebox(0,0)[lb]{\smash{{\SetFigFont{11}{13.2}{\rmdefault}{\mddefault}{\updefault}{\color[rgb]{0,0,0}$E$}%
}}}}
\end{picture}%

%% file: example.pstex_t
\begin{picture}(0,0)%
\includegraphics{example.eps}%
\end{picture}%
\setlength{\unitlength}{3947sp}%
\begingroup\makeatletter\ifx\SetFigFont\undefined%
\gdef\SetFigFont#1#2#3#4#5{%
  \reset@font\fontsize{#1}{#2pt}%
  \fontfamily{#3}\fontseries{#4}\fontshape{#5}%
  \selectfont}%
\fi\endgroup%
\begin{picture}(5318,5489)(936,-5071)
\put(5058,-1476){\makebox(0,0)[lb]{\smash{{\SetFigFont{9}{10.8}{\rmdefault}{\mddefault}{\updefault}{\color[rgb]{0,0,0}$d=\sqrt{3}-1/\sqrt{3}$}%
}}}}
\put(2216,-1465){\makebox(0,0)[lb]{\smash{{\SetFigFont{9}{10.8}{\rmdefault}{\mddefault}{\updefault}{\color[rgb]{0,0,0}$\ell/2=\sqrt{3}$}%
}}}}
\put(1157,-426){\makebox(0,0)[lb]{\smash{{\SetFigFont{9}{10.8}{\rmdefault}{\mddefault}{\updefault}{\color[rgb]{0,0,0}$1$}%
}}}}
\put(4696,281){\makebox(0,0)[lb]{\smash{{\SetFigFont{9}{10.8}{\rmdefault}{\mddefault}{\updefault}{\color[rgb]{0,0,0}$t=1/2\sqrt{3}$}%
}}}}
\put(4013,-34){\makebox(0,0)[lb]{\smash{{\SetFigFont{10}{12.0}{\rmdefault}{\mddefault}{\updefault}{\color[rgb]{0,0,0}$(b)$}%
}}}}
\put(951,-43){\makebox(0,0)[lb]{\smash{{\SetFigFont{10}{12.0}{\rmdefault}{\mddefault}{\updefault}{\color[rgb]{0,0,0}$(a)$}%
}}}}
\put(951,-1883){\makebox(0,0)[lb]{\smash{{\SetFigFont{10}{12.0}{\rmdefault}{\mddefault}{\updefault}{\color[rgb]{0,0,0}$(c)$}%
}}}}
\put(951,-3712){\makebox(0,0)[lb]{\smash{{\SetFigFont{10}{12.0}{\rmdefault}{\mddefault}{\updefault}{\color[rgb]{0,0,0}$(d)$}%
}}}}
\end{picture}%

%% file: dlgm.pstex_t
\begin{picture}(0,0)%
\includegraphics{dlgm.eps}%
\end{picture}%
\setlength{\unitlength}{3947sp}%
\begingroup\makeatletter\ifx\SetFigFont\undefined%
\gdef\SetFigFont#1#2#3#4#5{%
  \reset@font\fontsize{#1}{#2pt}%
  \fontfamily{#3}\fontseries{#4}\fontshape{#5}%
  \selectfont}%
\fi\endgroup%
\begin{picture}(6125,1533)(237,-1441)
\put(1804,-133){\makebox(0,0)[lb]{\smash{{\SetFigFont{11}{13.2}{\rmdefault}{\mddefault}{\updefault}{\color[rgb]{0,0,0}$S$}%
}}}}
\put(1741,-1374){\makebox(0,0)[lb]{\smash{{\SetFigFont{11}{13.2}{\rmdefault}{\mddefault}{\updefault}{\color[rgb]{0,0,0}$B_r(x)$}%
}}}}
\put(2362,-54){\makebox(0,0)[lb]{\smash{{\SetFigFont{10}{12.0}{\rmdefault}{\mddefault}{\updefault}{\color[rgb]{0,0,0}$(a)$}%
}}}}
\put(4503,-73){\makebox(0,0)[lb]{\smash{{\SetFigFont{10}{12.0}{\rmdefault}{\mddefault}{\updefault}{\color[rgb]{0,0,0}$(b)$}%
}}}}
\put(867,-1284){\makebox(0,0)[lb]{\smash{{\SetFigFont{11}{13.2}{\rmdefault}{\mddefault}{\updefault}{\color[rgb]{0,0,0}$A$}%
}}}}
\end{picture}%

%% file: cup00.pstex_t
\begin{picture}(0,0)%
\includegraphics{cup00.eps}%
\end{picture}%
\setlength{\unitlength}{3947sp}%
\begingroup\makeatletter\ifx\SetFigFont\undefined%
\gdef\SetFigFont#1#2#3#4#5{%
  \reset@font\fontsize{#1}{#2pt}%
  \fontfamily{#3}\fontseries{#4}\fontshape{#5}%
  \selectfont}%
\fi\endgroup%
\begin{picture}(4499,4737)(218,-4127)
\put(719,-1275){\makebox(0,0)[lb]{\smash{{\SetFigFont{11}{13.2}{\rmdefault}{\mddefault}{\updefault}{\color[rgb]{0,0,0}$E$}%
}}}}
\put(240,-3211){\makebox(0,0)[lb]{\smash{{\SetFigFont{10}{12.0}{\rmdefault}{\mddefault}{\updefault}{\color[rgb]{0,0,0}$(c)$}%
}}}}
\put(1403,-3973){\makebox(0,0)[lb]{\smash{{\SetFigFont{11}{13.2}{\rmdefault}{\mddefault}{\updefault}{\color[rgb]{0,0,0}$A$}%
}}}}
\put(715,-2915){\makebox(0,0)[lb]{\smash{{\SetFigFont{11}{13.2}{\rmdefault}{\mddefault}{\updefault}{\color[rgb]{0,0,0}$E$}%
}}}}
\put(1407,-2304){\makebox(0,0)[lb]{\smash{{\SetFigFont{11}{13.2}{\rmdefault}{\mddefault}{\updefault}{\color[rgb]{0,0,0}$A$}%
}}}}
\put(2027,-1228){\makebox(0,0)[lb]{\smash{{\SetFigFont{11}{13.2}{\rmdefault}{\mddefault}{\updefault}{\color[rgb]{0,0,0}$B_r(x)$}%
}}}}
\put(233,441){\makebox(0,0)[lb]{\smash{{\SetFigFont{10}{12.0}{\rmdefault}{\mddefault}{\updefault}{\color[rgb]{0,0,0}$(a)$}%
}}}}
\put(240,-1242){\makebox(0,0)[lb]{\smash{{\SetFigFont{10}{12.0}{\rmdefault}{\mddefault}{\updefault}{\color[rgb]{0,0,0}$(b)$}%
}}}}
\put(2020,455){\makebox(0,0)[lb]{\smash{{\SetFigFont{11}{13.2}{\rmdefault}{\mddefault}{\updefault}{\color[rgb]{0,0,0}$B_r(x)$}%
}}}}
\put(1400,-621){\makebox(0,0)[lb]{\smash{{\SetFigFont{11}{13.2}{\rmdefault}{\mddefault}{\updefault}{\color[rgb]{0,0,0}$A$}%
}}}}
\put(712,408){\makebox(0,0)[lb]{\smash{{\SetFigFont{11}{13.2}{\rmdefault}{\mddefault}{\updefault}{\color[rgb]{0,0,0}$E$}%
}}}}
\put(521,-2445){\makebox(0,0)[lb]{\smash{{\SetFigFont{11}{13.2}{\rmdefault}{\mddefault}{\updefault}{\color[rgb]{0,0,0}$S$}%
}}}}
\put(2657,-2458){\makebox(0,0)[lb]{\smash{{\SetFigFont{11}{13.2}{\rmdefault}{\mddefault}{\updefault}{\color[rgb]{0,0,0}$U_\eta(S)$}%
}}}}
\put(2982,-2915){\makebox(0,0)[lb]{\smash{{\SetFigFont{11}{13.2}{\rmdefault}{\mddefault}{\updefault}{\color[rgb]{0,0,0}$F$}%
}}}}
\put(4038,-2835){\makebox(0,0)[lb]{\smash{{\SetFigFont{11}{13.2}{\rmdefault}{\mddefault}{\updefault}{\color[rgb]{0,0,0}$Y$}%
}}}}
\put(4293,400){\makebox(0,0)[lb]{\smash{{\SetFigFont{11}{13.2}{\rmdefault}{\mddefault}{\updefault}{\color[rgb]{0,0,0}$Y$}%
}}}}
\put(2981,-1275){\makebox(0,0)[lb]{\smash{{\SetFigFont{11}{13.2}{\rmdefault}{\mddefault}{\updefault}{\color[rgb]{0,0,0}$F$}%
}}}}
\put(2974,408){\makebox(0,0)[lb]{\smash{{\SetFigFont{11}{13.2}{\rmdefault}{\mddefault}{\updefault}{\color[rgb]{0,0,0}$F$}%
}}}}
\put(4300,-1283){\makebox(0,0)[lb]{\smash{{\SetFigFont{11}{13.2}{\rmdefault}{\mddefault}{\updefault}{\color[rgb]{0,0,0}$Y$}%
}}}}
\put(4433,-682){\makebox(0,0)[lb]{\smash{{\SetFigFont{8}{9.6}{\rmdefault}{\mddefault}{\updefault}{\color[rgb]{0,0,0}$\eta$}%
}}}}
\end{picture}%

%% file: cupext.pstex_t
\begin{picture}(0,0)%
\includegraphics{cupext.eps}%
\end{picture}%
\setlength{\unitlength}{3947sp}%
\begingroup\makeatletter\ifx\SetFigFont\undefined%
\gdef\SetFigFont#1#2#3#4#5{%
  \reset@font\fontsize{#1}{#2pt}%
  \fontfamily{#3}\fontseries{#4}\fontshape{#5}%
  \selectfont}%
\fi\endgroup%
\begin{picture}(4692,2210)(582,-1861)
\put(4741,-1711){\makebox(0,0)[lb]{\smash{{\SetFigFont{11}{13.2}{\rmdefault}{\mddefault}{\updefault}{\color[rgb]{0,0,0}$M_\eta(U_\eta(S))$}%
}}}}
\put(2196,-125){\makebox(0,0)[lb]{\smash{{\SetFigFont{11}{13.2}{\rmdefault}{\mddefault}{\updefault}{\color[rgb]{0,0,0}$B_R(0)$}%
}}}}
\put(1057,-1418){\makebox(0,0)[lb]{\smash{{\SetFigFont{11}{13.2}{\rmdefault}{\mddefault}{\updefault}{\color[rgb]{0,0,0}$A$}%
}}}}
\put(4726,-511){\makebox(0,0)[lb]{\smash{{\SetFigFont{11}{13.2}{\rmdefault}{\mddefault}{\updefault}{\color[rgb]{0,0,0}$F$}%
}}}}
\put(1749,-499){\makebox(0,0)[lb]{\smash{{\SetFigFont{11}{13.2}{\rmdefault}{\mddefault}{\updefault}{\color[rgb]{0,0,0}$E$}%
}}}}
\put(770,-463){\makebox(0,0)[lb]{\smash{{\SetFigFont{11}{13.2}{\rmdefault}{\mddefault}{\updefault}{\color[rgb]{0,0,0}$W$}%
}}}}
\put(1956,-1794){\makebox(0,0)[lb]{\smash{{\SetFigFont{11}{13.2}{\rmdefault}{\mddefault}{\updefault}{\color[rgb]{0,0,0}$S$}%
}}}}
\put(4827,194){\makebox(0,0)[lb]{\smash{{\SetFigFont{11}{13.2}{\rmdefault}{\mddefault}{\updefault}{\color[rgb]{0,0,0}$M_\eta(Y)$}%
}}}}
\end{picture}%

%% file: cone.pstex_t
\begin{picture}(0,0)%
\includegraphics{cone.eps}%
\end{picture}%
\setlength{\unitlength}{3947sp}%
\begingroup\makeatletter\ifx\SetFigFont\undefined%
\gdef\SetFigFont#1#2#3#4#5{%
  \reset@font\fontsize{#1}{#2pt}%
  \fontfamily{#3}\fontseries{#4}\fontshape{#5}%
  \selectfont}%
\fi\endgroup%
\begin{picture}(5012,1598)(698,-1022)
\put(713,371){\makebox(0,0)[lb]{\smash{{\SetFigFont{10}{12.0}{\rmdefault}{\mddefault}{\updefault}{\color[rgb]{0,0,0}$E$}%
}}}}
\put(2349,430){\makebox(0,0)[lb]{\smash{{\SetFigFont{10}{12.0}{\rmdefault}{\mddefault}{\updefault}{\color[rgb]{0,0,0}$B$}%
}}}}
\put(3542,371){\makebox(0,0)[lb]{\smash{{\SetFigFont{10}{12.0}{\rmdefault}{\mddefault}{\updefault}{\color[rgb]{0,0,0}$F$}%
}}}}
\put(1836,-366){\makebox(0,0)[lb]{\smash{{\SetFigFont{9}{10.8}{\rmdefault}{\mddefault}{\updefault}{\color[rgb]{0,0,0}$x$}%
}}}}
\end{picture}%

%% file: step8geometry.pstex_t
\begin{picture}(0,0)%
\includegraphics{step8geometry.eps}%
\end{picture}%
\setlength{\unitlength}{3947sp}%
\begingroup\makeatletter\ifx\SetFigFont\undefined%
\gdef\SetFigFont#1#2#3#4#5{%
  \reset@font\fontsize{#1}{#2pt}%
  \fontfamily{#3}\fontseries{#4}\fontshape{#5}%
  \selectfont}%
\fi\endgroup%
\begin{picture}(4139,2358)(478,-1765)
\put(2456,372){\makebox(0,0)[lb]{\smash{{\SetFigFont{11}{13.2}{\rmdefault}{\mddefault}{\updefault}{\color[rgb]{0,0,0}$B_r(x)$}%
}}}}
\put(1388,-921){\makebox(0,0)[lb]{\smash{{\SetFigFont{11}{13.2}{\rmdefault}{\mddefault}{\updefault}{\color[rgb]{0,0,0}$W_{\tau,\s,r}^-$}%
}}}}
\put(3609,-531){\makebox(0,0)[lb]{\smash{{\SetFigFont{10}{12.0}{\rmdefault}{\mddefault}{\updefault}{\color[rgb]{0,0,0}$\tau r$}%
}}}}
\put(3140,-530){\makebox(0,0)[lb]{\smash{{\SetFigFont{10}{12.0}{\rmdefault}{\mddefault}{\updefault}{\color[rgb]{0,0,0}$\s r$}%
}}}}
\put(1396,-296){\makebox(0,0)[lb]{\smash{{\SetFigFont{11}{13.2}{\rmdefault}{\mddefault}{\updefault}{\color[rgb]{0,0,0}$W_{\tau,\s,r}^+$}%
}}}}
\put(2243,-631){\makebox(0,0)[lb]{\smash{{\SetFigFont{11}{13.2}{\rmdefault}{\mddefault}{\updefault}{\color[rgb]{0,0,0}$V_{\s,r}$}%
}}}}
\put(563,-631){\makebox(0,0)[lb]{\smash{{\SetFigFont{11}{13.2}{\rmdefault}{\mddefault}{\updefault}{\color[rgb]{0,0,0}$V_{\s,r}$}%
}}}}
\put(4602,129){\makebox(0,0)[lb]{\smash{{\SetFigFont{11}{13.2}{\rmdefault}{\mddefault}{\updefault}{\color[rgb]{0,0,0}$\nu(x)$}%
}}}}
\put(1594,-515){\makebox(0,0)[lb]{\smash{{\SetFigFont{11}{13.2}{\rmdefault}{\mddefault}{\updefault}{\color[rgb]{0,0,0}$x$}%
}}}}
\put(945,-21){\makebox(0,0)[lb]{\smash{{\SetFigFont{11}{13.2}{\rmdefault}{\mddefault}{\updefault}{\color[rgb]{0,0,0}$\Gamma_{\tau,\s,r}^+$}%
}}}}
\put(949,-1248){\makebox(0,0)[lb]{\smash{{\SetFigFont{11}{13.2}{\rmdefault}{\mddefault}{\updefault}{\color[rgb]{0,0,0}$\Gamma_{\tau,\s,r}^-$}%
}}}}
\end{picture}%

%% file: slab1.pstex_t
\begin{picture}(0,0)%
\includegraphics{slab1.eps}%
\end{picture}%
\setlength{\unitlength}{3947sp}%
\begingroup\makeatletter\ifx\SetFigFont\undefined%
\gdef\SetFigFont#1#2#3#4#5{%
  \reset@font\fontsize{#1}{#2pt}%
  \fontfamily{#3}\fontseries{#4}\fontshape{#5}%
  \selectfont}%
\fi\endgroup%
\begin{picture}(6227,2072)(336,-1532)
\put(1948,-539){\makebox(0,0)[lb]{\smash{{\SetFigFont{9}{10.8}{\rmdefault}{\mddefault}{\updefault}{\color[rgb]{0,0,0}$x$}%
}}}}
\put(1524,391){\makebox(0,0)[lb]{\smash{{\SetFigFont{10}{12.0}{\rmdefault}{\mddefault}{\updefault}{\color[rgb]{0,0,0}$B_r(x)$}%
}}}}
\put(3891,-800){\makebox(0,0)[lb]{\smash{{\SetFigFont{10}{12.0}{\rmdefault}{\mddefault}{\updefault}{\color[rgb]{0,0,0}$F_j$}%
}}}}
\put(351,-805){\makebox(0,0)[lb]{\smash{{\SetFigFont{10}{12.0}{\rmdefault}{\mddefault}{\updefault}{\color[rgb]{0,0,0}$E_j$}%
}}}}
\put(3191,-588){\makebox(0,0)[lb]{\smash{{\SetFigFont{9}{10.8}{\rmdefault}{\mddefault}{\updefault}{\color[rgb]{0,0,0}$\tau\,r$}%
}}}}
\put(1439,-246){\makebox(0,0)[lb]{\smash{{\SetFigFont{10}{12.0}{\rmdefault}{\mddefault}{\updefault}{\color[rgb]{0,0,0}$A_{r,j}^{\rm out}$}%
}}}}
\end{picture}%

%% file: slab2.pstex_t
\begin{picture}(0,0)%
\includegraphics{slab2.eps}%
\end{picture}%
\setlength{\unitlength}{3947sp}%
\begingroup\makeatletter\ifx\SetFigFont\undefined%
\gdef\SetFigFont#1#2#3#4#5{%
  \reset@font\fontsize{#1}{#2pt}%
  \fontfamily{#3}\fontseries{#4}\fontshape{#5}%
  \selectfont}%
\fi\endgroup%
\begin{picture}(6243,2084)(336,-1547)
\put(351,-805){\makebox(0,0)[lb]{\smash{{\SetFigFont{10}{12.0}{\rmdefault}{\mddefault}{\updefault}{\color[rgb]{0,0,0}$E_j$}%
}}}}
\put(3891,-800){\makebox(0,0)[lb]{\smash{{\SetFigFont{10}{12.0}{\rmdefault}{\mddefault}{\updefault}{\color[rgb]{0,0,0}$F_j$}%
}}}}
\put(3191,-588){\makebox(0,0)[lb]{\smash{{\SetFigFont{9}{10.8}{\rmdefault}{\mddefault}{\updefault}{\color[rgb]{0,0,0}$\tau\,r$}%
}}}}
\put(1524,391){\makebox(0,0)[lb]{\smash{{\SetFigFont{10}{12.0}{\rmdefault}{\mddefault}{\updefault}{\color[rgb]{0,0,0}$B_r(x)$}%
}}}}
\put(1948,-539){\makebox(0,0)[lb]{\smash{{\SetFigFont{9}{10.8}{\rmdefault}{\mddefault}{\updefault}{\color[rgb]{0,0,0}$x$}%
}}}}
\put(2603,-1484){\makebox(0,0)[lb]{\smash{{\SetFigFont{10}{12.0}{\rmdefault}{\mddefault}{\updefault}{\color[rgb]{0,0,0}$A_{r,j}^{\rm in}$}%
}}}}
\end{picture}%

%% file: slab3.pstex_t
\begin{picture}(0,0)%
\includegraphics{slab3.eps}%
\end{picture}%
\setlength{\unitlength}{3947sp}%
\begingroup\makeatletter\ifx\SetFigFont\undefined%
\gdef\SetFigFont#1#2#3#4#5{%
  \reset@font\fontsize{#1}{#2pt}%
  \fontfamily{#3}\fontseries{#4}\fontshape{#5}%
  \selectfont}%
\fi\endgroup%
\begin{picture}(6265,2084)(336,-1547)
\put(1948,-539){\makebox(0,0)[lb]{\smash{{\SetFigFont{9}{10.8}{\rmdefault}{\mddefault}{\updefault}{\color[rgb]{0,0,0}$x$}%
}}}}
\put(3191,-588){\makebox(0,0)[lb]{\smash{{\SetFigFont{9}{10.8}{\rmdefault}{\mddefault}{\updefault}{\color[rgb]{0,0,0}$\tau\,r$}%
}}}}
\put(351,-805){\makebox(0,0)[lb]{\smash{{\SetFigFont{10}{12.0}{\rmdefault}{\mddefault}{\updefault}{\color[rgb]{0,0,0}$E_j$}%
}}}}
\put(1524,391){\makebox(0,0)[lb]{\smash{{\SetFigFont{10}{12.0}{\rmdefault}{\mddefault}{\updefault}{\color[rgb]{0,0,0}$B_r(x)$}%
}}}}
\put(3927,-803){\makebox(0,0)[lb]{\smash{{\SetFigFont{10}{12.0}{\rmdefault}{\mddefault}{\updefault}{\color[rgb]{0,0,0}$F_j$}%
}}}}
\put(2603,-1484){\makebox(0,0)[lb]{\smash{{\SetFigFont{10}{12.0}{\rmdefault}{\mddefault}{\updefault}{\color[rgb]{0,0,0}$A_{r,j}^{\rm out}$}%
}}}}
\end{picture}%

%% file: push.pstex_t
\begin{picture}(0,0)%
\includegraphics{push.eps}%
\end{picture}%
\setlength{\unitlength}{3947sp}%
\begingroup\makeatletter\ifx\SetFigFont\undefined%
\gdef\SetFigFont#1#2#3#4#5{%
  \reset@font\fontsize{#1}{#2pt}%
  \fontfamily{#3}\fontseries{#4}\fontshape{#5}%
  \selectfont}%
\fi\endgroup%
\begin{picture}(6281,2138)(439,-1532)
\put(2704,-1290){\makebox(0,0)[lb]{\smash{{\SetFigFont{10}{12.0}{\rmdefault}{\mddefault}{\updefault}{\color[rgb]{0,0,0}$B_{r_0}(x)$}%
}}}}
\put(454,-81){\makebox(0,0)[lb]{\smash{{\SetFigFont{9}{10.8}{\rmdefault}{\mddefault}{\updefault}{\color[rgb]{0,0,0}$\eta$}%
}}}}
\put(4123,379){\makebox(0,0)[lb]{\smash{{\SetFigFont{10}{12.0}{\rmdefault}{\mddefault}{\updefault}{\color[rgb]{0,0,0}$\Om'$}%
}}}}
\put(649,392){\makebox(0,0)[lb]{\smash{{\SetFigFont{10}{12.0}{\rmdefault}{\mddefault}{\updefault}{\color[rgb]{0,0,0}$\Om$}%
}}}}
\end{picture}%